\documentclass[11pt]{amsart}
\usepackage{amsmath}
\usepackage{amssymb}
\usepackage{amsthm}
\usepackage[all]{xy}
\usepackage{graphicx}
\usepackage[margin=1in]{geometry}
\usepackage{color}
\usepackage{hyperref}
\usepackage{verbatim}
\usepackage{accents}
\usepackage{pb-diagram,pb-xy}

\numberwithin{equation}{section}

\newcommand{\R}{\ensuremath{\mathbb{R}}}
\newcommand{\N}{\ensuremath{\mathbb{N}}}

\newcommand{\Z}{\ensuremath{\mathbb{Z}}}
\newcommand{\Q}{\ensuremath{\mathbb{Q}}}

\newtheorem{theorem}{Theorem}[section]
\newtheorem{lemma}[theorem]{Lemma}
\newtheorem{proposition}[theorem]{Proposition}

\newtheorem{corollary}[theorem]{Corollary}

\newtheorem{notation}{Notation}[section]

\theoremstyle{definition} % makes body in upright roman (not italics)
\newtheorem{defn}{Definition}[section]

\newtheorem{remark}{Remark}[section] 
\newtheorem{Notation}{Notational Convention}[section]
\newtheorem{Construction}{Construction}[section]

\newcommand{\mb}[1]{\mathbf{{#1}}}

\DeclareMathOperator*{\Int}{Int}

\DeclareMathOperator*{\Diff}{Diff}
\DeclareMathOperator*{\Aut}{Aut}
\DeclareMathOperator*{\Emb}{Emb}
\DeclareMathOperator*{\BDiff}{BDiff}

\DeclareMathOperator*{\Hom}{Hom}
\DeclareMathOperator*{\Maps}{Maps}

\DeclareMathOperator*{\Imm}{Imm}

\DeclareMathOperator*{\Ker}{Ker}

\DeclareMathOperator*{\Ab}{Ab}

\DeclareMathOperator*{\rel}{rel}
\DeclareMathOperator*{\lk}{lk}
\DeclareMathOperator*{\Id}{Id}
\DeclareMathOperator*{\Ob}{Ob}

\author{Nathan Perlmutter}

\address{Stanford University Department of Mathematics, Building 380, Stanford, California,  94305, USA} 
  
  \email{nperlmut@stanford.edu}

\title[Linking Forms and Stabilization]{Linking Forms and Stabilization of Diffeomorphism Groups of Manifolds of Dimension $4n+1$}

\begin{document}
\maketitle
\begin{abstract}
Let $n \geq 2$.  We prove a homological stability theorem for the
diffeomorphism groups of $(4n+1)$-dimensional manifolds, with respect
to forming the connected sum with certain $(2n-1)$-connected,
$(4n+1)$-dimensional manifolds that are stably parallelizable.  
Our techniques involve the study of the action of the diffeomorphism group of a manifold $M$ on the \textit{linking form} associated to the homology groups of $M$. 
In particular, we construct a geometric model for the linking form using the intersections of embedded and immersed $\Z/k$-manifolds. 
In addition to our main homological stability theorem, we prove several disjunction results for the embeddings and immersions of $\Z/k$-manifolds that could be of independent interest. 
\end{abstract}

\setcounter{tocdepth}{1}
\tableofcontents

\section{Introduction} \label{section: introduction} 
\subsection{Main result}
Let $M$ be a smooth, compact, connected manifold with non-empty boundary and let $\dim(M) = m$. 
We denote by $\Diff^{\partial}(M)$ the
group of self diffeomorphisms of $M$ which fix some neighborhood of
the boundary pointwise, topologized in the $C^{\infty}$-topology.  
Let $\BDiff^{\partial}(M)$ denote the \textit{classifying space} of
$\Diff^{\partial}(M)$. 
Choose a closed, connected manifold $W$ with $\dim(W)= m$. 
There is a natural \textit{stabilization} homomorphism $\Diff^{\partial}(M)\to \Diff^{\partial}(M\# W)$ which gives rise to
the direct system of maps of the classifying spaces:
\begin{equation} \label{equation: direct system}
\xymatrix{
\BDiff^{\partial}(M)\longrightarrow  
\BDiff^{\partial}(M\# W)\longrightarrow \; \cdots \; \longrightarrow \BDiff^{\partial}(M\# W^{\#g}) \longrightarrow \cdots
}
\end{equation}
In this paper we study the homological stability of this direct system in the
  case when $M$ and $W$ are odd-dimensional, highly connected
  manifolds.  
  The theorem below is the main result of the paper. 

\begin{theorem} \label{theorem: independence of rank}
For $n \geq 2$, 
let $M$ be a $2$-connected, $(4n+1)$-dimensional,
compact manifold with non-empty boundary.
Let $W$ be a closed, $(2n-1)$-connected, $(4n+1)$-dimensional manifold
that satisfies the following conditions:
\begin{itemize} \itemsep.2cm
\item $W$ is stably parallelizable;
\item the homology group $H_{2n}(W; \Z)$ is finite and has trivial $2$-torsion.
\end{itemize}
Then the homomorphism of homology groups induced by (\ref{equation: direct system}),
$$
H_{\ell}(\textstyle{\BDiff^{\partial}}(M\# W^{\# g}); \Z) \longrightarrow H_{\ell}(\textstyle{\BDiff^{\partial}}(M\# W^{\# g+1}); \Z),
$$
is an isomorphism if $g \geq 2\ell + 3$ and an epimorphism if $g \geq 2\ell + 1$.
\end{theorem} 
\begin{remark}
This result yields an odd-dimensional analogue of the theorem of
Galatius and Randal-Williams from \cite{GRW 12} and \cite{GRW 14},
regarding the homological stability of diffeomorphism groups of
manifolds of dimension $2n$ with respect to forming connected sums
with $S^{n}\times S^{n}$. 
The version of Theorem \ref{theorem: independence of rank} for $W = S^{2n}\times S^{2n+1}$ follows from the main result of 
\cite[Theorem 1.3]{NP 14}.
\end{remark}

\subsection{$(2n-1)$-connected, $(4n+1)$-dimensional manifolds}
Let $\mathcal{W}_{4n+1}$ denote the set of all stably parallelizable, 
 $(2n-1)$-connected, $(4n+1)$-dimensional, closed manifolds. 
 Let $\mathcal{W}^{f}_{4n+1} \subset \mathcal{W}_{4n+1}$ denote the subset consisting of those manifolds $W \in \mathcal{W}_{4n+1}$ such that $H_{2n}(W; \Z)$ is finite.
For $n \geq 2$, the classification of manifolds in $\mathcal{W}_{4n+1}$ was studied by Wall in \cite{W 67} and by De Sapio in \cite{S 69}.
 Recall that two closed manifolds $M_{1}$ and $M_{2}$ are said to be \textit{almost diffeomorphic} if there exists a homotopy sphere $\Sigma$ such that $M_{1}\#\Sigma$ is diffeomorphic to $M_{2}$. 
Let $N \in \mathcal{W}_{4n+1}$ denote the unit sphere bundle associated to the tangent bundle of $S^{2n+1}$.
 It follows from \cite[Theorem 1]{S 69} that any element $W \in \mathcal{W}_{4n+1}$ is almost diffeomorphic to a manifold of the form, 
 $
 (S^{2n}\times S^{2n+1})^{\# g_{1}}\; \# \; N^{\#g_{2}} \;\# \; M_{\tau},
 $
 where $M_{\tau}$ is some element of $\mathcal{W}^{f}_{4n+1}$.

Let $\widehat{\mathcal{W}}^{f}_{4n+1} \subset \mathcal{W}^{f}_{4n+1}$ denote the subset consisting of those $W \in \mathcal{W}^{f}_{4n+1}$ with $H_{2n}(W; \Z/2) = 0$. 
Our main theorem will require a description of the manifolds in $\widehat{\mathcal{W}}^{f}_{4n+1}$.
The primary diffeomorphism invariant associated to a manifold $M \in \widehat{\mathcal{W}}^{f}_{4n+1}$ is the \textit{linking form}, which is a strictly skew symmetric, non-singular, bilinear pairing
\begin{equation} \label{equation: homotopical linking form}
b:  \pi_{2n}(M)\otimes\pi_{2n}(M) \longrightarrow \Q/\Z.
\end{equation}
Here a strictly skew symmetric form is one for which $b(x, x) = 0$ for all $x$, which implies the property $b(x, y) = -b(y, x) = 0$.
It follows from Wall's classification \cite[Theorem 7]{W 67} that two elements $M_{1}, M_{2} \in \widehat{\mathcal{W}}^{f}_{4n+1}$ are almost diffeomorphic if and only if there exists an isomorphism of linking forms,
$$(\pi_{2n}(M_{1}), \; b) \cong (\pi_{2n}(M_{2}), \; b).$$
Furthermore, given any finite abelian group $G$ with $G\otimes\Z/2 = 0$, equipped with a non-singular, strictly skew symmetric bilinear form $b': G\otimes G \longrightarrow \Q/\Z$, there exists a manifold $M \in \widehat{\mathcal{W}}^{f}_{4n+1}$ such that there is an isomorphism 
$(\pi_{2n}(M), \; b) \cong (G, \; b').$
It follows that a manifold in $\widehat{\mathcal{W}}^{f}_{4n+1}$ is completely determined (up to almost diffeomorphism) by the isomorphism class of its linking form. 

\begin{remark} \label{remark: Wall's invariants}
For manifolds $M \in \mathcal{W}^{f}_{4n+1}$ with $H_{2n}(M; \Z/2) \neq 0$, the linking form is still defined and plays a key role in the diffeomorphism classification. 
However,  the linking form is not a complete invariant for such manifolds.
In \cite{W 67} Wall associates to any $(2n-1)$-connected, $(4n+1)$-dimensional manifold $W$, a cohomology class $\widehat{\phi} \in H^{2n+1}(W; \Z/2)$ which is a diffeomorphism invariant of $W$. 
It is possible for $\widehat{\phi} \neq 0$ even if $W$ is stably parallelizable.  
Thus if $H^{2n+1}(W; \Z/2) \neq 0$, then this invariant $\widehat{\phi}$ must be considered.
In this case, the almost diffeomorphism class of $W$ is not completely determined by its linking form.
If $W \in \widehat{\mathcal{W}}^{f}_{4n+1}$, then $H^{2n+1}(W; \Z/2) = 0$ and we can ignore the invariant $\widehat{\phi}$.
Since $W$ is stably parallelizable the other invariants defined by Wall, 
$$\alpha: \pi_{2n}(W) \longrightarrow \pi_{2n-1}(SO)) \quad \text{and} \quad \widehat{\beta} \in H^{2n+1}(W; \pi_{2n}(SO)),$$ 
automatically vanish as well. 
It follows that the almost diffeomorphism type of $W \in \widehat{\mathcal{W}}^{f}_{4n+1}$ is determined by its linking form and the classification of manifolds in $\widehat{\mathcal{W}}^{f}_{4n+1}$ reduces to the classification of strictly skew symmetric bilinear forms on finite abelian groups from \cite{W 64}. 
\end{remark}

We use the classification result discussed above to specify certain basic elements of $\widehat{\mathcal{W}}^{f}_{4n+1}$.
For each odd integer $k > 2$, fix a manifold $W_{k} \in \widehat{\mathcal{W}}^{f}_{4n+1}$ whose linking form $(\pi_{2n}(W_{k}), \; b)$ is given by the data,
$$\xymatrix@C-1.0pc@R-1.5pc{
\pi_{2n}(W_{k}) = \Z/k\oplus \Z/k,  & 
b(\sigma, \sigma) = b(\rho, \rho) = 0, &  b(\sigma, \rho) = -b(\rho, \sigma) = \frac{1}{k} \mod 1,
}
$$
where $\langle \rho, \sigma \rangle$ is the standard basis for $\Z/k\oplus\Z/k$. 
By the classification result discussed above, this data determines the manifold $W_{k}$ completely up to almost diffeomorphism.  
It follows from \cite[Theorem 7]{W 67} and the classification of skew symmetric forms over $\Q/\Z$ in \cite[Lemma 7]{W 64} that any element $M \in \widehat{\mathcal{W}}^{f}_{4n+1}$ is diffeomorphic to a manifold of the form,
$W_{k_{1}}\#\cdots \#W_{k_{l}}\#\Sigma,$
where $\Sigma$ is a homotopy sphere.

\begin{comment}
\begin{remark}
It follows from these classification results, \cite[Theorem 7]{W 67} and \cite[Lemma 7]{W 64},  that if $k$ and $\ell$ are relatively prime, then $W_{k}\# W_{\ell} \cong W_{k\cdot \ell}$. 
In this way, the (almost) diffeomorphism classification of $\widehat{\mathcal{W}}^{f}_{4n+1}$ mirrors the classification of finitely abelian groups with trivial $2$-torsion.
Thus it will suffice to restrict our attention to the manifolds $W_{k}$ in the case that $k = p^{j}$ for a prime number $p$. 
\end{remark}
\end{comment}

\subsection{The stabilization map}
Let $M$ be a $(4n+1)$-dimensional, compact, connected, manifold with non-empty boundary.
Fix a connected component $\partial_{0}M \subset \partial M$ of the boundary of $M$.
For each integer $k > 2$, let $\widetilde{W}_{k}$ denote the manifold obtained by forming the connected sum of $[0,1]\times \partial_{0}M$ with $W_{k}$. 
Denote by $M\cup_{\partial_{0}M}\widetilde{W}_{k}$ the manifold obtained by gluing $\widetilde{W}_{k}$ to $M$ along $\{0\}\times \partial_{0}M$. 
It is clear that there is a diffeomorphism $M\cup_{\partial_{0}M}\widetilde{W}_{k} \cong M\# W_{k}$. 
Consider the continuous homomorphism $\Diff^{\partial}(M) \longrightarrow \Diff^{\partial}(M\cup_{\partial_{0}M}\widetilde{W}_{k})$ defined by extending $f \in  \Diff^{\partial}(M)$ identically over $\widetilde{W}_{k}$.
For each $k$, this homomorphism induces a continuous map on the level of classifying spaces,
\begin{equation} \label{eq: k-stabilization map}
\xymatrix{
s_{k}: \BDiff^{\partial}(M) \longrightarrow \BDiff^{\partial}(M\cup_{\partial_{0}M}\widetilde{W}_{k}).
}
\end{equation}
We will refer to this map as the $k$-th stabilization map.
Let $r_{k}(M)$ be the quantity defined by,
\begin{equation} \label{equation: k-rank mfd}
r_{k}(M) = \max\{g \in N\; | \; \text{there exists an embedding, $W^{\# g}_{k}\setminus D^{4n+1} \longrightarrow M$}\}.
\end{equation}
Using the diffeomorphism classification for manifolds in
$\widehat{\mathcal{W}}^{f}_{4n+1}$ described in the previous section (and again in Section \ref{section:
  Highly Connected Manifolds of Odd Dimension}), the following result, combined with \cite{NP 14}
implies Theorem \ref{theorem: independence of rank}.  This is the main
 homological stability result that we prove in this paper.
\begin{theorem} \label{theorem: Main theorem}
For $n \geq 2$, let $M$ be a $2$-connected, compact, $(4n+1)$-dimensional manifold with non-empty boundary. 
If $k > 2$ is an odd integer, then the map on homology induced by (\ref{eq: k-stabilization map}),
$$\xymatrix{
(s_{k})_{*}: H_{\ell}(\BDiff^{\partial}(M); \; \Z) \longrightarrow H_{\ell}(\BDiff^{\partial}(M\cup_{\partial_{0}M}\widetilde{W}_{k}); \; \Z),
}$$
is an isomorphism if $2\ell \leq r_{k}(M) - 3$ and an epimorphism when $2\ell \leq r_{k}(M) - 1$.
\end{theorem}

\begin{remark} \label{remark: stabilization by exotic spheres}
In order to deduce Theorem \ref{theorem: independence of rank} from Theorem \ref{theorem: Main theorem}, one also needs to observe that the stabilization map $H_{*}(\BDiff^{\partial}(M); \; \Z) \longrightarrow H_{*}(\BDiff^{\partial}(M\cup_{\partial_{0}M}\Sigma); \; \Z)$ is an isomorphism for any $(4n+1)$-dimensional homotopy sphere $\Sigma$. 
This follows from the fact if $M'$ has non-empty boundary, then the connected sum $M'\#\Sigma$ is diffeomorphic to $M'$ for any homotopy sphere $\Sigma$.
\end{remark}

\subsection{Methodology} Our methods are similar to those used in \cite{GRW 12} and \cite{GRW 14}. 
For any integer $k \geq 2$, we construct a highly connected, semi-simplicial space $X_{\bullet}(M)_{k}$, which admits an action of the topological group $\Diff^{\partial}(M)$ that is transitive on the zero-simplices. 
Let $W'_{k}$ denote the manifold with boundary obtained from $W_{k}$ by removing an open disk. 
Roughly, the space of $p$-simplices of $X_{\bullet}(M)_{k}$ is defined to be the space of ordered $(p+1)$-tuples of pairwise disjoint embeddings $W'_{k} \hookrightarrow M$, with a certain prescribed boundary condition. 
This semi-simplicial space is similar to the ones constructed in \cite{GRW 12} and \cite{GRW 14}. The majority of the technical work of this paper is devoted to proving that if $M$ is $2$-connected and $k$ is odd, then the geometric realization $|X_{\bullet}(M)_{k}|$ is $\frac{1}{2}(r_{k}(M) - 4)$-connected. 
This is established in Section \ref{section: semi-simplicial spaces}, Theorem \ref{theorem: high connectivity of X}.

In order to prove that $|X_{\bullet}(M)_{k}|$ is $\frac{1}{2}(r_{k}(M) - 4)$-connected, we must compare it to an auxiliary simplicial complex $L(\pi^{\tau}_{2n}(M))_{k}$, based on the linking form associated to $M$. 
A $p$-simplex of $L(\pi^{\tau}_{2n}(M))_{k}$ is defined to be a set of $p+1$ pairwise orthogonal morphisms of linking forms 
$\xymatrix{
(\pi^{\tau}_{2n}(W'_{k}), \; b) \longrightarrow (\pi^{\tau}_{2n}(M), \; b),
}$
which mimic the pairwise disjoint embeddings $W'_{k} \rightarrow M$ from the semi-simplicial space $X_{\bullet}(M)_{k}$. 
In Section \ref{subsection: the linking complex}, we prove that the geometric realization $|L(\pi^{\tau}_{2n}(M))_{k}|$ is $\frac{1}{2}(r_{k}(M) - 4)$-connected (see Theorem \ref{theorem: linking form high-connectivity}). 
The proof of this theorem is very similar to the proof of \cite[Theorem 3.2]{GRW 14}.
One can view this as a ``mod $k$''-version of the result of Charney from \cite{Ch 87}. 

There is a map $|X_{\bullet}(M)_{k}| \longrightarrow |L(\pi^{\tau}_{2n}(M))_{k}|$
induced by sending an embedding $\varphi: W'_{k} \longrightarrow M$, which represents a $0$-simplex in $X_{\bullet}(M)_{k}$, to its induced morphism of linking forms, $\varphi_{*}: (\pi^{\tau}_{2n}(W'_{k}), \; b) \longrightarrow  (\pi^{\tau}_{2n}(M), \; b)$, which represents a vertex in $L(\pi^{\tau}_{2n}(M))_{k}$. 
To prove Theorem \ref{theorem: high connectivity of X} it will suffice to prove that this map induces an injection on homotopy groups $\pi_{j}(\underline{\hspace{.3cm}})$ when $j \leq \frac{1}{2}(r_{k}(M) - 4)$. 
This will require a number of new geometric constructions. 
In particular, we need a technique for realizing morphisms $(\pi^{\tau}_{2n}(W'_{k}), \; b) \longrightarrow  (\pi^{\tau}_{2n}(M), \; b)$ by actual embeddings $W'_{k} \rightarrow M$. 

We will need a suitable geometric model for the linking form.
This geometric model for the linking form will be based on $\Z/k$-manifolds and their intersections;  
this approach is similar to the one taken by Morgan and Sullivan in \cite{SM 74}.
The main technical device that we develop is a certain modulo-$k$ version of the \textit{Whitney trick} for modifying the intersections of embedded or immersed $\Z/k$-manifolds, see Theorems \ref{theorem: mod k whitney trick} and \ref{thm: modifying intersections}.
In Appendix \ref{section: k-immersions} we develop some results regarding the immersions and embeddings of $\Z/k$-manifolds. These results about $\Z/k$-manifolds could be of independent interest.

\begin{remark}
Our main homological stability result requires the integer $k$ to be odd. 
The source of this restriction is two-fold. 
The first source has to do with what was discussed in Remark \ref{remark: Wall's invariants} in that the classification of manifolds $W \in \mathcal{W}^{f}_{4n+1}$ with $H_{2n}(W; \Z/2) \neq 0$ is more complicated than the classification of the manifolds in $\widehat{\mathcal{W}}^{f}_{4n+1}$.
Indeed, for such $W \in \mathcal{W}^{f}_{4n+1}\setminus\widehat{\mathcal{W}}^{f}_{4n+1}$ the linking form is no longer a complete (almost) diffeomorphism invariant.
There is another invariant that one must consider, namely $\widehat{\phi} \in H^{2n+1}(W; \Z/2)$ defined in \cite{W 67}.
The second source of the need to restrict to odd $k$ is the technical results Theorem \ref{thm: modifying intersections} and Theorem \ref{theorem: represent by embedding 2}, which only hold for odd $k$.
If these theorems could be upgraded to include the case that $k$ is even, then a version of Theorem \ref{theorem: Main theorem} could in principle be proven to include manifolds $W$ with non-trivial $H_{2n}(W; \Z/2)$ using similar techniques to those developed in this paper. 
\end{remark}

\subsection{Organization}
Section \ref{section: simplicial techniques} is a recollection of some basic definitions and results about simplicial complexes and semi-simplicial spaces. 
In Section \ref{section: Algebra} we give an algebraic treatment of the linking form. 
In Section \ref{section: Highly Connected Manifolds of Odd Dimension} we review the diffeomorphism classification of the manifolds in $\widehat{\mathcal{W}}^{f}_{4n+1}$.
In Sections \ref{Bordism Groups of Singular Manifolds}, \ref{section: k-l manifolds}, and \ref{section: intersections} we give the necessary background on $\Z/k$-manifolds used in the proof of Theorem \ref{theorem: Main theorem}. 
In these three sections we state all of the necessary technical results regarding the intersections of immersions and embeddings of $\Z/k$-manifolds, but we put off most of the difficult proofs until Appendix \ref{section: modyifing higher dim intersections} and \ref{section: k-immersions}. 
In Section \ref{section: semi-simplicial spaces} we construct the primary semi-simplical space $X_{\bullet}(M)_{k}$ and prove that its geometric realization is highly connected. 
In Section \ref{Homological Stability} we show how high-connectivity of $|X_{\bullet}(M)_{k}|$ implies Theorem \ref{theorem: Main theorem}. 
In Appendix \ref{section: modyifing higher dim intersections} and Appendix \ref{section: k-immersions}, we prove several technical results regarding the intersections of immersions and embeddings of $\Z/k$-manifolds that were used earlier in the paper.

\subsection{Acknowledgments}
This paper forms part of the author's doctoral thesis at the University of Oregon. 
The author thanks Boris Botvinnik, his thesis advisor, for suggesting this particular problem and for the many helpful discussions relating to this project.  
The author also thanks the anonymous referee for numerous helpful suggestions. 

 \section{Simplicial Techniques} \label{section: simplicial techniques}
 \noindent
 In this section we recall a number of simplicial techniques that we will need to use throughout the paper. 
We will need to consider a variety of different simplicial complexes and semi-simplicial spaces. 
 \subsection{Cohen-Macaulay complexes}
Let $X$ be a simplicial complex. 
Recall that the \textit{link} of a
simplex $\sigma < X$, denoted by $\lk_{X}(\sigma)$, is defined to be the subcomplex
of $X$ consisting of all simplices $\zeta$ disjoint from $\sigma$, for which there exists a simplex $\xi$ such that both $\sigma$ and $\zeta$ are faces of $\xi$.
The following proposition was proven in \cite[Section 2.1]{GRW 14}. 
We will use it in the proof of Theorem \ref{theorem: linking form high-connectivity}.
\begin{proposition} \label{proposition: inclusion complex}
Let $X$ be a simplicial complex and let $Y \subset X$ be a full subcomplex.
Let $n$ be an integer with the property that for each $p$-simplex $\sigma < X$, the complex $Y\cap\lk_{X}(\sigma)$ is $(n-p-1)$-connected. 
Then the inclusion $|Y| \hookrightarrow |X|$ is $n$-connected.
\end{proposition}

We now present a key definition that will be used throughout the paper. 
\begin{defn} A simplicial complex $X$ is said to be
\textit{weakly Cohen-Macaulay} of dimension $n$ if it is
$(n-1)$-connected and the link of any $p$-simplex is
$(n-p-2)$-connected. In this case we write $\omega CM(X) \geq
n$. The complex $X$ is said to be \textit{locally weakly
Cohen-Macaulay} of dimension $n$ if the link of any $p$-simplex is
$(n - p - 2)$-connected (but no global connectivity is required on
$X$ itself). In this case we shall write $lCM(X) \geq n$.
\end{defn}
\noindent
The next result is from \cite[Theorem 2.4]{GRW 14} and is a generalization of the ``Coloring Lemma'' of Hatcher and Wahl from \cite[Lemma 3.1]{HW 10}.
\begin{theorem}[Galatius, Randal-Williams 2016] \label{thm: cohen mac trick}
Let $X$ be a
simplicial complex with $lCM(X) \geq n$, let $f: \partial I^{n}
\rightarrow |X|$ be a map, and let $h: I^{n} \rightarrow |X|$ be a null-homotopy of $f$. 
If $f$ is simplicial with respect to some PL triangulation $\partial I^{n} \cong |L|$, then this triangulation extends to a PL triangulation $I^{n} \cong |K|$ 
and $h$ is homotopic relative to $\partial I^{n}$ to a simplicial map $g: |K| \longrightarrow |X|$ such that:
\begin{enumerate} \itemsep.2cm
\item[(i)] for each vertex $v \in K\setminus L$, the star $v\ast\lk_{K}(v)$ intersects $L$ in a single (possibly empty) simplex;
\item[(ii)] for each vertex $v \in K\setminus L$, $g(\lk_{K}(v)) \subset \lk_{X}(g(v))$.
\end{enumerate}
\end{theorem}

%%%%%%%%%%%%%%%%%%%%%%%%%%%%%%%%%%%%%%%%%%%%%%%%%%%%%%%%%%%%%%%%%%%%%%%%%%%%%%%%%%%%%%%%%%%%%%%%%%%%%%%%%%%%%%%%%%%%%%%%%%%%%%%%%%%%%%%%%%%%%%%%%%%%%%%%%%%%%%%%%%%%%%%%%%%%%%%%%%%%%%%%%%%%%

We will need to use an application of the above result. 
In order to state this application we need to introduce a new definition and some terminology. 
\begin{defn} \label{defn: link lifting property 1}
Let $f: X \longrightarrow Y$ be a simplicial map between two simplicial complexes. 
The map is said to have the \textit{link lifting property} if the following condition holds:
\begin{itemize}
\item Let $y \in Y$ be a vertex and let $A \subset X$ be a set of vertices such that $f(a) \in \lk_{Y}(y)$ for all $a \in A$.
Then there exists $x \in X$ with $f(x) = y$ such that $a \in \lk_{X}(x)$ for all $a \in A$.
\end{itemize}
\end{defn}

In order for the map $f$ to have the link lifting property it is necessary that the above condition be satisfied for any set of vertices $A \subset X$. 
In practice, it can be difficult to verify the above lifting property for totally arbitrary subsets $A \subset X$. 
One may only be able to verify that it holds for sets of vertices subject to some condition. 
We will need to work with a refined version of Definition \ref{defn: link lifting property 1}. 

For any set $K$, a \textit{symmetric relation} is a subset $\mathcal{R} \subset K\times K$ that is invariant under the ``coordinate permutation'' map $(x, y) \mapsto (y, x)$. 
A subset $C \subset K$ is said to be in \textit{general position with respect to $\mathcal{R}$} if $(x, y) \in \mathcal{R}$ for any two elements $x, y \in C$. 
We will need to consider symmetric relations defined on the set of vertices of a simplicial complex. 
Let $X$ be a simplicial complex and let $\mathcal{R} \subset X\times X$ be a symmetric relation on the vertices of $X$. 
The relation $\mathcal{R}$ is said to be \textit{edge compatible} if for any $1$-simplex $\{x, y\} < X$, the pair $(x, y)$ is an element of $\mathcal{R}$.

\begin{defn} \label{defn: cone lifting property} 
Let $f: X \longrightarrow Y$ be a simplicial map between two simplicial complexes. 
Let $\mathcal{R} \subset X\times X$ be an edge compatible symmetric relation on the set of vertices of the complex $X$.
The map $f$ is said to have the \textit{link lifting property with respect to $\mathcal{R}$} if the following condition holds:
\begin{itemize}
\item
Let $y \in Y$ be any vertex and let $A \subset X$ be a finite set of vertices in general position with respect to $\mathcal{R}$ such that
 $f(a) \in \lk_{Y}(y)$ for all $a \in A$.
Then given any other finite set of vertices $B \subset X$ (not necessarily in general position), there exists a vertex $x \in X$ with $f(x) = y$ such that 
$a \in \textstyle{\lk_{X}(x)}$ for all $a \in A$ and $(b, x) \in \mathcal{R}$ for all $b \in B$.
\end{itemize}
\end{defn}
The lemma below abstracts and formalizes the argument used in the proof of \cite[Lemma 5.4]{GRW 14}.
This lemma will be applied in the proof of Lemma \ref{thm: high connectedness of K}. 
\begin{lemma} \label{lemma: link lift lemma}
Let $X$ and $Y$ be simplicial complexes and let $f: X \longrightarrow Y$ be a simplicial map.
Let $\mathcal{R} \subset X\times X$ be an edge compatible symmetric relation. 
Suppose that the following conditions are met:
\begin{enumerate} \itemsep2pt
\item[(i)] $f$ has the \textit{link lifting property} with respect to $\mathcal{R}$;
\item[(ii)] $lCM(Y) \geq n$;
\end{enumerate}
Then the induced map $|f|_{*}: \pi_{j}(|X|) \longrightarrow \pi_{j}(|Y|)$ is injective for all $j \leq n-1$. 
Furthermore, suppose that in addition to properties (i) and (ii) the map $f$ satisfies:
\begin{enumerate}
\item[(iii)] $f(\lk_{X}(\zeta)) \leq \lk_{Y}(f(\zeta))$ for all simplices $\zeta < X$.
\end{enumerate}
Then it follows that $lCM(X) \geq n$. 
\end{lemma}
\begin{proof}
For $l+1 \leq n$, let $h: \partial I^{l+1} \longrightarrow |X|$ be a map. 
Let $H: I^{l+1} \longrightarrow |Y|$ be a null-homotopy of the composition $|f|\circ h$, i.e.\ $H|_{\partial I^{l+1}} = |f|\circ h$. 
To prove that $|f|_{*}: \pi_{l}(|X|) \longrightarrow \pi_{l}(|Y|)$ is injective for all $l \leq n-1$, it will suffice to construct a lift $\widehat{H}$ of $H$ that makes the diagram
$$\xymatrix{
\partial I^{l+1} \ar@{^{(}->}[d] \ar[rr]^{h} && |X| \ar[d]^{|f|} \\
I^{l+1} \ar[rr]^{H}  \ar@{-->}[urr]^{\widehat{H}} && |Y| 
}$$
commute. 
We may assume that $h$ is simplicial with respect to some PL triangulation $\partial I^{l+1} \cong |L|$. 
Since $lCM(Y) \geq n$, by Theorem \ref{thm: cohen mac trick} we may assume that $H$ is simplicial with respect to a PL triangulation $I^{l+1} \cong |K|$ extending the one on the boundary. 
Furthermore, we may arrange that the map $H$ and the triangulation $K$ satisfy the following further conditions:
\begin{enumerate} \itemsep.2cm
\item[(1)] for any interior vertex $v \in K \setminus L$, the star $v\ast \lk_{K}(v)$ intersects $L$ in a single (possibly empty) simplex;
\item[(2)]  $H(\lk_{K}(v)) \leq \lk_{Y}(H(v))$ for any interior vertex $v \in K\setminus L$.
\end{enumerate}
Choose an enumeration of the vertices in $K$ as $v_{1}, \dots, v_{N}$ with the vertices in $L$ before the vertices in $K\setminus L$. 
We shall inductively pick lifts of each $H(v_{i}) \in Y$ with $v_{i} \in K\setminus L$ to a vertex $\widehat{H}(v_{i}) \in X$ satisfying the following:
\begin{enumerate} \itemsep.2cm
\item[(a)] if $v_{i}, v_{k}$ are adjacent vertices in $K$ then $\widehat{H}(v_{i}), \widehat{H}(v_{k})$ are adjacent in $X$; 
\item[(b)] if $v_{i} \in K\setminus L$, then $(\widehat{H}(v_{i}), \widehat{H}(v_{k})) \in \mathcal{R}$ for $k < i$.
\end{enumerate}
Notice that condition (a) implies that the resulting lift $\widehat{H}$ yields a well defined simplicial map. 
Condition (b) is an extra condition that will allow us to inductively define the lifts using the link lifting property with respect to $\mathcal{R}$.

We construct $\widehat{H}$ inductively as follows.
Suppose that $H(v_{1}), \dots, H(v_{i-1})$ have already been lifted to $\widehat{H}(v_{1}), \dots, \widehat{H}(v_{i-1})$ satisfying (a) and (b) above.
Let $v_{i} \in K\setminus L$ be the next vertex in the list. 
Let $A \subset \{v_{1}, \dots, v_{i-1}\}$ be the subset consisting of those vertices that are adjacent to $v_{i}$ and let $B = \{v_{1}, \dots, v_{i-1}\}\setminus A$. 
By condition (2) we have $H(a) \in \lk_{Y}(v_{i})$ for all $a \in A$.
We wish to apply the link lifting property (with respect to $\mathcal{R}$) to find a vertex $\widehat{v}_{i} \in X$ with $f(\widehat{v}_{i}) = H(v_{i})$ such that 
$\widehat{H}(a) \in \lk_{X}(\widehat{v}_{i})$ for all $a \in A$ and $(\widehat{v}_{i}, \widehat{H}(b)) \in \mathcal{R}$ for all $b \in B$.
Then by setting $\widehat{H}(v_{i}) = \widehat{v}_{i}$, we obtain the desired lift of $H(v_{i})$ thus completing the inductive step. 

In order to apply the link lifting property in this context, we must verify that the set of vertices $\{\widehat{H}(a) | a \in A\} \subset X$ is in general position with respect to $\mathcal{R}$.
We do this as follows.
Let $A' \subset A$ denote the intersection $A\cap L$. 
By condition ($1$) above it follows that the set of vertices $A'$ forms a single (possibly empty) simplex in $L$. 
Since $\widehat{H}$ is a simplicial map, the image $\widehat{H}(A') \subset X$ forms a single simplex as well. 
Since the relation $\mathcal{R}$ is assumed to be edge compatible, it follows that the set $\widehat{H}(A') \subset X$ is in general position with respect to $\mathcal{R}$. 
Let $A''$ denote the complement $A \setminus A'$. 
By condition (b), the image $\widehat{H}(A'') \subset X$ is in general position with respect to $\mathcal{R}$ and furthermore (by condition (b) again) we have $(\widehat{H}(a'), \widehat{H}(a'')) \in \mathcal{R}$ for all $a' \in A'$ and $a'' \in A''$. 
It follows from these observations that $\widehat{H}(A) = \widehat{H}(A')\cup \widehat{H}(A'')$ is in general position with respect to $\mathcal{R}$.
With this established, we may apply the link lifting property to extend the definition of $\widehat{H}$ over the vertex $v_{i}$. 
This completes the induction and proves that the map 
$|f|_{*}: \pi_{j}(|X|) \longrightarrow \pi_{j}(|Y|)$ is injective for all $j \leq n-1$.

Assume now that in addition to properties (i) and (ii) from the statement of the lemma we have $f(\lk_{X}(\sigma)) \leq \lk_{Y}(f(\sigma))$ for all simplices $\sigma < X$.
We will show that $lCM(X) \geq n$. 
Let $\zeta \leq X$ be a $p$-simplex. 
Since $f$ has the link lifting property with respect to $\mathcal{R}$, it follows easily that the restriction map 
$
\xymatrix{
f|_{\lk_{X}(\zeta)}: \lk_{X}(\zeta) \; \longrightarrow \; \lk_{Y}(f(\zeta))
}
$
has the link lifting property with respect to $\mathcal{R}$ as well. 
Since $lCM(Y) \geq n$, it follows from \cite[Lemma 2.2]{GRW 12} that
$l CM[\lk_{Y}(f(\zeta))] \; \geq \; n - p - 1.$ 
It follows from the result proven in the previous paragraph that map on $\pi_{j}(\underline{\hspace{.3cm}})$ induced by $f|_{\lk_{X}(\zeta)}: \lk_{X}(\zeta) \; \longrightarrow \; \lk_{Y}(f(\zeta))$
is injective for $j \leq n - p - 2$. 
Since $\lk_{Y}(f(\zeta))$ is $(n-p-2)$-connected it follows that $\lk_{X}(\zeta)$ is $(n-p-2)$-connected as well. 
This proves that $lCM(X) \geq n$ and completes the proof of the result.
\end{proof}
\begin{remark}
The main technical challenge in this paper will be to prove that a certain simplicial map (see (\ref{eq: algebraic comparison map}) and Section \ref{subsection: complex of k embeddings}) has the link lifting property with respect to a certain suitable symmetric relation. 
This is established in the proof of Lemma \ref{thm: high connectedness of K} but it uses the geometric techniques regarding $\Z/k$-manifolds developed throughout Sections \ref{Bordism Groups of Singular Manifolds}, \ref{section: k-l manifolds}, \ref{section: intersections} and in the appendix. 
\end{remark}
Notice that the link lifting property (as in Definition \ref{defn: link lifting property 1}) is a stronger condition than the link lifting property with respect to some edge compatible symmetric relation $\mathcal{R}$. 
In other words, if $f: X \longrightarrow Y$ has the link lifting property then it has the link lifting property with respect to any edge compatible symmetric relation $\mathcal{R} \subset X\times X$.
From this observation we obtain the following corollary.
\begin{corollary} \label{corollary: link lifting no relation}
Let $X$ and $Y$ be simplicial complexes and let $f: X \longrightarrow Y$ be a simplicial map.
Suppose that::
\begin{enumerate} \itemsep.2cm
\item[(i)] $f$ has the \textit{link lifting property} (as in Definition \ref{defn: link lifting property 1});
\item[(ii)] $lCM(Y) \geq n$;
\item[(iii)] $f(\lk_{X}(\zeta)) \leq \lk_{Y}(f(\zeta))$ for all simplices $\zeta < X$.
\end{enumerate}
Then the induced map $|f|_{*}: \pi_{j}(|X|) \longrightarrow \pi_{j}(|Y|)$ is injective for all $j \leq n-1$ and $lCM(X) \geq n$. 
\end{corollary}

\subsection{Topological flag complexes}
We will need to work with a certain class of semi-simplicial spaces called \textit{topological flag complexes} (see \cite[Definition 6.1]{GRW2 14}). 
\begin{defn}
Let $X_{\bullet}$ be a semi-simplicial space.
We say that $X_{\bullet}$ is a \textit{topological flag complex} if for each integer $p \geq 0$, 
\begin{enumerate} \itemsep.2cm
\item[(i)] the map $X_{p} \longrightarrow (X_{0})^{\times (p+1)}$ to the $(p+1)$-fold product (which takes a $p$-simplex to its $(p+1)$ vertices) is a homeomorphism onto its image, which is an open subset, 
\item[(ii)] a tuple $(v_{0}, \dots, v_{p}) \in  (X_{0})^{\times(p+1)}$ lies in the image of $X_{p}$ if and only if $(v_{i}, v_{j}) \in X_{1}$ for all $i < j$.
\end{enumerate}
\end{defn}
\noindent
If $X_{\bullet}$ is a topological flag complex, we may denote any $p$-simplex $x \in X_{p}$ by a $(p+1)$-tuple $(x_{0}, \dots, x_{p})$ of zero-simplices. 

\begin{defn} \label{defn: semi-simplicial cohen mac}
Let $X_{\bullet}$ be a topological flag complex and let $x = (x_{0}, \dots, x_{p}) \in X_{p}$ be a $p$-simplex. 
The \textit{link} of $x$, denoted by $X_{\bullet}(x) \subset X_{\bullet}$, is defined to be the sub-semi-simplicial space whose $l$-simplices are given by the space of all ordered lists $(y_{0}, \dots, y_{l}) \in X_{l}$ such that the list $(x_{0}, \dots, x_{p}, y_{0}, \dots, y_{\ell}) \in (X_{0})^{\times (p + \ell + 2)}$, is a $(p+ \ell + 1)$-simplex.

It is easily verified that the link $X_{\bullet}(x)$ is a topological flag complex as well. 
The topological flag complex $X_{\bullet}$ is said to be
\textit{weakly Cohen-Macaulay} of dimension $n$ if its geometric realization is 
$(n-1)$-connected and if for any $p$-simplex $x \in X_{p}$, the geometric realization of the link $|X_{\bullet}(x)|$
 is $(n-p-2)$-connected. In this case we write $\omega CM(X_{\bullet}) \geq n$. 
\end{defn}
\noindent
The main result from this section is a result about the discretization of a topological flag complex. 
\begin{defn} \label{defn: associated simplicial complex}
Let $X_{\bullet}$ be a semi-simplicial space. 
Let $X^{\delta}_{\bullet}$ be the semi-simplicial set defined by setting $X^{\delta}_{p}$ equal to the discrete topological space with underlying set equal to $X_{p}$, for each integer $p \geq 0$. 
We will call the semi-simplicial set $X^{\delta}_{\bullet}$ the \textit{discretization} of $X_{\bullet}$. 
\end{defn}
\noindent
The following theorem is proven by repackaging several results from \cite{GRW 14}. 
In particular, the proof is basically the same as the proof of \cite[Theorem 5.5]{GRW 14}.
We provide a sketch of the proof here and we provide references to the key technical lemmas employed from \cite{GRW 14}.

\begin{theorem} \label{theorem: discretization}
Let $X_{\bullet}$ be a topological flag complex and suppose that $\omega CM(X^{\delta}_{\bullet}) \geq n$. 
Then the geometric realization $|X_{\bullet}|$ is $(n-1)$-connected. 
\end{theorem}

\begin{proof}[Proof Sketch]
For integers $p, q \geq 0$, let $Y_{p,q} = X_{p+q+1}$ be topologized as a subspace of the product
$(X_{0})^{\times p}\times (X^{\delta}_{0})^{\times q}$. 
The assignment $[p, q] \mapsto Y_{p,q}$ defines a bi-semi-simplicial space with augmentations 
$$
\varepsilon: Y_{\bullet, \bullet} \longrightarrow X_{\bullet}, \quad
\delta: Y_{\bullet, \bullet} \longrightarrow X^{\delta}_{\bullet}.
$$
This doubly augmented bi-semi-simplicial space is analogous to the one considered in \cite[Definition 5.6]{GRW 14}. 
Let $\iota: X^{\delta}_{\bullet} \longrightarrow X_{\bullet}$ be the map induced by the identity. 
By \cite[Lemma 5.7]{GRW 14}, there exists a homotopy of maps, 
\begin{equation} \label{eq: homotopy of augmentations}
|\iota|\circ|\delta| \simeq |\epsilon|: |Y_{\bullet, \bullet}| \longrightarrow |X_{\bullet}|.
\end{equation}
For each integer $p$, consider the map
\begin{equation} \label{equation: p-augmentation}
|Y_{p, \bullet}| \longrightarrow X_{p}
\end{equation}
induced by $\epsilon$.
By how $Y_{\bullet, \bullet}$ was constructed, it follows from \cite[Proposition 2.8]{GRW 14} that for each $p$, (\ref{equation: p-augmentation}) is a \textit{Serre-microfibration}.
For any $x \in X_{p}$, the fibre over $x$ is equal to the space $|X^{\delta}_{\bullet}(x)|$, where $X^{\delta}_{\bullet}(x)$ is the link of the $p$-simplex $x$, as defined in Definition \ref{defn: semi-simplicial cohen mac}. 
Since $\omega CM(X^{\delta}_{\bullet}) \geq n$, this implies that the fibre of (\ref{equation: p-augmentation}) over any $x \in X_{p}$ is $(n-p-2)$-connected. 
Using the fact that this map is a Serre-microfibration, \cite[Proposition 2.6]{GRW 14} then implies that (\ref{equation: p-augmentation})
is $(n-p-1)$-connected. 
It then follows by \cite[Proposition 2.7]{GRW 14} that the map 
\begin{equation}
|\epsilon|: |Y_{\bullet, \bullet}| \longrightarrow |X_{\bullet}|
\end{equation}
is $(n-1)$-connected.  
The homotopy from (\ref{eq: homotopy of augmentations}) implies that the map $|\iota|: |X_{\bullet}^{\delta}| \longrightarrow |X_{\bullet}|$ induces a surjection on homotopy groups $\pi_{j}(\underline{\hspace{.3cm}})$ for all $j \leq n-1$. 
The proof of the theorem then follows from the fact that $|X_{\bullet}^{\delta}|$ is $(n-1)$-connected by hypothesis. 
\end{proof}

\subsection{Transitive group actions}
In order to prove our homological stability theorem, we will need to consider groups acting on semi-simplicial spaces and simplicial complexes. 
We will need a technique for determining when such actions are transitive. 
For the lemma that follows, 
let $X_{\bullet}$ be a topological flag complex, let $G$ be a topological group, and let 
$$G\times X_{\bullet} \longrightarrow X_{\bullet}, \quad (g, \sigma) \mapsto g\cdot \sigma$$
be a continuous group action. 
\begin{lemma} \label{lemma: flag complex transitivity}
Let $G$ and $X_{\bullet}$ be as above and suppose that the following conditions hold:
\begin{itemize}
\item for any $1$-simplex $(v, w) \in X_{1}$, there exists $g \in G$ such that $g\cdot v = w$,
\item for any two vertices $x, y$ that lie on the same path-component of $X_{0}$, there exists $g \in G$ such that $g\cdot x = y$,
\item the geometric realization $|X_{\bullet}|$ is path-connected.
\end{itemize}
Then for any two vertices $x, y \in X_{0}$, there exists $g \in G$ such that $g\cdot x = y$.
\end{lemma}
\begin{proof}
We define an equivalence relation on the elements of $X_{0}$ by setting $x \sim y$ if there exists $g \in G$ such that $g\cdot x = y$. 
Since $G$ is a group (and thus every element has a multiplicative inverse), it follows that this relation is indeed an equivalence relation, i.e.\ it is transitive, reflexive, and symmetric. 
By transitivity of the relation, it follows from the first condition in the statement of the lemma that $x \sim y$ if there exists some zig-zag of edges connecting $x$ and $y$. 
It also follows that $x \sim y$ if $x$ and $y$ lie on the same path component of $X_{0}$. 

Let $v, w \in X_{0}$ be any two zero simplices. 
We will prove that there exists $g \in G$ such that $g\cdot v = w$.  
Since the geometric realization $|X_{\bullet}|$ is path-connected, it follows that there exists a vertex $v'$ in the path component containing $v$ and a vertex $w'$ in the path component containing $w$, such that $v'$ and $w'$ are connected by a zig-zag of edges. 
We have $v \sim v' \sim w' \sim w$, and thus $v \sim w$. 
This concludes the proof of the lemma. 
\end{proof}

%%%%%%%%%%%%%%%%%%%%%%%%%%%%%%%%%%%%%%%%%%%%%%%%%%%%%%%%%%%%%%%%%%%%%%%%%%%%%%%%%%%%%%%%%%%%%%%%%%%%%%%%%%%%%%%%%%%%%%%%%%%%%%%%%%%%%%%%%%%%%%%%%%%%%%%%%%%%%%%%%%%%%%%%%%%%%%%%%%%%%%%%%%%%%%%%%%%%%%%%%%%%%%%

\section{Algebraic Structures} \label{section: Algebra}
\subsection{Linking forms} \label{subsection: linking forms}
The basic algebraic structure that we will encounter is that of a bilinear form on a finite abelian group. 
For $\epsilon = \pm 1$, a pair $(\mb{M}, b)$ is said to be a ($\epsilon$-symmetric) \textit{linking form} if $\mb{M}$ is a finite abelian group and $b: \mb{M}\otimes \mb{M} \longrightarrow \Q/\Z$ an $\epsilon$-symmetric bilinear map. 
A morphism between linking forms is defined to be a group homomorphism $f: \mb{M} \longrightarrow \mb{N}$ such that $b_{\mb{M}}(x, y) = b_{\mb{N}}(f(x), f(y))$ for all $x, y \in \mb{M}$. 
We denote by $\mathcal{L}_{\epsilon}$ the category of all $\epsilon$-symmetric linking forms. 
By forming direct sums, $\mathcal{L}_{\epsilon}$ obtains the structure of an \textit{additive category.}
\begin{Notation}
We will usually denote linking forms by their underlying abelian group. 
We will always denote the associated bilinear map by $b$. 
If more than one linking form is present, we will decorate $b$ with a subscript so as to eliminate ambiguity. 
\end{Notation}

\noindent
For $\mb{M}$ a linking form and $\mb{N} \leq \mb{M}$ a subgroup, $\mb{N}$ automatically inherits the structure of a sub-linking form of $\mb{M}$ by restricting $b_{\mb{M}}$ to $\mb{N}$.  
We will denote by $\mb{N}^{\perp} \leq \mb{M}$ the \textit{orthogonal complement} to $\mb{N}$ in $\mb{M}$.  
Two sub-linking forms $\mb{N}_{1}, \mb{N}_{2} \leq \mb{M}$ are said to be \textit{orthogonal} if
$\mb{N}_{1} \leq \mb{N}_{2}^{\perp}$,
 $\mb{N}_{2} \leq \mb{N}_{1}^{\perp}$,
and $\mb{N}_{1}\cap\mb{N}_{2} = 0$. 
If $\mb{N}_{1}, \mb{N}_{2} \leq \mb{M}$ are orthogonal sub-linking forms, we let $\mb{N}_{1}\perp\mb{N}_{2} \leq \mb{M}$ denote the sub-linking form given by the sum $\mb{N}_{1}+\mb{N}_{2}$. 
If $\mb{M}_{1}$ and $\mb{M}_{2}$ are two linking forms, the (external) direct sum $\mb{M}_{1}\oplus\mb{M}_{2}$ obtains the structure of a linking form in a natural way by setting
\begin{equation} \label{equation: direct sum linking form}
b_{\mb{M}_{1}\oplus\mb{M}_{2}}(x_{1} + x_{2}, \; y_{1} + y_{2}) = b_{\mb{M}_{1}}(x_{1}, y_{1}) + b_{\mb{M}_{1}}(x_{2}, y_{2}) \quad \quad \text{for \; $x_{1}, y_{1} \in \mb{M}_{1}$, \; $x_{2}, y_{2} \in \mb{M}_{2}$.}
\end{equation}
We will always assume that the direct sum $\mb{M}_{1}\oplus\mb{M}_{2}$ is equipped with the linking form structure given by (\ref{equation: direct sum linking form}). 
An element $\mb{M} \in \Ob(\mathcal{L}_{\epsilon})$ is said to be \textit{non-singular} if the duality homomorphism 
\begin{equation} \label{eq: standard duality map}
\xymatrix{
T: \mb{M} \longrightarrow \Hom_{\Ab}(\mb{M}, \Q/\Z), \quad x \mapsto b(x, \underline{\hspace{.3cm}})
}
\end{equation}
is an isomorphism of abelian groups. 

We will mainly need to consider the category $\mathcal{L}_{\epsilon}$ in the case where $\epsilon = -1$. 
We denote by $\mathcal{L}_{-1}^{s}$ the full subcategory of $\mathcal{L}_{-1}$ consisting of linking forms that are \textit{strictly skew symmetric}, or in other words $\mathcal{L}_{-1}^{s}$ is the category of all skew symmetric linking forms $\mb{M}$ for which $b_{\mb{M}}(x, x)= 0$ for all $x \in \mb{M}$ (even in the case when $x$ is an element of order $2$).

We proceed to define certain basic, non-singular elements of $\mathcal{L}_{-1}^{s}$ as follows. 
For a positive integer $k \geq 2$, let $\mb{W}_{k}$ denote the abelian group $\Z/k\oplus \Z/k$. 
Let $\rho$ and $\sigma$ denote the standard generators $(1, 0)$ and $(0, 1)$ respectively. 
We then let $b: \mb{W}_{k} \longrightarrow \Q/\Z$ be the  strictly skew symmetric bilinear form determined by the values
\begin{equation}
b(\rho, \sigma) = -b(\sigma, \rho) =  \tfrac{1}{k}, \quad \quad b(\rho, \rho) = b(\sigma, \sigma) = 0. 
\end{equation}
With $b$ defined in this way, it follows that $\mb{W}_{k}$ is a non-singular object of $\mathcal{L}_{-1}^{s}$. 
It follows easily that if $k$ and $\ell$ are relatively prime, then $\mb{W}_{k}\oplus\mb{W}_{\ell}$ and $\mb{W}_{k\cdot \ell}$ are isomorphic as objects of $\mathcal{L}_{-1}^{s}$. 
For $g \geq 2$ an integer, we will let $\mb{W}^{g}_{k}$ denote the $g$-fold direct sum $(\mb{W}_{k})^{\oplus g}$. 

For $k \in \N$, let $C_{k}$ denote the cyclic subgroup of $\Q/\Z$ generated by the element $1/k \mod 1$. 
Any group homomorphism $h: \mb{W}_{k} \longrightarrow \Q/\Z$ must factor through the inclusion $C_{k} \hookrightarrow \Q/\Z$. 
Hence, it follows that the duality map from (\ref{eq: standard duality map}) induces an isomorphism of abelian groups,
\begin{equation}
\xymatrix{
\mb{W}_{k} \stackrel{\cong} \longrightarrow \Hom_{\Ab}(\mb{W}_{k}, C_{k}). 
}
\end{equation}

\begin{lemma} \label{lemma: split injectivity}
Let $k \geq 2$ be a positive integer and let $\mb{M} \in \Ob(\mathcal{L}_{-1}^{s})$. 
Then any morphism 
$$f: \mb{W}_{k} \longrightarrow \mb{M}$$
is split injective and there is an orthogonal direct sum decomposition, $f(\mb{W}_{k})\perp f(\mb{W}_{k})^{\perp} = \mb{M}$.  
\end{lemma}
\begin{proof}
Let $x$ and $y$ denote the elements of $\mb{M}$ given by $f(\rho)$ and $f(\sigma)$ respectively where $\rho$ and $\sigma$ are the standard generators of $\mb{W}_{k}$. 
Let $T: \mb{M} \longrightarrow \Hom(\mb{M}, \Q/\Z)$ denote the duality map from (\ref{eq: standard duality map}). 
Since both $x$ and $y$ have order $k$, it follows that the homomorphisms 
$$b(x, \underline{\hspace{.3cm}}), \; b(y, \underline{\hspace{.3cm}}): \mb{M} \longrightarrow \Q/\Z$$ 
factor through the inclusion $C_{k} \hookrightarrow \Q/\Z$. 
Define a group homomorphism (which is not a morphism of linking forms) by the formula
$$\varphi: \mb{M} \longrightarrow \mb{W}_{k}, \quad \varphi(z) = b(x, z)\cdot\rho  + b(y, z)\cdot\sigma.$$
It is clear that the kernel of $\varphi$ is the orthogonal complement $f(\mb{W}_{k})^{\perp}$ and that the morphism $f: \mb{W}_{k} \longrightarrow \mb{M}$ gives a section of $\varphi$. 
This completes the proof. 
\end{proof}
\noindent
The following theorem is a particular case of the classification theorem of Wall from \cite[Lemma 7]{W 64}.
The classification of non-singular objects of $\mathcal{L}_{-1}^{s}$ is analogous to the classification of finite abelian groups. 
\begin{theorem} \label{thm: classification of linking modules}
Let $\mathbf{M} \in \Ob(\mathcal{L}_{-1}^{s})$ be non-singular. 
Then there is an isomorphism,
$$\mathbf{M} \; \cong \; \mathbf{W}^{\ell_{1}}_{p_{1}^{n_{1}}}\oplus \cdots \oplus \mathbf{W}^{\ell_{r}}_{p_{r}^{n_{r}}}$$
where $p_{j}$ is a prime number and $\ell_{j}$ and $n_{j}$ are positive integers for $j = 1, \dots, r$.
Furthermore, the above direct sum decomposition is unique up to isomorphism. 
\end{theorem}
\noindent
We now define a notion of rank analogous to (\ref{equation: k-rank mfd}) for skew-symmetric linking forms. 
\begin{defn} \label{defn: rank of linking form}
Let $\mb{M}$ be a linking form and let $k \geq 2$ be a positive integer. 
We define the \textit{$k$-rank} of $\mb{M}$ to be the quantity,
$r_{k}(\mb{M}) = \max\{ g \in \N \; | \; \text{there exists a morphism $\mb{W}^{g}_{k} \rightarrow \mb{M}$}\}.$
We then define the \textit{stable $k$-rank} of $\mb{M}$ to be the quantity,
$\bar{r}_{k}(M) = \max\{r_{k}(\mb{M}\oplus\mb{W}^{g}_{k}) - g \; | \; g \in \N \}$.
\end{defn}

\begin{corollary} \label{corollary: stable rank reduction}
Let $f: \mb{W}^{g}_{k} \longrightarrow \mb{M}$ be a morphism of linking forms. 
Then $\bar{r}_{k}(f(\mb{W}_{k})^{\perp}) \geq \bar{r}_{k}(\mb{M}) - g$. 
\end{corollary}
\begin{proof}
This follows immediately from the orthogonal splitting $f(\mb{W}^{g}_{k})\perp f(\mb{W}^{g}_{k})^{\perp} = \mb{M}$ and the definition of the stable $k$-rank. 
\end{proof}

\subsection{The linking complex} \label{subsection: the linking complex}
We now define a certain simplicial complex, analogous to the one from \cite[Definition 3.1]{GRW 14}, to be used in our proof of Theorem \ref{theorem: Main theorem}. 
\begin{defn}
Let $\mb{M} \in \Ob(\mathcal{L}_{-1}^{s})$ and let $k \geq 2$ be a positive integer. 
We define $L(\mb{M})_{k}$ to be the simplicial complex whose vertices are given by morphisms $f: \mb{W}_{k} \longrightarrow \mb{M}$ of linking forms. 
The set $\{f_{0}, \dots, f_{p}\}$ is a $p$-simplex if the sub-linking forms $f_{i}(\mb{W}_{k}) \leq \mb{M}$ are pairwise orthogonal.  
\end{defn}
Suppose that $\sigma = \{f_{0}, \dots, f_{p}\}$ is a $p$-simplex in $L(\mb{M})_{k}$.
Let $\mb{M}' \leq \mb{M}$ denote the sub-linking-form given by the orthogonal complement $[\sum f_{i}(\mb{W}_{k})]^{\perp}$.
It follows from the definition of the link of a simplex that there is an isomorphism of simplicial complexes, 
\begin{equation}  \label{equation: link of simplex}
\xymatrix{
\lk_{L(\mb{M})_{k}}(\sigma) \cong L(\mb{M}')_{k}.
}
\end{equation}
Below are two formal consequences of path connectivity of $L(\mb{M})_{k}$.
They are proven in the exact same way as \cite[Propositions 3.3 and 3.4]{GRW 14}. 
\begin{proposition}[Transitivity] \label{proposition: transitivity}
If $|L(\mb{M})_{k}|$ is path-connected and $f_{0}, f_{1}: \mb{W}_{k} \rightarrow \mb{M}$ are morphisms of linking forms, then there is an automorphism of linking forms $h: \mb{M} \rightarrow \mb{M}$ such that $f_{1} = h\circ f_{0}$. 
\end{proposition}

\noindent

\begin{proposition}[Cancellation] \label{proposition: cancelation alg}
Suppose that $\mb{M}$ and $\mb{N}$ are linking forms and there is an isomorphism $\mb{M}\oplus\mb{W}_{k} \cong \mb{N}\oplus \mb{W}_{k}$. 
If $|L(\mb{M}\oplus\mb{W}_{k})_{k}|$ is path-connected, then there is also an isomorphism $\mb{M} \cong \mb{N}$. 
\end{proposition}

The main result that we will prove about the above complex is the following theorem. 
The proof is very similar to the proof of \cite[Theorem 3.2]{GRW 14}.
\begin{theorem} \label{theorem: linking form high-connectivity}
Let $g, k \in \N$ and let $\mb{M} \in \Ob(\mathcal{L}_{-1}^{s})$ be a linking form with $\bar{r}_{k}(\mb{M}) \geq g$. 
Then the geometric realization $|L(\mb{M})_{k}|$ is $\frac{1}{2}(g -4)$-connected and $lCM(L(\mb{M})_{k}) \geq \frac{1}{2}(g -1)$. 
\end{theorem}

The proof of Theorem \ref{theorem: linking form high-connectivity} follows the same inductive argument as the proof of \cite[Theorem 3.2]{GRW 14}.
We will need two key algebraic results (Proposition \ref{prop: automorphisms} and Corollary \ref{cor: rank reduction}) given below which are analogous to \cite[Proposition 4.1 and Corollary 4.2]{GRW 14}.

\begin{proposition} \label{prop: automorphisms}
Let $k, g \in \N$ with $k \geq 2$. 
Let $\Aut(\mb{W}^{g+1}_{k})$ act on $\mb{W}^{g+1}_{k}$, and consider the orbits of elements of $\mb{W}_{k}\oplus 0 \leq \mb{W}_{k}^{g+1}$. 
We then have 
$\Aut(\mb{W}_{k}^{g+1})\cdot(\mb{W}_{k}\oplus \mb{0}) = \mb{W}_{k}^{g+1}.$
\end{proposition}
\begin{proof}
We will prove that for any $v \in \mb{W}^{g+1}_{k}$, there is an automorphism $\varphi: \mb{W}^{g+1}_{k} \longrightarrow \mb{W}_{k}^{g+1}$ such that $v \in \varphi(\mb{W}_{k}\oplus\mb{0})$. 
An element $v \in \mb{W}^{g+1}_{k}$ is said to be \textit{primitive} if the subgroup $\langle v \rangle \leq \mb{W}_{k}^{g+1}$ generated by $v$, splits as a direct summand. 
Every element of $\mb{W}^{g+1}_{k}$ is the integer multiple (reduced $\mod k$) of a primitive element. 
Hence it will suffice to prove the statement in the case that $v$ is a primitive element. 

Let $v \in \mb{W}_{k}^{g+1}$ be a primitive element. 
Since the linking form $\mb{W}^{g+1}_{k}$ is non-singular and $v$ is primitive, it follows that there exists $w \in \mb{W}^{g+1}_{k}$ such that $b(w, v) = \frac{1}{k} \mod 1$. 
We may then define a morphism $f: \mb{W}_{k} \longrightarrow \mb{W}^{g+1}_{k}$ by setting $f(\sigma)  = v$ and $f(\rho) = w$, where $\sigma$ and $\rho$ are the standard generators of $\mb{W}_{k}$. 
Consider the orthogonal splitting $f(\mb{W}_{k})\perp f(\mb{W}_{k})^{\perp} = \mb{W}_{k}^{g+1}$. 
Since both $\mb{W}^{g+1}_{k}$ and $f(\mb{W}_{k})$ are non-singular, it follows that the orthogonal complement $f(\mb{W}_{k})^{\perp}$ is nonsingular as well.
It then follows from the classification theorem (Theorem \ref{thm: classification of linking modules}) that there exists an isomorphism 
$h: \mb{W}^{g}_{k} \stackrel{\cong} \longrightarrow  f(\mb{W}_{k})^{\perp}$ (according to Theorem \ref{thm: classification of linking modules}, there is only one such way, up to isomorphism, to write $\mb{W}^{g+1}_{k}$ as the direct sum of $\mb{W}_{k}$ with another non-singular linking form).  
The morphism given by the direct sum of maps
$$
\varphi := f\oplus h: \mb{W}_{k}\oplus\mb{W}^{g+1}_{k} \longrightarrow f(\mb{W})\perp f(\mb{W})^{\perp},
$$
is an isomorphism such that  $v \in \varphi(\mb{W}_{k}\oplus\mb{0})$. 
This concludes the proof of the proposition.
\end{proof}

\begin{corollary} \label{cor: rank reduction}
Let $\mb{M}$ be a linking form with $r_{k}(\mb{M}) \geq g$ and let $\varphi: \mb{M} \longrightarrow C_{k}$ be a group homomorphism. 
Then $r_{k}(\Ker(\varphi)) \geq g - 1$. 
Similarly if $\bar{r}_{k}(\mb{M}) \geq g$ then $\bar{r}_{k}(\Ker(\varphi)) \geq g-1$. 
\end{corollary}
\begin{proof}
Since $r_{k}(\mb{M}) \geq g$, there is a morphism $f: \mb{W}^{g}_{k} \longrightarrow \mb{M}$. 
Consider the group homomorphism given by 
$$\varphi\circ f: \mb{W}^{g}_{k} \longrightarrow C_{k}.$$
Since $\mb{W}^{g}_{k}$ is non-singular, there exists $v \in \mb{W}^{g}_{k}$ such that 
$\varphi\circ f(x) = b(v, x)$ for all $x \in \mb{W}^{g}_{k}$.
By Proposition \ref{prop: automorphisms}, there exists an automorphism $h: \mb{W}^{g}_{k} \longrightarrow \mb{W}^{g}_{k}$ such that $h^{-1}(v)$ is in the sub-module $\mb{W}_{k}\oplus\mb{0} \; \leq \; \mb{W}^{g}_{k}$. 
It follows that the submodule $\mb{0}\oplus\mb{W}^{g-1}_{k}$ is contained in the kernel of the homomorphism given by the composition, 
$$\xymatrix{
\mb{W}^{g}_{k} \ar[rr]^{h} && \mb{W}^{g}_{k} \ar[rr]^{f} && \mb{M} \ar[rr]^{\varphi} && C_{k}.
}$$
This implies that $f(h(\mb{0}\oplus\mb{W}^{g-1}_{k}))$ is contained in the kernel of $\varphi$ and thus $r_{k}(\Ker(\varphi)) \geq g - 1$.

Now suppose that $\bar{r}_{k}(\mb{M}) \geq g$ and let  $\varphi: \mb{M} \longrightarrow C_{k}$ be given. 
It follows that $r_{k}(\mb{M}\oplus\mb{W}_{k}^{j}) \geq g$ for some integer $j \geq 0$. 
Consider the map $\bar{\varphi}$ given by the composition,
$$\xymatrix{
\mb{M}\oplus\mb{W}_{k}^{j} \ar[rr]^{\text{proj}_{\mb{M}}} && \mb{M} \ar[rr]^{\varphi} && C_{k}.
}$$
By the result proven in the first paragraph, $r_{k}(\Ker(\bar{\varphi})) \geq g - 1$. 
Clearly we have $\Ker(\bar{\varphi}) = \Ker(\varphi)\oplus\mb{W}^{j}_{k}$. 
It then follows that $\bar{r}_{k}(\Ker(\varphi)) \geq g-1$. 
This completes the proof of the corollary.
\end{proof}
\noindent 

The next proposition yields the first non-trivial case of Theorem \ref{theorem: linking form high-connectivity}. 
Compare with \cite[Proposition 4.3]{GRW 14}.

\begin{proposition} \label{proposition: base case}
If $\bar{r}_{k}(\mb{M}) \geq 2$, then $L(\mb{M})_{k} \neq \emptyset$. 
If $\bar{r}_{k}(\mb{M}) \geq 4$, then $L(\mb{M})_{k}$ is connected.
\end{proposition}
\begin{proof}
Let us first make the slightly stronger assumption that $r_{k}(\mb{M}) \geq 4$. 
It follows that there exists some morphism $f_{0}: \mb{W}_{k} \longrightarrow \mb{M}$ such that $r_{k}(f_{0}(\mb{W}_{k})^{\perp}) \geq 3$. 
Given any morphism $f: \mb{W}_{k} \longrightarrow \mb{M}$, we have a homomorphism of abelian groups 
$f_{0}(\mb{W}_{k})^{\perp} \longrightarrow \mb{M} \longrightarrow f(\mb{W}_{k}),$
where the first map is the inclusion and the second is orthogonal projection.
The kernel of this map is the intersection $f_{0}(\mb{W}_{k})^{\perp}\cap f(\mb{W}_{k})^{\perp}$. 
Since $\mb{W}_{k} = \Z/k\oplus\Z/k \cong C_{k}\oplus C_{k}$ (as an abelian group), it follows from Corollary \ref{cor: rank reduction} that $r_{k}(f_{0}(\mb{W}_{k})^{\perp}\cap f(\mb{W}_{k})^{\perp}) \geq 1$.
Thus, we can find a morphism 
$$f': \mb{W}_{k} \longrightarrow f_{0}(\mb{W}_{k})^{\perp}\cap f(\mb{W}_{k})^{\perp}.$$ 
It follows that the sets $\{f_{0}, f\}$ and $\{f_{0}, f'\}$ are both $1$-simplices, and so there is a path of length $2$ from $f$ to $f'$. 

Now suppose that $\bar{r}_{k}(\mb{M}) \geq 4$. 
We then have an isomorphism of linking forms $\mb{M}\oplus\mb{W}^{j}_{k} \cong \mb{N}\oplus\mb{W}_{k}^{j}$ for some $j$ where $r_{k}(\mb{N}) \geq 4$.
By the first paragraph, $L(\mb{N}\oplus\mb{W}^{j}_{k})_{k}$  is connected for all $j \geq 0$, and so we may apply Proposition \ref{proposition: cancelation alg} inductively to deduce that $\mb{M} \cong \mb{N}$ and thus $r_{k}(\mb{M}) \geq 4$. 
We then apply the result of the first paragraph to conclude that $L(\mb{M})_{k}$ is connected.

If $\bar{r}_{k}(\mb{M}) \geq 2$ we may write $\mb{M}\oplus\mb{W}_{k}^{j} \cong \mb{N}\oplus\mb{W}_{k}^{j}$ for some integer $j$ and linking form $\mb{N}$ such that $r_{k}(\mb{N}) \geq 2$. 
We may then inductively apply Proposition \ref{proposition: cancelation alg} to obtain an isomorphism $f: \mb{M}\oplus\mb{W}_{k} \stackrel{\cong} \longrightarrow \mb{N}\oplus\mb{W}_{k}$. 
The linking form $\mb{M}$ is then isomorphic to the kernel of the orthogonal projection, $\mb{N}\oplus\mb{W}_{k} \longrightarrow f(\mb{0}\oplus\mb{W}_{k})$. 
Since $r_{k}(\mb{N}\oplus\mb{W}_{k}) \geq 3$ and $\mb{W}_{k} \cong C_{k}\oplus C_{k}$, it follows from Corollary \ref{cor: rank reduction} that 
$r_{k}(\mb{M}) \geq 1$. 
From this, it follows that $L(\mb{M})_{k}$ is non-empty. 
This concludes the proof of the proposition. 
\end{proof}

\begin{proof}[Proof of Theorem \ref{theorem: linking form high-connectivity}]
We proceed by induction on $g$. 
The base case of the induction, which is the case of the theorem where $g = 4$ and $\bar{r}(\mb{M}) \geq 4$, follows immediately from Proposition \ref{proposition: base case}.
 Now suppose that the theorem holds for the $g-1$ case. 
 Let $\mb{M}$ be a linking form with $\bar{r}_{k}(\mb{M}) \geq g$ and $g \geq 4$. 
 By Proposition \ref{proposition: base case} there exists a morphism $f: \mb{W}_{k} \longrightarrow \mb{M}$ and by Corollary \ref{corollary: stable rank reduction} it follows that 
$\bar{r}_{k}(f(\mb{W}_{k})^{\perp}) \geq g - 1$. 
Let $\mb{M}'$ denote the orthogonal complement $f(\mb{W}_{k})^{\perp}$ and consider the subgroup $\mb{M}'\perp \langle f(\sigma)\rangle \leq \mb{M}$, where $\sigma$ is one of the standard generators of $\mb{W}_{k}$ ($\mb{M}'\perp \langle f(\sigma)\rangle$ indicates an orthogonal direct sum).
The chain of inclusions 
$\mb{M}' \hookrightarrow \mb{M}'\perp\langle f(\sigma)\rangle \hookrightarrow \mb{M}$ 
induces a chain of embeddings of sub-simplicial-complexes
\begin{equation}
\xymatrix{
L(\mb{M}')_{k} \ar[rr]^{i_{1}} && L(\mb{M}'\perp \langle f(\sigma)\rangle)_{k} \ar[rr]^{i_{2}} && L(\mb{M})_{k}.
}
\end{equation}
The composition is null-homotopic since the vertex in $L(\mb{M})_{k}$ determined by the morphism $f: \mb{W}_{k} \longrightarrow \mb{M}$, is adjacent to every simplex in the subcomplex $L(\mb{M}')_{k} \leq L(\mb{M})_{k}$. 
To prove that $L(\mb{M})_{k}$ is $\frac{1}{2}(g - 4)$-connected, we apply Proposition \ref{proposition: inclusion complex} to the maps $i_{1}$ and $i_{2}$ with $n := \frac{1}{2}(g - 4)$. 
Since $L(\mb{M}')_{k}$ is $(n - 1)$-connected by the induction assumption (recall that $\bar{r}(\mb{M}') \geq g-1$), Proposition \ref{proposition: inclusion complex} together with the fact that $i_{2}\circ i_{1}$ is null-homotopic will imply that $L(\mb{M})_{k}$ is $\frac{1}{2}(g - 4)$-connected. 

Let $\xi$ be a $p$-simplex of $L(\mb{M}'\perp \langle f(\sigma)\rangle)_{k}$. 
The linking form on the subgroup $f(\sigma) \leq  \mb{M}'$ is trivial and thus it follows that the projection homomorphism, $\pi: \mb{M}'\perp \langle f(\sigma)\rangle \longrightarrow \mb{M}'$ preserves the linking form structure.
Thus, there is an induced simplicial map 
$$\bar{\pi}: L(\mb{M}'\perp \langle f(\sigma)\rangle)_{k} \longrightarrow L(\mb{M}')_{k},$$ 
and it follows easily that $i_{1}$ is a section of $\bar{\pi}$.
It follows from (\ref{equation: link of simplex}) that there is an equality of simplicial complexes,
$$\xymatrix{
[\lk_{L(\mb{M}'\perp \langle f(\sigma)\rangle)_{k}}(\xi)]\cap L(\mb{M}')_{k} \; = \; \lk_{L(\mb{M}')_{k}}(\bar{\pi}(\xi)).
}$$
Since $\bar{r}_{k}(\mb{M}') \geq g - 1$, the induction assumption (which is that $lCM(L(\mb{M}')_{k}) \geq \frac{1}{2}(g - 2)$) implies that the above complex is 
$$\tfrac{1}{2}(g-2) - p - 2 = (n - p - 1)-\text{connected},$$
where recall, $n = \frac{1}{2}(g - 4)$. 
Proposition \ref{proposition: inclusion complex} then implies that the map $i_{1}$ is $n$-connected.

We now focus on the map $i_{2}$.
Since $b(\sigma, \sigma) = 0$, it follows that the subgroup
$$\mb{M}'\perp \langle f(\sigma)\rangle \leq \mb{M}$$ 
is precisely the orthogonal complement of $\langle f(\sigma) \rangle$ in $\mb{M}$.
Let $\zeta := \{f_{0}, \dots, f_{p}\} \leq L(\mb{M})_{k}$ be a  $p$-simplex, and 
denote $\mb{M}'' := [\sum(f_{i}(\mb{W}_{k}))]^{\perp} \leq \mb{M}$. 
We have, 
\begin{equation} \label{link intersection 2}
L(\mb{M}'\perp \langle f(\sigma)\rangle)_{k}\cap\text{lk}_{L(\mb{M})}(\zeta) \; = \; L(\mb{M}''\cap \langle f(\sigma)\rangle^{\perp})_{k}.
\end{equation}
Corollary \ref{corollary: stable rank reduction} implies that $\bar{r}_{k}(\mb{M}'') \geq g - p - 1$. 
Passing to the kernel of the homomorphism
$$b(\underline{\hspace{.3cm}}, f(\sigma))|_{\mb{M}''}: \mb{M}'' \longrightarrow C_{k},$$
reduces the stable $k$-rank by $1$, and so we have $\bar{r}_{k}(\mb{M}''\cap \langle f(\sigma)\rangle^{\perp}) \geq g - p - 2$. 
By the induction assumption, it follows that the complex $L(\mb{M}''\cap \langle f(\sigma)\rangle^{\perp})_{k}$
is at least 
$$\tfrac{1}{2}(g - p - 2 - 4) \geq (n - p - 1)-\text{connected}.$$ 
By Proposition \ref{proposition: inclusion complex} it follows that the inclusion $i_{2}$ is $n$-connected. 
Combining with the previous paragraph implies that $i_{2}\circ i_{1}$ is $n$-connected. 
It then follows that $L(\mb{M})_{k}$ is $n = \tfrac{1}{2}(g - 4)$-connected since $i_{2}\circ i_{1}$ is null-homotopic.  

The fact that $lCM(L(\mb{M})_{k}) \geq \frac{1}{2}(g - 1)$ is proven as follows. 
Let $\zeta = \{f_{0}, \dots, f_{p}\} \leq L(\mb{M})_{k}$ be a $p$-simplex 
and let $\mb{V}$ denote the orthogonal complement $[\sum f_{i}(\mb{W})]^{\perp}$.
We have $\bar{r}_{k}(\mb{V}) \geq g - p - 1$. 
By (\ref{equation: link of simplex}) we have $\lk_{L(\mb{M})_{k}}(\zeta) \cong L(\mb{V})_{k}$ and so by the induction assumption it follows that $|\lk_{L(\mb{M})_{k}}(\zeta)|$ is $\frac{1}{2}(g - p - 1 - 4)$-connected. 
The inequality 
$$\tfrac{1}{2}(g - p - 1 - 4) \; = \; \tfrac{1}{2}(g - p - 1) - 2 \; \geq \; \tfrac{1}{2}(g - 1) - p - 2$$
implies that $|\lk_{L(\mb{M})_{k}}(\zeta)|$ is $(\frac{1}{2}(g - 1) - p - 2)$-connected. 
This proves that $lCM(L(\mb{M})_{k}) \geq \frac{1}{2}(g - 1)$ and concludes the proof of the Theorem.
\end{proof}

\section{$(2n-1)$-Connected, $(4n+1)$-Dimensional Manifolds} \label{section: Highly Connected Manifolds of Odd Dimension}
\subsection{The Homological Linking Form}
For what follows, let $M$ be an oriented manifold of dimension $2s+1$.
Let $H^{\tau}_{s}(M; \Z) \leq H_{s}(M; \Z)$ denote the torsion subgroup of $H_{s}(M; \Z)$.
Following \cite{W 67}, the \textit{homological linking form} 
$\tilde{b}: H^{\tau}_{s}(M; \Z)\otimes H^{\tau}_{s}(M; \Z) \longrightarrow \Q/\Z$ 
is defined as follows. 
Let $x, y \in H^{\tau}_{s}(M; \Z)$ and suppose that $x$ has order $r > 1$. Represent $x$ by a chain $\xi$ and let $\partial\zeta = r\cdot\xi$. Then if $y$ is represented by the chain $\chi$, we define
\begin{equation} \label{eq: linking form chain level} 
\tilde{b}(x, y) = \tfrac{1}{r}[\zeta\cap\chi] \quad  \text{mod $ 1$,}
\end{equation}
where $\zeta\cap\chi$ denotes the algebraic intersection number associated to the two chains (after being deformed so as to meet transversally). 
It is proven in \cite[Page 274]{W 67} that $\tilde{b}$ is $(-1)^{s+1}$-symmetric. 
We refer the reader to \cite{W 67} for further details on this construction. 

Let $\pi_{s}^{\tau}(M) \leq \pi_{s}(M)$ denote the torsion component of the homotopy group $\pi_{s}(M)$. 
Using the homological linking form and the Hurewicz homomorphism $h: \pi_{s}(M) \longrightarrow H_{s}(M)$, we can define a similar bilinear pairing 
\begin{equation} \label{eq: homotopical linking form}
b: \pi^{\tau}_{s}(M)\otimes\pi^{\tau}_{s}(M) \longrightarrow \Q/\Z; \quad b(x, y) = \tilde{b}( h(x), \; h(y) ).
\end{equation}
The pair $(\pi^{\tau}_{s}(M), b)$ is a $(-1)^{s+1}$-symmetric linking form in the sense of Section \ref{subsection: linking forms} and we will refer to it as the \textit{homotopical linking form} associated to $M$.  In the case that $M$ is $(s-1)$-connected, the homotopical linking form is isomorphic to the homological linking form by the \textit{Hurewicz theorem}.

\subsection{The classification theorem.} \label{subsection: the classification theorem}
We are mainly interested in manifolds which are $(4n+1)$-dimensional with $n \geq 2$. 
In this case the homological (and homotopical) linking form is skew symmetric.
It follows from this that $b(x, x) = 0$ whenever $x$ is an element of odd order. 
The following lemma of Wall from \cite{W 67} implies that for $(2n-1)$-connected, $(4n+1)$-dimensional manifolds with $n \geq 2$, the linking form is \textit{strictly} skew symmetric. 
\begin{lemma} \label{lemma: strict anti-symmetry} For $n \geq 2$, let $M$ be a $(2n-1)$-connected, $(4n+1)$-dimensional manifold. Then $b(x, x) = 0$ for all $x \in \pi^{\tau}_{2n}(M)$.  \end{lemma}

It follows from Lemma \ref{lemma: strict anti-symmetry} that if $M$ is a $(2n-1)$-connected, $(4n+1)$-dimensional manifold then the homotopical linking form
$(\pi^{\tau}_{2n}(M), b)$ 
is an object of the category $\mathcal{L}_{-1}^{s}$.
If $M$ is closed (or has boundary a homotopy sphere), then $(\pi^{\tau}_{2n}(M), b)$ is non-singular. 
The following theorem is a specialization of Wall's classification theorem \cite[Theorem 7]{W 67} to the set $\widehat{\mathcal{W}}_{2n}^{f}$.
\begin{theorem} \label{thm: classification theorem}
For $n \geq 2$, two manifolds $M_{1}, M_{2} \in \widehat{\mathcal{W}}_{4n+1}^{f}$ are almost diffeomorphic if and only if
there is an isomorphism of linking forms, 
$(\pi_{2n}(M_{1}), \;  b) \cong (\pi_{2n}(M_{2}), \;  b).$
Furthermore, given any non-singular object $\mathbf{M} \in \Ob(\mathcal{L}^{s}_{-1})$ whose underlying finite abelian group has trivial $2$-torsion, there exists an element $M \in \widehat{\mathcal{W}}_{4n+1}^{f}$ such that $(\pi_{2n}(M), \;  b) \cong (\mathbf{M}, \; b_{\mb{M}})$. 
\end{theorem}

\begin{remark}
As discussed in the introduction, $\widehat{\mathcal{W}}_{2n}^{f}$ is the class of manifolds for which the linking form is a complete invariant (up to almost diffeomorphism). 
Indeed, for all stably parallelizable, $(2n-1)$-connected, $(4n+1)$-dimensional, closed manifolds $W$, the tangential invariants $\alpha$ and $\widehat{\beta}$ considered by Wall in \cite[Theorem 7]{W 67} vanish. 
Furthermore, if $H_{2n}(W; \Z)$ is finite and has trivial $2$-torsion then the invariant $\widehat{\phi} \in H^{2n+1}(W; \Z/2)$ must also vanish as well and the only remaining invariant is the linking form. 
Thus the classification of manifolds in $\widehat{\mathcal{W}}_{2n}^{f}$ reduces to the classification of strictly skew symmetric linking forms and we obtain Theorem \ref{thm: classification theorem}. 
\end{remark}

Using the above classification theorem and the classification of strictly skew symmetric linking forms from Theorem \ref{thm: classification of linking modules}, we may specify certain basic manifolds. 
For each odd integer $k > 2$, fix a manifold $W_{k} \in \widehat{\mathcal{W}}^{f}_{4n+1}$ whose linking form $(\pi_{2n}(W_{k}), \;  b)$ is isomorphic to $\mb{W}_{k}$.
By the above theorem, this condition determines $W_{k}$ uniquely up to almost diffeomorphism.
It follows from Theorem \ref{thm: classification theorem} that every element of $\widehat{\mathcal{W}}^{f}_{4n+1}$ is almost diffeomorphic to a manifold of the form $W_{k_{1}}\#\cdots\# W_{k_{s}}$.
The manifolds $W_{k}$ are the subject of our main result of this paper, Theorem \ref{theorem: Main theorem}. 

\begin{remark} \label{remark: boundary sphere classification}
The manifolds $W_{k} \in \widehat{\mathcal{W}}^{f}_{4n+1}$ are uniquely determined by their linking forms up to \textit{almost diffeomorphism}. 
For each $k$, let $W'_{k}$ denote the manifold obtained from $W_{k}$ by removing an open disk. 
It follows from \cite[Theorem 7]{W 67} that $W'_{k}$ is determined by its linking form up to diffeomorphism. 
\end{remark}

\section{$\Z/k$-Manifolds} \label{Bordism Groups of Singular Manifolds} 
\subsection{Basic Definitions}
One of the main tools we will use to study the diffeomorphism groups of odd dimensional manifolds will be manifolds with certain types of \textit{Baas-Sullivan} singularities, namely $\Z/k$-manifolds (which in this paper we refer to as $\langle k \rangle$-manifolds). 
We will use these manifolds to construct a geometric model for the linking form. Here we give an overview of the definition and basic properties of such manifolds. 
For further reference on $\Z/k$-manifolds or manifolds with general Baas-Sullivan singularities, see \cite{Ba 73}, \cite{B 92}, and \cite{SM 74}.
\begin{Notation}
For a positive integer $k$, we let $\langle k \rangle$ denote the set consisting of $k$-elements, $\{1, \dots, k\}$. 
We will consider this set to be a zero-dimensional manifold. 
\end{Notation}

\begin{defn} \label{defn: k-manifolds} Let $k$ be a positive integer. Let $P$ be a $p$-dimensional smooth manifold equipped with the following extra structure:
\begin{enumerate} 
\item[(i)] The boundary of $P$ has the decomposition,
$\partial P = \partial_{0}P \cup \partial_{1}P$
where $\partial_{0}P$ and $\partial_{1}P$ are $(p-1)$-dimensional manifolds with boundary and 
$$\partial_{0, 1}P := (\partial_{0}P)\cap(\partial_{1}P) = \partial(\partial_{0}P) = \partial(\partial_{1}P)$$ 
is a $(p-2)$-dimensional manifold without boundary.
\item[(ii)] There is a manifold $\beta P$ and diffeomorphism
$\Phi: \partial_{1}P \stackrel{\cong} \longrightarrow \beta P\times\langle k \rangle.$
\end{enumerate}
With $P$, $\beta P,$ and $\Phi$ as above, the pair $(P, \Phi)$ is said to be a smooth $\langle k \rangle$-manifold. 
The diffeomorphism $\Phi$ is referred to as the \textit{structure-map} and the manifold $\beta P$ is called the \textit{Bockstein}.
\end{defn}

\begin{Notation} We will usually drop the structure-map from the notation and denote $P := (P, \Phi)$. 
We will always denote the structure-map associated to a $\langle k \rangle$-manifold by the same capital greek letter $\Phi$. 
If another $\langle k \rangle$-manifold is present, say $Q$, we will decorate the structure map with the subscript $Q$, i.e.\ $\Phi_{Q}$.

Any smooth manifold $M$ is automatically a $\langle k \rangle$-manifold by setting $\partial_{0}M = \partial M$, $\partial_{1}M = \emptyset$, and $\beta M = \emptyset$. 
Such a $\langle k \rangle$-manifold $M$ with $\partial_{1}M = \emptyset$, $\beta M = \emptyset$ is said to be \textit{non-singular}. 
\end{Notation}

Now, let $P$ be a $\langle k \rangle$-manifold as in the above definition. Notice that the diffeomorphism
$\Phi$ maps the submanifold $\partial_{0,1}P \subset \partial_{1}P$ diffeomorphically onto $\partial(\beta P)\times\langle k \rangle$. 
In this way, if we set 
$$\partial_{0}(\partial_{0}P) := \emptyset, \quad  \partial_{1}(\partial_{0}P) := (\partial_{0}P)\cap(\partial_{1}P) = \partial_{0,1}P, \quad \text{and} \quad \beta(\partial_{0}P) = \partial(\beta P),$$ 
the pair
$\partial_{0}P := (\partial_{0}P, \; \Phi|_{\partial_{0,1}P})$
is a $\langle k \rangle$-manifold. We will refer to $\partial_{0}P$ as the 
 \textit{boundary} of $P$. If $P$ is compact and $\partial_{0}P = \emptyset$, then $P$ is said to be a \textit{closed} $\langle k \rangle$-manifold.  

Given a $\langle k \rangle$-manifold $P$, one can construct a manifold with \textit{cone-type singularities} in a natural way as follows.
\begin{defn} \label{defn: associated singular space}  Let $P$ be a $\langle k \rangle$- manifold.
Let $\bar{\Phi}: \partial_{1}P \longrightarrow \beta P$ be the map given by the composition
$\xymatrix{
\partial_{1}P \ar[rr]^{\Phi}_{\cong} && \beta P\times\langle k \rangle \ar[rr]^{\text{proj}_{\beta P}} && \beta P.
}$
We define $\widehat{P}$ to be the quotient space obtained from $P$ by identifying points $x, y \in \partial_{1}P$ if and only if 
$\bar{\Phi}(x) = \bar{\Phi}(y).$
\end{defn}

We will need to consider maps from $\langle k \rangle$-manifolds to non-singular manifolds. 
\begin{defn} \label{defn: k-map} Let $P$ be a $\langle k \rangle$-manifold and let $X$ be a topological space. 
A map $f: P \longrightarrow X$ is said to be a $\langle k \rangle$-map if there exists a map $f_{\beta}: \beta P \rightarrow X$ 
such that the restriction of $f$ to $\partial_{1}P$ has the factorization
$
\xymatrix{
\partial_{1}P \ar[rr]^{\bar{\Phi}} && \beta P \ar[rr]^{f_{\beta}} && X,}
$
where $\bar{\Phi}: \partial_{1}P \longrightarrow \beta P$ is the map from Definition \ref{defn: associated singular space}.
Clearly the map $f_{\beta}$ is uniquely determined by $f$. 
\end{defn}

We denote by $\Maps_{\langle k \rangle}(P, X)$ the space of $\langle k\rangle$-maps $P \rightarrow M$, topologized as a subspace of $\Maps(P, X)$ with the compact-open topology.
 It is immediate that any $\langle k \rangle$-map $f: P \rightarrow X$ induces a unique map $\widehat{f}: \widehat{P} \longrightarrow X$ and that the correspondence, $f \mapsto \widehat{f}$ induces a homeomorphism, $\Maps_{\langle k \rangle}(P, X) \cong \Maps(\widehat{P}, X)$. 
 Throughout the paper we will denote by $\hat{f}: \widehat{P} \longrightarrow Y$, the map induced by the $\langle k \rangle$-map $f$. 
In the case that $X$ is a smooth manifold, $f$ is said to be a smooth $\langle k \rangle$-map if both $f$ and $f_{\beta}$ are smooth.

\subsection{Bordism of $\langle k \rangle$-manifolds.} \label{subsection: framed bordism}
We will need to consider the oriented bordism groups of $\langle k \rangle$-manifolds. 
For a space $X$ and non-negative integer $j$, 
we denote by $\Omega^{SO}_{j}(X)_{\langle k \rangle}$ the bordism group of $j$-dimensional, closed, oriented, smooth $\langle k \rangle$-manifolds associated to $X$. 
We refer the reader to 
 \cite{B 92} and \cite{SM 74} for precise details of the definitions. 
We have the following theorem from \cite{B 92}.
\begin{theorem} \label{thm: bockstein-sullivan}
For any space $X$ and integer $k \geq 2$, there is a long exact sequence:
\begin{equation}
\xymatrix{
\cdots \ar[r] & \Omega^{SO}_{j}(X) \ar[rr]^{\times k} && \Omega^{SO}_{j}(X)  \ar[rr]^{j_{k}} && \Omega^{SO}_{j}(X)_{\langle k\rangle}\ar[rr]^{\beta} && \Omega^{SO}_{j-1}(X) \ar[r] & \cdots }
\end{equation}
where $\times k$ denotes multiplication by the integer $k$, $j_{k}$ is induced by inclusion (since an oriented smooth manifold is an oriented $\langle k \rangle$-manifold), and $\beta$ is the map induced by $P \mapsto \beta P$. 
\end{theorem}
It is immediate from the above long exact sequence that for all integers $k \geq 2$, there are isomorphisms
\begin{equation} \label{equation: elementary bordism calculation}
\Omega^{SO}_{0}(\text{pt.})_{\langle k\rangle} \cong \Z/k \quad \text{and} \quad \Omega^{SO}_{1}(\text{pt.})_{\langle k\rangle} \cong 0.
\end{equation}

%%%%%%%%%%%%%%%%%%%%%%%%%%%%%%%%%%%%%%%%%%%%%%%%%%%%%%%%%%%%%%%%%%%%%%%%%%%%%%%%%%%%%%%%%%%%%%%%%%%%%%%%%%%%%%%%%%%%%%%%%%%%%%%%%%%%%%%%%%%%%%%%%%%%%%%%%%%%%%%%%%%%%%%%%%%%%%%%%%%%%%%%%%%%%%%%%%%%%%%%%%%%%%%%%%%%%%%%%%%%%%%%%%%%%%%%%%

\subsection{$\Z/k$-homotopy groups.} For integers $k, n \geq 2$, let $M(\Z/k, n)$ denote the $n$-th $\Z/k$-Moore-space. 
Recall that $M(\Z/k, n)$ is uniquely determined up to homotopy by the calculation, 
$$H_{j}(M(\Z/k, n);\; \Z) \cong 
\begin{cases}
\Z/k &\quad \text{if $j = n$,}\\
\Z & \quad \text{if $j = 0$,}\\
0 &\quad \text{else.}
\end{cases}
$$ 
For a space $X$, we denote by $\pi_{n}(X; \; \Z/k)$ the set of based homotopy classes of maps $M(\Z/k, n) \longrightarrow X$. 
Since $M(\Z/k, n)$ is a suspension when $n \geq 2$, the set $\pi_{n}(X; \; \Z/k)$ has the structure of a group, which is abelian when $n \geq 3$. 

For integers $n, k \geq 2$, we define a $\langle k \rangle$-manifold which will play the role of the sphere in the category of $\langle k \rangle$-manifolds.  
\begin{Construction}
Choose an embedding $\Phi': D^{n}\times\langle k \rangle \longrightarrow S^{n}$. 
Let $V^{n}_{k}$ denote the manifold obtained from $S^{n}$ by removing the interior of $\Phi'(D^{n}\times\langle k \rangle)$ from $S^{n}$. 
The inverse of the restriction of the map $\Phi'$ to $\partial D^{n}\times\langle k \rangle$ induces a diffeomorphism,
$\Phi: \partial V^{n}_{k} \stackrel{\cong} \longrightarrow S^{n-1}\times\langle k \rangle.$
By setting $\beta V^{n}_{k} = S^{n-1}$, the above diffeomorphism $\Phi$ gives $V^{n}_{k}$ the structure of a closed $\langle k \rangle$-manifold. 
\end{Construction}
Let $\widehat{V}_{k}^{n}$ denote the singular space obtained from $V^{n}_{k}$ as in Definition \ref{defn: associated singular space}. 
An elementary calculation shows that,
\begin{equation}
H_{j}(\widehat{V}_{k}^{n}) \cong 
\begin{cases}
\Z/k &\quad \text{if $j = n-1$ or $0$,}\\
\Z^{\oplus(k-1)} &\quad \text{if $j = 1$,}
\end{cases}
\quad \quad \text{and} \quad \quad \pi_{1}(\widehat{V}_{k}^{n}) \cong \Z^{\star(k-1)},
\end{equation}
where $\Z^{\star(k-1)}$ denotes the free group on $(k-1)$-generators.
It follows that the Moore-space $M(\Z/k, n-1)$ can be constructed from $\widehat{V}_{k}^{n}$ by attaching a collection of $k-1$ $2$-cells, one for each generator of the fundamental group. 
This yields the following result. 
\begin{lemma} \label{lemma: Z/k bijection}
Let $X$ be a $2$-connected space and let $k \geq 2$ and $n \geq 3$ be integers. 
The inclusion map $\widehat{V}_{k}^{n} \hookrightarrow  M(\Z/k, n-1)$ induces a bijection of sets,
$\pi_{0}(\Maps_{\langle k \rangle}(V^{n}_{k}, X)) \stackrel{\cong} \longrightarrow \pi_{n-1}(X; \Z/k).$
\end{lemma}
\begin{proof}
Since $X$ is simply connected, any map $\widehat{V}_{k}^{n} \longrightarrow X$ extends to a map $M(\Z/k, n-1) \longrightarrow X$ and since $X$ is $2$-connected, it follows that any such extension is unique up to homotopy. 
This proves that the inclusion $\widehat{V}_{k}^{n} \hookrightarrow  M(\Z/k, n-1)$ induces a bijection $\pi_{0}(\Maps(\widehat{V}_{k}^{n}, X)) \cong \pi_{n-1}(X; \Z/k)$. 
The lemma then follows from composing this bijection with the natural bijection, $\pi_{0}(\Maps_{\langle k \rangle}(V^{n}_{k}, X)) \stackrel{\cong} \longrightarrow \pi_{0}(\Maps(\widehat{V}_{k}^{n}, X))$. 
\end{proof}

\begin{corollary} \label{corollary: represent by k-map}
Let $X$ be a $2$-connected space and let $k \geq 2$ and $n \geq 3$ be integers. 
Let $x \in \pi_{n-1}(X)$ be an element of order $k$. 
Then there exists a $\langle k \rangle$-map $f: V^{n}_{k} \longrightarrow X$ such that the associated map $f_{\beta}: S^{n-1} \longrightarrow X$ is a representative of $x$. 
\end{corollary}
\begin{proof}
The cofibre sequence 
$\xymatrix{
S^{j} \ar[r]^{\times k} & S^{j} \ar[r] & M(\Z/k, j)
}$
induces a long exact sequence,
$$\xymatrix{
\cdots \ar[r] & \pi_{n}(X) \ar[r]^{\times k} & \pi_{n}(X) \ar[r] & \pi_{n-1}(X; \Z/k) \ar[r]^{\partial} & \pi_{n-1}(X) \ar[r]^{\times k}  & \pi_{n-1}(X) \ar[r] & \cdots
}$$
It follows that if $x \in \pi_{n-1}(X)$ is of order $k$, then there is an element $y \in \pi_{n-1}(X; \Z/k)$ such that $\partial y = x$. 
Let $r_{\beta}: \pi_{0}(\Maps_{\langle k \rangle}(V^{n}_{k}, X)) \longrightarrow \pi_{n-1}(X)$ denote the map induced by, $f \mapsto f_{\beta}$. 
It follows from the construction of the the map $\partial$ in the above long exact sequence that the diagram,
$$\xymatrix{
\pi_{0}(\Maps_{\langle k \rangle}(V^{n}_{k}, X)) \ar[drr]^{r_{\beta}} \ar[rr]^{\cong} && \pi_{n-1}(X; \Z/k) \ar[d]^{\partial} \\
&& \pi_{n-1}(X)
}$$
commutes, where the upper horizontal map is the bijection from Lemma \ref{lemma: Z/k bijection}. 
The result then follows from commutativity of this diagram. 
\end{proof}

\subsection{Immersions and embeddings of $\langle k \rangle$-manifolds} \label{subsection: k-immersions}
We will need to consider smooth immersions and embeddings of a smooth $\langle k\rangle$-manifold into a smooth manifold. 
For what follows, let $P$ be a $\langle k \rangle$-manifold and let $M$ be a manifold. 

\begin{defn} \label{defn: k-immersion} 
A $\langle k\rangle$-map $f: P \longrightarrow M$ is said to be a \textit{$\langle k\rangle$-immersion} if it is a smooth immersion when considering $P$ as a smooth manifold with boundary.
Two $\langle k\rangle$-immersions $f, g: P \longrightarrow M$ are said to be regularly homotopic if there exists a smooth homotopy $F_{t}: P \longrightarrow M$ with $F_{0} = f$ and $F_{1} = g$ such that $F_{t}$ is a $\langle k \rangle$-immersion for all $t \in [0, 1]$.
\end{defn}

In addition to immersions we will mainly need to deal with embeddings of $\langle k \rangle$-manifolds. 
\begin{defn}
A $\langle k \rangle$-immersion $f: P \longrightarrow M$ is said to be a $\langle k \rangle$-embedding if the induced map $\widehat{f}: \widehat{P} \longrightarrow M$ is an embedding. 
\end{defn}

The main result about $\langle k \rangle$-embeddings that we will use is the following. 
The proof is given in Section \ref{subsection: immersion to embedding}, using the techniques developed throughout all of Section \ref{section: k-immersions}. 
\begin{theorem} \label{theorem: immersion to embedding} 
Let $n \geq 2$ be an integer and let $k > 2$ be an odd integer.
Let $M$ be a $2$-connected, oriented manifold of dimension $4n+1$.
Then any $\langle k \rangle$-map $f: V^{2n+1}_{k} \longrightarrow M$ is homotopic through $\langle k \rangle$-maps to a $\langle k \rangle$-embedding. 
\end{theorem}

The following corollary follows immediately by combining Theorem \ref{theorem: immersion to embedding} with Corollary \ref{corollary: represent by k-map}. 

\begin{corollary} \label{lemma: represent by $k$-embedding} \label{lemma: represent by embedding}
Let $n \geq 2$ be an integer and let $k > 2$ be an odd integer.
Let $M$ be a $2$-connected, oriented manifold of dimension $4n+1$.
Let $x \in \pi_{2n}(M)$ be a class of order $k$.
Then there exists a $\langle k \rangle$-embedding $f: V^{2n+1}_{k} \longrightarrow M$ such that the embedding $f_{\beta}: \beta V^{2n+1}_{k} = S^{2n} \longrightarrow M$ is a representative of the class $x$. 
\end{corollary}

%%%%%%%%%%%%%%%%%%%%%%%%%%%%%%%%%%%%%%%%%%%%%%%%%%%%%%%%%%%%%%%%%%%%%%%%%%%%%%%%%%%%%%%%%%%%%%%%%%%%%%%%%%%%%%%%%%%%%%%%%%%%%%%%%%%%%%%%%%%%%%%%%%%%%%%%%%%%%%%%%%%%%%%%%%%%%%%%%%%%%%%%%%%%%%%%%%%%%%%%%%%%%%%%%%%%%%%%%%%%%%%%%%%%%%%%%%

\section{$\langle k, l\rangle$-Manifolds} \label{section: k-l manifolds}
 We will have to consider certain spaces with more complicated singularity structure than that of the $\langle k\rangle$-manifolds encountered in the previous section. 
\begin{defn} \label{defn: k-l manifold} Let $k$ and $l$ be positive integers. Let $N$ be a smooth $d$-dimensional manifold equipped with the following extra structure:
\begin{enumerate} \itemsep.2cm
\item[(i)] The boundary $\partial N$ has the decomposition, 
$$\partial N \; = \partial_{0}N\cup\partial_{1}N\cup\partial_{2}N$$
such that $\partial_{0}N$, $\partial_{1}N$ and $\partial_{2}N$ are $(d-1)$-dimensional smooth manifolds, the intersections
$$\partial_{0, 1}N := \partial_{0}N\cap\partial_{1}N, \quad \partial_{0, 2}N := \partial_{0}N\cap\partial_{2}N, \quad \partial_{1, 2}N := \partial_{1}N\cap\partial_{2}N$$
are $(d-2)$-dimensional smooth manifolds, and 
$$\partial_{0, 1, 2}N := \partial_{0}N\cap\partial_{1}N\cap \partial_{2}N$$
is a $(d-3)$-dimensional smooth manifold with empty boundary.
\item[(ii)] There exist smooth manifolds $\beta_{1}N,$ $\beta_{2}N,$ and $\beta_{1, 2}N$, and diffeomorphisms 
$$\xymatrix@C-.10pc@R-1.5pc{
\Phi_{1}: \partial_{1}N \ar[r]^{\cong} & \beta_{1}N\times\langle k \rangle, \\
\Phi_{2}: \partial_{2}N \ar[r]^{\cong} & \beta_{2}N\times\langle l \rangle, \\
\Phi_{1,2}: \partial_{1,2}N \ar[r]^{\cong \ \ \ \ } & \beta_{1, 2}N\times\langle k \rangle\times\langle l \rangle,
}$$
such that the maps
$$\xymatrix@C-.10pc@R-1.5pc{
\Phi_{1}\circ\Phi_{1,2}^{-1}: \beta_{1, 2}N\!\times\langle k \rangle\!\times\langle l \rangle \ar[r] & \beta_{1}N\!\times\!\langle k \rangle, \\
\Phi_{2}\circ\Phi_{1,2}^{-1}: \beta_{1, 2}N\!\times\langle k \rangle\!\times\langle l \rangle \ar[r] & \beta_{1}N\!\times\!\langle l \rangle,
}$$
are identical on the direct factors of $\langle k\rangle$ and $\langle l\rangle$ respectively.
\end{enumerate}
With the above conditions satisfied, the $4$-tuple $N := (N, \Phi_{1}, \Phi_{2}, \Phi_{1,2})$ is said to be a \text{$\langle k, l\rangle$-manifold} of dimension $d$. 
\end{defn}
\begin{remark}
For $d \geq 2$, a $d$-dimensional $\langle k, l \rangle$-manifold is technically a \textit{manifold with corners} (see \cite{L 00}).
In our applications however we will only have to consider $\langle k, l \rangle$-manifolds of dimension $1$ and hence no actual corners will ever be encountered. 
If $N$ is a $1$-dimensional $\langle k, l \rangle$-manifold then the faces $\partial_{0}N, \partial_{1}N, \partial_{2}N$ are $0$-dimensional manifolds. 
It follows automatically that $\partial_{0,1}N = \partial_{0,2}N = \partial_{1,2}N = \emptyset$.
Thus, a $1$-dimensional $\langle k, l\rangle$-manifold $N$ is simply a manifold with boundary where the boundary is decomposed as the disjoint union of three distinct closed manifolds $\partial_{0}N$, $\partial_{1}N$, and $\partial_{2}N$, and where $\partial_{1}N$ and $\partial_{2}N$ are equipped with factorizations 
$$
\Phi_{1}: \partial_{1}N \stackrel{\cong} \longrightarrow \beta_{1}N\times\langle k\rangle, \quad \quad \Phi_{2}: \partial_{2}N \stackrel{\cong} \longrightarrow \beta_{2}N\times\langle l\rangle.
$$
\end{remark}

\begin{remark}
The above definition is a specialization of $\Sigma$-manifold from \cite[Definition 1.1.1]{B 92} and a generalization of the definition of $\langle k \rangle$-manifold. In fact, any $\langle k \rangle$-manifold $P$ is a $\langle k, l \rangle$-manifold with $\beta_{2}P = \emptyset$.
\end{remark}

As for the case with $\langle k \rangle$-manifolds, we will drop the structure maps $\Phi_{1}, \Phi_{2}, \Phi_{1,2}$ from the notation and denote $N := (N, \Phi_{1}, \Phi_{2}, \Phi_{1,2})$.
The manifold $\partial_{0}N$ is referred to as the boundary of the $\langle k, l\rangle$-manifold and is a $\langle k, l\rangle$-manifold in its own right. 
A compact $\langle k, l\rangle$-manifold $N$ is said to be \textit{closed} if $\partial_{0}N = \emptyset$. 
From a $\langle k, l\rangle$-manifold $N$, one obtains a manifold with cone-type singularities in the following way. 
\begin{defn} \label{defn: singular k-k mfd} Let $N$ be a $\langle k, l\rangle$-manifold. 
Let $\bar{\Phi}_{1}: \partial_{1}N \longrightarrow \beta_{1}N$ be the map defined by the composition 
$\xymatrix{
\partial_{1}N \ar[r]^{\Phi_{1}\ \ \ }_{\cong \ \ \ } & \beta_{1}N\times\langle k \rangle \ar[rr]^{\text{proj}_{\beta_{1}N}} && \beta_{1}N.}$
Define $\bar{\Phi}_{2}: \partial_{2}N \longrightarrow \beta_{2}N$ similarly. 
We define $\widehat{N}$ to be the quotient space obtained from $N$ by identifying two points $x, y$ if and only if
for $i = 1 \; \text{or} \; 2$, both $x$ and $y$ are in $\partial_{i}N$ and 
$\bar{\Phi}_{i}(x) = \bar{\Phi}_{i}(y).$
\end{defn}

\subsection{Oriented $\langle k, l\rangle$-Bordism} We will need to make use of the oriented bordism groups of $\langle k, l\rangle$-manifolds. 
For any space $X$ and non-negative integer $j$, we denote by $\Omega_{j}^{SO}(X)_{\langle k, l\rangle}$ the $j$-th $\langle k, l\rangle$-bordism group associated to the space $X$. We refer the reader to \cite{B 92} for details on the definition. 
 There are maps 
$$\beta_{1}: \Omega_{j}^{SO}(X)_{\langle k, l\rangle} \longrightarrow \Omega^{SO}_{j-1}(X)_{\langle l\rangle}, \quad \beta_{2}: \Omega_{j}^{SO}(X)_{\langle k, l\rangle} \longrightarrow \Omega^{SO}_{j-1}(X)_{\langle k\rangle}$$
defined by sending a $\langle k, l\rangle$-manifold $N$ to $\beta_{1}N$ and $\beta_{2}N$ respectively. We also have maps 
$$j_{1}: \Omega_{j}^{SO}(X)_{\langle k\rangle} \longrightarrow \Omega_{j}^{SO}(X)_{\langle k, l\rangle}, \quad j_{2}: \Omega_{j}^{SO}(X)_{\langle l\rangle} \longrightarrow \Omega_{j}^{SO}(X)_{\langle k, l\rangle}$$
defined by considering a $\langle k \rangle$-manifold or an $\langle l \rangle$-manifold as a $\langle k, l\rangle$-manifold. 
We have the following theorem from \cite{B 92}. 
\begin{theorem} \label{thm: k-l bordism LES} The following sequences are exact,
 $$\xymatrix@C-1pc@R-1.0pc{
  \cdots \ar[r] & \Omega_{j}^{SO}(X)_{\langle l \rangle} \ar[rr]^{\times l} && \Omega_{j}^{SO}(X)_{\langle l \rangle} \ar[rr]^{j_{1}} && \Omega_{j}^{SO}(X)_{\langle k, l\rangle} \ar[rr]^{\beta_{1}} && \Omega_{j-1}^{SO}(X)_{\langle l\rangle} \ar[r] & \cdots\\
 \cdots \ar[r] & \Omega_{j}^{SO}(X)_{\langle k \rangle} \ar[rr]^{\times k} && \Omega_{j}^{SO}(X)_{\langle k \rangle} \ar[rr]^{j_{2}} && \Omega_{j}^{SO}(X)_{\langle k, l\rangle} \ar[rr]^{\beta_{2}} && \Omega_{j-1}^{SO}(X)_{\langle k \rangle} \ar[r] & \cdots
}$$
 \end{theorem} 
 Using the isomorphisms 
 $\Omega^{SO}_{0}(\text{pt.})_{\langle k \rangle} \cong \Z/k$ and $\Omega^{SO}_{1}(\text{pt.})_{\langle k \rangle} = 0,$
we obtain the following basic calculations using the above exact sequence. 
 \begin{corollary} \label{corollary: 0th group calculation}
 For any two integers $k, l \geq 2$ we have the following isomorphisms,
$$\Omega^{SO}_{0}(\text{pt.})_{\langle k, l\rangle} \cong \Z/\gcd(k,l) \quad \text{and} \quad  \Omega^{SO}_{1}(\text{pt.})_{\langle k, l\rangle} \cong \Z/\gcd(k,l).$$
In particular we have,
$$\Omega^{SO}_{0}(\text{pt.})_{\langle k, k\rangle} \cong \Z/k \quad \text{and} \quad  \Omega^{SO}_{1}(\text{pt.})_{\langle k, k\rangle} \cong \Z/k.$$
  \end{corollary}
   
 \subsection{$1$-dimensional, closed, oriented, $\langle k, k\rangle$-manifolds.} \label{subsection: classification of k, l manifolds}
 We will need to consider $1$-dimensional $\langle k, k \rangle$-manifolds. 
 They will arise for us as the intersections of $(n+1)$-dimensional $\langle k\rangle$-manifolds immersed in a $(2n+1)$-dimensional manifold. 
Denote by $A_{k}$ the space $[0,1]\times\langle k \rangle$. By setting 
 $$\partial_{1}A_{k} = \{0\}\times\langle k \rangle \quad \text{and} \quad  \partial_{2}A_{k} = \{1\}\times\langle k \rangle,$$ 
 $A_{k}$ naturally has the structure of a closed $\langle k, k\rangle$-manifold with, $\beta_{1}A_{k} = \langle 1\rangle = \beta_{2}A_{k}$ (the single point space). 
 We denote by $+A_{k}$ the oriented $\langle k, k\rangle$-manifold with orientation induced by the standard orientation on $[0,1]$.
 We denote by $-A_{k}$ the $\langle k, k\rangle$-manifold equipped with the opposite orientation. 
It follows that 
\begin{equation} \label{equation: bockstein of A-k}
\beta_{1}(\pm A_{k}) = \pm\langle 1 \rangle \quad \text{and} \quad \beta_{2}(\pm A_{k}) = \mp\langle 1\rangle.
\end{equation}
Using the fact that the map
$\beta_{i}: \Omega^{SO}_{1}(\text{pt.})_{\langle k, k\rangle} \longrightarrow \Omega^{SO}_{0}(\text{pt.})_{\langle k\rangle} \quad \text{for $i = 1, 2$}$
is an isomorphism (this follows from Corollary \ref{corollary: 0th group calculation} and the exact sequence in Theorem \ref{thm: k-l bordism LES}), we have the following proposition. 
\begin{proposition} \label{proposition: generator of k-k bordism}
The oriented, closed, $\langle k, k\rangle$ manifold $+A_{k}$ represents a generator for $\Omega^{SO}_{1}(\text{pt.})_{\langle k, k\rangle}$.
Furthermore, any oriented, closed, $1$-dimensional $\langle k, k\rangle$-manifold that represents a generator of $\Omega^{SO}_{1}(\text{pt.})_{\langle k, k\rangle}$ is of the form 
$$(+A_{k}\times\langle r \rangle)\sqcup (-A_{k}\times\langle s\rangle) \sqcup X,$$
where $r, s \in \N$ are such that $r - s$ is relatively prime to $k$, and where $X$ is some null-bordant $\langle k, k\rangle$-manifold such that 
$\beta_{1}X = \emptyset$ or $\beta_{2}X = \emptyset$ (in other words, $X$ has the structure of  $\langle k \rangle$-manifold). 
\end{proposition}
Throughout, we will consider the element of $\Omega^{SO}_{1}(\text{pt.})_{\langle k, k\rangle}$ determined by the oriented $\langle k, k \rangle$-manifold $+A_{k}$ to be the \textit{standard generator}.

 \section{Intersections} \label{section: intersections}
 In this section we will discuss the intersections of embeddings of $\langle k \rangle$-manifolds. 
 \subsection{Preliminaries} \label{subsection: intersection preliminaries}
 Here we review some of the basics about intersections of embedded smooth manifolds and introduce some terminology and notation.
For what follows, let $M$, $X$, and $Y$ be compact, oriented, smooth manifolds of dimension $m$, $r$, and $s$ respectively and let $t$ denote the integer $r + s - m$. 
 Let 
 \begin{equation} \label{equation: the two given maps}
 \varphi: (X, \partial X) \longrightarrow (M, \partial M) \quad \text{and} \quad \psi: (Y, \partial Y) \longrightarrow (M, \partial M)
 \end{equation}
 be smooth, transversal maps such that $\varphi(\partial X)\cap \psi(\partial Y) = \emptyset$ (for these two maps to be transversal, we mean that the product map $\varphi\times\psi: X\times Y \longrightarrow M\times M$ is transverse to the diagonal submanifold $\triangle_{M} \subset M\times M$). 
 We let  $\varphi\pitchfork \psi$ denote the \textit{transverse pull-back} $(\varphi\times\psi)^{-1}(\triangle_{M})$, which is a closed submanifold of $X\times Y$ of dimension $t$. 
The orientations on $X$, $Y$, and $M$ induce an orientation on $\varphi\pitchfork\psi$ and thus $\varphi\pitchfork\psi$ determines a bordism class in $\Omega_{t}^{SO}(\text{pt.})$ which we denote by $\Lambda^{t}(\varphi, \psi; M)$.
It follows easily that, $\Lambda^{t}(\varphi, \psi; M) = (-1)^{(m-s)\cdot(m-r)}\Lambda^{t}(\psi, \varphi; M).$

\subsection{Intersections of $\langle k\rangle$-Manifolds}
We now proceed to consider intersections of $\langle k\rangle$-manifolds.
Let $M$ be a compact, oriented manifold of dimension $m$, let $X$ be a compact, oriented manifold of dimension $r$, and let $P$ be a compact, oriented $\langle k \rangle$-manifold of dimension $p$.
Let $t$ denote the integer $r + p - m$. 
Let 
$$\varphi: (X, \partial X) \longrightarrow (M, \partial M) \quad \text{and} \quad  f: (P, \partial_{0}P) \longrightarrow (M, \partial M)$$
be a smooth map and a smooth $\langle k\rangle$-map respectively. 
Suppose that $f$ and $\varphi$ are transversal and that $f(\partial_{0}P)\cap\varphi(\partial X) = \emptyset$ (when we say that $f$ and $\varphi$ are transversal, we mean that both $f$ and $f_{\beta}$ are transverse to $\varphi$ as smooth maps). 
The pull-back,
$$f\pitchfork\varphi = (f\times\varphi)^{-1}(\triangle_{M}) \;  \subset \; P\times X$$ 
has the structure of a closed $\langle k\rangle$-manifold as follows. 
We denote, 
$$\partial_{1}(f\pitchfork\varphi) := f|_{\partial_{1}P}\pitchfork\varphi \quad \text{and} \quad \beta(f\pitchfork\varphi) := f_{\beta}\pitchfork\varphi.$$
The factorization, 
$\xymatrix{
\partial_{1}P \ar[r]^{\bar{\Phi}\ \ }  & \beta P \ar[r]^{\ \ f_{\beta}\  } & M
}$
of the restriction map $f|_{\partial_{1}P}$ implies that the diffeomorphism, 
$$\xymatrix{
\Phi\times\Id_{X}:\partial_{1}P\times X \ar[rr]^{\cong} && (\beta P\times\langle k\rangle)\times X
}$$
maps $\partial_{1}(f\pitchfork X)$ diffeomorphically onto $\beta(f\pitchfork X)\times\langle k \rangle$.
It follows that $f\pitchfork\varphi$ has the structure of a $\langle k\rangle$-manifold of dimension $t = p+r-m$.
Furthermore, $f\pitchfork\varphi$ inherits an orientation from the orientations of $X$, $P$ and $M$. 
\begin{defn} \label{defn: k-bordism class intersection}
Let $f: (P, \partial_{0}P) \longrightarrow (M, \partial M)$ and $\varphi: (X, \partial X) \longrightarrow (M, \partial M)$ 
be exactly as above. 
We define $\Lambda^{t}_{k}(f, \varphi; M) \in \Omega^{SO}_{t}(\text{pt.})_{\langle k \rangle}$ to be the oriented bordism class determined by the pull-back $f\pitchfork \varphi$ and its induced orientation. 
\end{defn}

Recall from Section \ref{Bordism Groups of Singular Manifolds} the Bockstein homomorphism,
$\beta: \Omega_{t}^{SO}(\text{pt.})_{\langle k\rangle} \longrightarrow \Omega_{t-1}^{SO}(\text{pt.}).$
We have the following proposition. 
\begin{proposition} \label{proposition: k-intersection} 
Let $f: (P, \partial_{0}P) \longrightarrow (M, \partial M)$ and $\varphi: (X, \partial X) \longrightarrow (M, \partial M)$ 
be exactly as above.
Then
$$\beta(\Lambda^{t}_{k}(f, \varphi; M)) = \Lambda^{t-1}(f_{\beta}, \varphi; M),$$
where $\Lambda^{t-1}(f_{\beta}, \varphi; M) \in \Omega_{t-1}^{SO}(\text{pt.})$ is the bordism class defined in Section \ref{subsection: intersection preliminaries}. 
\end{proposition} 

%%%%%%%%%%%%%%%%%%%%%%%%%%%%%%%%%%%%%%%%%%%%%%%%%%%%%%%%%%%%%%%%%%%%%%%%%%%%%%%%%%%%%%%%%%%%%%%%%%%%%%%%%%%%%%%%%%%%%%%%%%%%%%%%%%%%%%%%%%%%%%%%%%%%%%%%%%%%%%%%%%%%%%%%%%%%%%%%%%%%%%%%%%%%%%%%%%%%%%%%%%%%%%%%%%%%%%%%%%%%%%%%%%%%%%%%%%%%%%

%%%%%%%%%%%%%%%%%%%%%%%%%%%%%%%%%%%%%%%%%%%%%%%%%%%%%%%%%%%%%%%%%%%%%%%%%%%%%%%%%%%%%%%%%%%%%%%%%%%%%%%%%%%%%%%%%%%%%%%%%%%%%%%%%%%%%%%%%%%%%%%%%%%%%%%%%%%%%%%%%%%%%%%%%%%%%%%%%%%%%%%%%%%%%%%%%%%%%%%%%%%%%%%%%%%%%%%%%%%%%%%%%%%%%%%%%%%%%%

\subsection{$\langle k, l\rangle$-Manifolds and intersections} \label{subsection: k,l intersections}
We now consider the intersection of a $\langle k\rangle$-manifold with an $\langle l\rangle$-manifold. 
For what follows, let $P$ be a compact, oriented $\langle k\rangle$-manifold of dimension $p$, let $Q$ be a compact, oriented $\langle l \rangle$-manifold of dimension $q$, and let $M$ be a compact, oriented manifold of dimension $m$. 
 Let 
 $$f: (P, \partial_{0}P) \longrightarrow (M, \partial M) \quad \text{and} \quad g: (Q, \partial_{0}Q) \longrightarrow (M, \partial M)$$ 
 be a smooth $\langle k \rangle$-map and a smooth $\langle l \rangle$-map respectively. 
 Suppose that $f$ and $g$ are transversal and that $f(\partial_{0}P)\cap g(\partial_{0}Q) = \emptyset$ (when we say that $f$ and $g$ are transversal, we mean that $f$ and $f_{\beta}$ are each transverse to both $g$ and $g_{\beta}$ as smooth maps). 
Let $t$ denote the integer $p+q - m$. 
We will analyze the $t$-dimensional submanifold 
 $$f\pitchfork g = (f\times g)^{-1}(\triangle_{M}) \subset P\times Q.$$ 
The transversality condition on $f$ and $g$ implies that the space $f\pitchfork g$, and the subspaces  
$$\xymatrix@C-.10pc@R-2.0pc{
f|_{\partial P}\pitchfork g \subset \partial P\times Q, \quad &f\pitchfork g|_{\partial Q} \subset P\times\partial Q, \quad &f|_{\partial P}\pitchfork g|_{\partial Q} \subset \partial P\times\partial Q,\\
f_{\beta}\pitchfork g \subset \beta P\times Q, \quad &f\pitchfork g_{\beta} \subset P\times\beta Q, \quad &f_{\beta}\pitchfork g_{\beta} \subset \beta P\times \beta Q,
}$$
are all smooth submanifolds. 
We define
$$\xymatrix@C-.10pc@R-2.0pc{
\partial_{1}(f\pitchfork g) :=  f|_{\partial P}\pitchfork g, & \quad \partial_{2}(f\pitchfork g) := f\pitchfork g|_{\partial Q}, &\quad \partial_{1, 2}(f\pitchfork g) := f|_{\partial P}\pitchfork g|_{\partial Q}, \\
\beta_{1}(f\pitchfork g) :=  f_{\beta}\pitchfork g, & \quad \beta_{2}(f\pitchfork g) := f\pitchfork g_{\beta}, &\quad \beta_{1, 2}(f\pitchfork g) := f_{\beta}\pitchfork g_{\beta}. 
}$$
\noindent
The structure maps, $\Phi_{P}: \partial P \stackrel{\cong} \longrightarrow \beta P\times\langle k\rangle$ and $\Phi_{Q}: \partial Q \stackrel{\cong} \longrightarrow \beta Q\times\langle l \rangle$ induce diffeomorphisms,
\begin{equation} \label{eq: product structure maps}
\xymatrix@C-.10pc@R-2.0pc{
\Phi_{P}\times Id: \partial P\times Q  \ar[rr]^{\cong} && \beta P\times\langle k \rangle\times Q, \\
Id\times \Phi_{Q}: P\times\partial Q \ar[rr]^{\cong} &&  P\times \beta Q\times \langle l \rangle, \\
\Phi_{P}\times \Phi_{Q}: \partial P\times\partial Q \ar[rr]^{\cong} && \beta P\times\langle k \rangle\times \beta Q\times \langle l \rangle. 
}
\end{equation}
The factorizations, 
$$\xymatrix@C-.10pc@R-2.0pc{
\partial P \ar[rr]^{\bar{\Phi}_{P}}  && \beta P \ar[rr]^{f_{\beta}} && M, \\
\partial Q \ar[rr]^{\bar{\Phi}_{Q}} && \beta Q \ar[rr]^{g_{\beta}} && M, \\
}$$
of the restriction maps $f|_{\partial P}$ and $g|_{\partial Q}$ imply that the  diffeomorphisms from (\ref{eq: product structure maps}) map the submanifolds
 $$\partial_{1}(f\pitchfork g) \subset \partial P\times Q, \quad \partial_{2}(f\pitchfork g) \subset P\times\partial Q,  \quad \text{and} \quad \partial_{1,2}(f\pitchfork g) \subset \partial P\times\partial Q$$ 
 diffeomorphically onto 
 $$\beta_{1}(f\pitchfork g)\times \langle k \rangle,  \quad \beta_{2}(f\pitchfork g)\times\langle l\rangle, \quad \text{and} \quad
\beta_{1,2}(f\pitchfork g)\times\langle k\rangle\times\langle l\rangle $$
 respectively.
 It follows that $f\pitchfork g$ has the structure of an oriented $\langle k, l \rangle$-manifold of dimension $t = p+q -m$. 
 \begin{defn}
 Let 
 $f: (P, \partial_{0}P) \longrightarrow (M, \partial M)$ and $g: (Q, \partial_{0}Q) \longrightarrow (M, \partial M)$
 be exactly as above.
 We denote by $\Lambda^{1}_{k, l}(f, g; M) \in \Omega^{SO}_{t}(\text{pt.})_{\langle k, l\rangle}$ the bordism class determined by the pull-back $f\pitchfork g$. 
 \end{defn}
 
 For the following proposition, recall from Section \ref{subsection: framed bordism} the Bockstein homomorphisms,
 $$\beta_{1}: \Omega^{SO}_{t}(\text{pt.})_{\langle k, l\rangle} \longrightarrow \Omega^{SO}_{t-1}(\text{pt.})_{\langle l\rangle} \quad \text{and} \quad 
\beta_{2}: \Omega^{SO}_{t}(\text{pt.})_{\langle k, l\rangle} \longrightarrow \Omega^{SO}_{t-1}(\text{pt.})_{\langle k \rangle}.$$
 \begin{proposition} \label{prop: k-l framed intersection} 
The bordism class $\Lambda^{t}_{k, l}(f, g; M) \in \Omega^{SO}_{t}(\text{pt.})_{\langle k, l\rangle}$ satisfies the following equations
\begin{enumerate}\itemsep.2cm
\vspace{.2cm}
\item[(i)] 
$\Lambda^{t}_{k,l}(f, g; M) = (-1)^{(m-p)\cdot(m-q)}\cdot\Lambda^{t}_{l,k}(g, f; M),$
 \item[(ii)]  
 $\beta_{1}(\Lambda^{t}_{k,l}(f, g; M)) = \Lambda^{t-1}_{l}(f_{\beta}, g; M),$ 
 \item[(iii)]
 $\beta_{2}(\Lambda^{t}_{k,l}(f, g; M)) = \Lambda^{t-1}_{k}(f, g_{\beta}; M).$
 \vspace{.2cm}
\end{enumerate}
\end{proposition} 

\subsection{Main disjunction result} \label{subsection: main result disjunction}
We now discuss the main result that we will need to use regarding the intersections of $\langle k\rangle$-manifolds.
We will need the following terminology.
 \begin{defn} \label{defn: diffeotopy}
 Let $M$ be a manifold. 
 We will call a smooth, one parameter family of diffeomorphisms $\Psi_{t}: M \longrightarrow M$ with $t \in [0, 1]$ and $\Psi_{0} = \Id_{M}$ a \textit{diffeotopy}. 
 For a subspace $N \subset M$,
 we say that $\Psi_{t}$ is a \textit{diffeotopy relative $N$}, and we write $\Psi_{t}: M \longrightarrow M \; \text{rel} \; N$, if in addition, $\Psi_{t}|_{N} = \Id_{N}$ for all $t \in [0,1]$.
  \end{defn} 
The main case of intersections of $\langle k \rangle$ and $\langle l \rangle$-manifolds that we will need to consider is the case when 
$$k = l \quad \text{and} \quad \dim(P) + \dim(Q) - \dim(M) = 1.$$ 
For $n \geq 2$, let $M$ be an oriented manifold of dimension $4n+1$
and let $P$ and $Q$ be compact oriented $\langle k \rangle$-manifolds of dimension $2n+1$. 
Let
 \begin{equation}
 f: (P, \partial_{0}P) \longrightarrow (M, \partial M) \quad \text{and} \quad g: (Q, \partial_{0}Q) \longrightarrow (M, \partial M)
 \end{equation}
 be transversal $\langle k\rangle$-embeddings such that $f(\partial_{0}P)\cap g(\partial_{0}Q) = \emptyset$.
 Suppose further that: 
 \begin{itemize} \itemsep.2cm
 \item $M$ is $2$-connected, 
 \item $P$ and $Q$ are both simply connected, 
 \item $\beta P$ and $\beta Q$ are both path connected. 
 \end{itemize}
 \begin{theorem} \label{thm: modifying intersections 1} 
 With $f$ and $g$ the $\langle k \rangle$-embeddings given above, suppose that $\Lambda^{1}_{k,k}(f, g; M) = 0$. 
 If the integer $k$ is odd, then there exists a diffeotopy 
 $\Psi_{t}: M \longrightarrow M \; \rel \partial M$ 
 such that $\Psi_{1}(f(P))\cap g(Q) = \emptyset$. 
  \end{theorem}
We also have:
 \begin{corollary} \label{corollary: intersection at A-k 1}
 Suppose that the class 
$\Lambda^{1}_{k, k}(f, g; M) \in \Omega^{SO}_{1}(\text{pt.})_{\langle k, k\rangle}$
 is equal to the class represented by the closed $1$-dimensional $\langle k, k\rangle$-manifold $+A_{k}$. 
 If $k$ is odd, there exists a diffeotopy $\Psi_{t}: M \longrightarrow M \; \rel \partial M$ such that the $\langle k, k\rangle$-manifold given by the transverse pull-back $(\Psi_{1}\circ f) \pitchfork g,$ is diffeomorphic to $A_{k}$. 
 \end{corollary}

Both of these results above are proven in Appendix \ref{section: modyifing higher dim intersections} (see Theorem \ref{thm: modifying intersections} and Corollary \ref{corollary: intersection at A-k}). 
These above results are crucial in the proof of our main homological stability theorem. 
The key place (only place) that they are used is in the proof of Lemma \ref{thm: high connectedness of K}. 
Later on we will need to apply Theorem \ref{thm: modifying intersections 1} inductively.
We will need the following important application of Theorem \ref{thm: modifying intersections 1}.
\begin{corollary} \label{corollary: inductive disjunction}
Let $n \geq 4$ and let $k \in \N$ be odd.
Let $M$ be a $2$-connected, $2n+1$ dimensional manifold. 
Let $P, Q_{1}, \dots, Q_{m}$ be compact $\langle k \rangle$-manifolds of dimension $n+1$.
Let 
$$
f: (P, \partial_{0}P) \longrightarrow (M, \partial M) \quad \text{and} \quad g_{i}: (Q_{i}, \partial_{0}Q_{i}) \longrightarrow (M, \partial M)
$$
be $\langle k \rangle$-embeddings for $i = 1, \dots m$.
Suppose that the following conditions are met:
\begin{enumerate} \itemsep.2cm
\item[(i)] $P, Q_{1}, \dots, Q_{m}$ are all simply connected;
\item[(ii)] $\beta P, \beta Q_{1}, \dots, \beta Q_{m}$ are all path-connected;
\item[(iii)] the $\langle k \rangle$-embeddings $g_{1}, \dots, g_{m}$, are pairwise transverse;
\item[(iv)] the submanifolds $f(\partial_{0}P), g_{1}(\partial_{0}Q_{0}), \dots, g_{m}(\partial_{0}Q_{m}) \subset \partial M$ are pairwise disjoint;
\item[(v)] $\Lambda_{k, k}^{1}(f, g_{i}; M) = 0$ for $i = 1, \dots, m$.
\end{enumerate}
Then exists a diffeotopy 
$\Psi_{t}: M \longrightarrow M \rel \partial M$ 
such that $\Psi_{1}(f(P))\cap g_{i}(Q_{i}) = \emptyset$ for $i = 1, \dots, m$.
\end{corollary}
\begin{proof}
We prove this by induction on $m$. 
If $m = 1$ then the corollary follows directly from Theorem \ref{thm: modifying intersections 1}. 
This establishes the base case. 
We now verify the inductive step. 
By the induction assumption we may assume that $f(P)\cap g_{i}(Q_{i}) = \emptyset$ for $i = 1, \dots, m-1$.
Since $n \geq 4$, it follows by general position that the manifold $M\setminus\left[\cup_{i=1}^{m-1}g_{i}(Q_{i})\right]$ is $2$-connected. 
Indeed, $\dim(M)-\dim(Q_{i}) \geq 4$ and thus it follows that any map $K \longrightarrow M$ with $K$ a $3$-manifold can be perturbed by a small homotopy so that it has its image disjoint from $\cup_{i=1}^{m-1}g_{i}(Q_{i}) \subset M$.
It follows from this that $\pi_{j}(M) \cong \pi_{j}\left(M\setminus[\cup_{i=1}^{m-1}g_{i}(Q_{i})]\right)$ for $j \leq 2$ and thus $M\setminus\left[\cup_{i=1}^{m-1}g_{i}(Q_{i})\right]$ is $2$-connected since $M$ is.

By condition (i), $g_{m}$ is transverse to $g_{i}$ for $i = 1, \dots, m-1$
and thus the pre-images 
$$g_{m}^{-1}(g_{i}(Q_{i})) \subset Q_{m} \quad \text{and} \quad (g_{m})_{\beta}^{-1}(g_{i}(Q_{i})) \subset \beta Q_{m}$$ 
are submanifolds of dimension $1$ and dimension $0$ respectively. 
It follows by general position (by the same basic general position argument given in the previous paragraph) that the complements 
\begin{equation} \label{equation: Q complement}
Q_{m}\setminus g_{m}^{-1}\left(\cup_{i=1}^{m-1}g_{i}(Q_{i})\right) \quad \text{and} \quad \beta Q_{m}\setminus g_{m}^{-1}\left(\cup_{i=1}^{m-1}(g_{i}(Q_{i})\right)
\end{equation}
are simply connected and path-connected respectively.

Let $U \subset M$ be a regular neighborhood of the subspace $\cup_{i=1}^{m-1}g_{i}(Q_{i}) \subset M$ and let $V$ denote the complement $M \setminus\Int(U)$. 
We set $Q'_{m} := g^{-1}_{m}(V)$. 
By condition (iv) it follows that $\partial_{0}Q_{m} \subset Q'_{m}$. 
By setting 
$$\beta Q'_{m} = ((g_{m})_{\beta})^{-1}(V) \quad \text{and} \quad \partial_{0}Q'_{m} = \partial_{0}Q_{m}\cup g_{m}^{-1}(\partial M),$$ 
$Q'_{m}$ obtains the structure of a compact $\langle k \rangle$-manifold.
By shrinking down the regular neighborhood $U$ sufficiently close to $\cup_{i=1}^{m-1}g_{i}(Q_{i})$, we may assume that $Q'_{m}$ is simply connected and that $\beta Q'_{m}$ is path-connected.
Similarly, $V$ is $2$-connected by same argument since $M\setminus\left[\cup_{i=1}^{m-1}g_{i}(Q_{i})\right]$ is $2$-connected. 
Let 
$$g'_{m}: (Q'_{m}, \partial_{0}Q'_{m}) \longrightarrow (V, \partial V)$$
be the $\langle k \rangle$-embedding obtained by restricting $g_{m}$ to $Q'_{m}$. 
It follows from condition (iii) that $\Lambda_{k, k}^{1}(f, g'_{m}; V) = 0$. 
We apply Theorem \ref{thm: modifying intersections 1} to obtain a diffeotopy $\Psi_{t}: (V, \partial V) \longrightarrow (V, \partial V) \rel \partial V$ such that $\Psi_{1}(f(P))\cap g'_{m}(Q'_{m}) = \emptyset$.
By extending the diffeotopy $\Psi_{t}$ identically over $M\setminus V$ we obtain the desired diffeotopy and prove the induction step. 
This concludes the proof of the corollary.  
\end{proof}

\begin{remark}
We emphasize that condition (iii) of Corollary \ref{corollary: inductive disjunction} is totally necessary in order conclude that the $\langle k \rangle$-manifold $Q'$ is simply connected with $\beta Q'$ path-connected, and 
without these connectivity conditions the desired diffeotopy cannot in general be constructed. 
It is not known by to author if there is any way to establish these connectivity conditions for $Q'$ without the pairwise transversality of the embeddings $g_{1}, \dots, g_{m}$ imposed in condition (iii).
\end{remark}

\subsection{Connection to the linking form}
In practice we will need to consider intersections of $\langle k \rangle$-embeddings $f, g: V^{2n+1}_{k} \longrightarrow M$. 
We will need to relate $\Lambda^{1}_{k, k}(f, g; M)$ to the homotopical linking form
$b: \pi^{\tau}_{2n}(M)\otimes\pi^{\tau}_{2n}(M) \longrightarrow \Q/\Z.$
Let 
$$T_{k}: \Omega^{SO}_{1}(\text{pt.})_{\langle k, k\rangle} \longrightarrow \Q/\Z$$ 
be the homomorphism given by the composition
$$\xymatrix{
\Omega^{SO}_{1}(\text{pt.})_{\langle k, k\rangle} \ar[rrr]^{+A_{k} \mapsto 1} &&& \Z/k \ar[rrr]^{1 \mapsto 1/k} &&& \Q/\Z.
}$$
The following proposition follows easily from the definition of the homological linking form (\ref{eq: linking form chain level}). 
\begin{proposition} \label{proposition: connection to the linking form}
Let $M$ be a $(4n+1)$-dimensional, oriented manifold. 
Let 
$$f, g: V^{2n+1} \longrightarrow M$$ 
be $\langle k \rangle$-embeddings.
Consider the homotopy classes $[f_{\beta}], [g_{\beta}] \in \pi^{\tau}_{2n}(M)$, which both have order $k$. 
Then 
$$b([f_{\beta}], [g_{\beta}]) \; = \; T_{k}(\Lambda^{1}_{k,k}(f, g; M)).$$
\end{proposition}

%%%%%%%%%%%%%%%%%%%%%%%%%%%%%%%%%%%%%%%%%%%%%%%%%%%%%%%%%%%%%%%%%%%%%%%%%%%%%%%%%%%%%%%%%%%%%%%%%%%%%%%%%%%%%%%%%%%%%%%%%%%%%%%%%%%%%%%%%%%%%%%%%%%%%%%%%%%%%%%%%%%%%%%%%%%%%%%%%%%%%%%%%%%%%%%%%%%%%%%%%%%%%%%%%%%
%%%%%%%%%%%%%%%%%%%%%%%%%%%%%%%%%%%%%%%%%%%%%%%%%%%%%%%%%%%%%%%%%%%%%%%%%%%%%%%%%%%%%%%%%%%%%%%%%%%%%%%%%%%%%%%%%%%%%%%%%%%%%%%%%%%%%%%%%%%%%%%%%%%%%%%%%%%%%%%%%%%%%%%%%%%%%%%%%%%%%%%%%%%%%%%%%%%%%%%%%%%%%%%%%%%%%%%%%%%%%%%%%%%%%%%%%%%%%%%%%%%%%
\section{Topological Flag Complexes} \label{section: semi-simplicial spaces}
In this section we define a series of simplicial complexes and semi-simplicial spaces.
\subsection{The primary semi-simplicial space.} 
Fix integers $k, n \geq 2$. 
Let $W_{k}$ denote the closed $(4n+1)$-dimensional manifold $W_{k}$ defined in Section \ref{subsection: the classification theorem}. 
We will make a slight alteration of $W_{k}$ as follows. 
Let $W'_{k}$ denote the manifold obtained from $W_{k}$ by removing an open disk. 
Choose an oriented embedding $\alpha: \{1\}\times D^{4n} \longrightarrow \partial W'_{k}$.
We then define $\bar{W}_{k}$ to be the manifold obtained by attaching $[0, 1]\times D^{4n}$ to $W'_{k}$ by the embedding $\alpha$, i.e.\ 
\begin{equation} \label{eq: W-bar}
\bar{W}_{k} := ([0,1]\times D^{4n})\cup_{\alpha}W'_{k}.
\end{equation}
Let $M$ be a $(4n+1)$-dimensional manifold with non-empty boundary. 
Fix an embedding
$$a: [0,\infty)\times \R^{4n} \longrightarrow M$$
with $a^{-1}(\partial M) = \{0\}\times\R^{4n}$. 

\begin{defn} \label{defn: the embedding complex 1} Let $M$ and $a: [0,\infty)\times\R^{4n} \longrightarrow M$ be as above and let $k \geq 2$ be an integer. 
We define a semi-simplicial space $X_{\bullet}(M, a)_{k}$ as follows:
\begin{enumerate} \itemsep.2cm
\item[(i)] 
Let  $X_{0}(M, a)_{k}$ be the set of pairs $(\phi, t)$, where $t \in \R$ and 
$\phi: \bar{W}_{k} \rightarrow M$ is an embedding that satisfies the following condition:
there exists $\epsilon > 0$ such that for
$(s, z) \in [0, \epsilon)\times D^{4n} \; \subset \; \bar{W}_{k}$,
the equality $\phi(s, z) = a(s, z + te_{1})$ is satisfied ($e_{1} \in \R^{4n}$ denotes the first basis vector). 
\item[(ii)] 
For an integer $p \geq 0$, $X_{p}(M, a)_{k}$ is defined to be the set of ordered $(p+1)$-tuples 
$$((\phi_{0}, t_{0}), \dots, (\phi_{p}, t_{p})) \in (X_{0}(M, a)_{k})^{\times(p+1)}$$
such that $t_{0} < \cdots < t_{p}$ and 
$\phi_{i}(\bar{W}_{k})\cap\phi_{j}(\bar{W}_{k}) = \emptyset \quad \text{whenever $i \neq j$.}$
\item[(iii)] For each $p$, the space $X_{p}(M, a)_{k}$ is topologized in the $C^{\infty}$-topology as a subspace of the product
$(\Emb(\bar{W}_{k}, M)\times \R)^{\times(p+1)}.$ 
\item[(iv)] The assignment $[p] \mapsto X_{p}(M, a)_{k}$ makes $X_{\bullet}(M, a)_{k}$ into a semi-simplicial space where the $i$-th face map $X_{p}(M, a)_{k} \rightarrow X_{p-1}(M, a)_{k}$ is given by 
$$((\phi_{0}, t_{0}, \dots, (\phi_{p}, t_{p})) \mapsto ((\phi_{0}, t_{0}, \dots, \widehat{(\phi_{i}, t_{i})}, \dots,  (\phi_{p}, t_{p})).$$ 
\end{enumerate}
\end{defn}
It is easy to verify that $X_{\bullet}(M, a)_{k}$ is a topological flag complex. 
For any $0$-simplex $(\phi, t) \in X_{0}(M, a)_{k}$, it follows from condition (i) that the number $t$ is determined by the embedding $\phi$. 
For this reason we will usually drop the number $t$ when denoting elements of $X_{0}(M, a)_{k}$.

We now state a consequence of connectivity of the geometric realization $|X_{\bullet}(M, a)_{k}|$, using Lemma \ref{lemma: flag complex transitivity}.
This is essentially the same as \cite[Proposition 4.4]{GRW 12} and so we only give a sketch of the proof. 

\begin{proposition}[Transitivity] \label{proposition: transitivity}
For $n \geq 2$, let $M$ be a $(4n+1)$-dimensional manifold with non-empty boundary. 
Let $k \geq 2$ be an integer, and let $\phi_{0}$ and $\phi_{1}$ be elements of $X_{0}(M, a)_{k}$. 
Suppose that the geometric realization $|X_{\bullet}(M, a)_{k}|$ is path connected. 
Then there exists a diffeomorphism $\psi: M \stackrel{\cong} \longrightarrow M$, isotopic to the identity when restricted to the boundary, such that $\psi\circ \phi_{0} = \phi_{1}$. 
\end{proposition}
\begin{proof}[Proof Sketch]
Let $a: [0,\infty)\times\R^{4n} \longrightarrow M$ be the embedding used in the definition of $X_{\bullet}(M, a)$. 
Let $\Diff(M, a)$ denote the group of diffeomorphisms 
$\psi: M \longrightarrow M$ 
with 
$$\psi(a([0,\infty)\times\R^{4n})) \subset a([0,\infty)\times\R^{4n})$$ 
and such that $\psi|_{\partial M}$ is isotopic to the identity. 
This group acts on $X_{\bullet}(M, a)_{k}$, and 
by Lemma \ref{lemma: flag complex transitivity} it will suffice to show that for $(\phi_{0}, \phi_{1}) \in X_{1}(M, a)_{k}$, there exists $\psi \in \Diff(M, a)$ such that $\psi\circ\phi_{0} = \phi_{1}$. 
Let $U \subset M$ be a collar neighborhood of the boundary of $M$, that contains $a([0, \infty)\times\R^{4n})$. 
The union 
$\phi_{0}(\bar{W}_{k})\cup\phi_{1}(\bar{W}_{k})\cup U$
is diffeomorphic to manifold 
\begin{equation} \label{equation: double connected-sum}
W_{k}\#(\partial M\times[0, 1])\#W_{k}.
\end{equation}
To find the desired diffeomorphism $\psi$, it will suffice to construct a diffeomorphism of (\ref{equation: double connected-sum}), that is isotopic to the identity on the first boundary component, is equal to the identity on the second boundary component, and that permutes the two embedded copies of $W'_{k}$, that come from the two connected-sum factors. 
Such a diffeomorphism can be constructed ``by hand'' using the same procedure that was employed in the proof of \cite[Proposition 4.4]{GRW 12}. 
We leave the details of this construction to the reader. 
\end{proof}
The next proposition is proven in the same way as \cite[Corollary 4.5]{GRW 12}, using Proposition \ref{proposition: transitivity}.
\begin{proposition}[Cancellation] \label{corollary: cancelation}
Let $M$ and $N$ be $(4n+1)$-dimensional manifolds with non-empty boundaries, equipped with a specified identification, $\partial M = \partial N$. 
For $k \geq 2$, suppose that there exists a diffeomorphism $M\# W_{k} \stackrel{\cong} \longrightarrow N\# W_{k}$,
equal to the identity when restricted to the boundary. 
Then if $|X_{\bullet}(M\# W_{k}, a)_{k}|$ is path-connected, there exists a diffeomorphism 
$M \stackrel{\cong} \longrightarrow N$
which is equal to the identity when restricted to the boundary. 
\end{proposition}

The main theorem that we will need is the following. 
Recall from  (\ref{equation: k-rank mfd}) the \textit{$k$-rank} $r_{k}(M)$, of a $(4n+1)$-dimensional manifold $M$.
\begin{theorem} \label{theorem: high connectivity of X}
Let $n, k \geq 2$ be integers with $k$ odd. 
Let $M$ be a $2$-connected, $(4n+1)$-dimensional manifold with non-empty boundary. 
Let $g \in \N$ be an integer such that $r_{k}(M) \geq g$. 
Then the geometric realization $|X_{\bullet}(M, a)_{k}|$ is $\frac{1}{2}(g - 4)$-connected. 
\end{theorem}
The proof of this theorem will require several intermediate constructions. The proof will be given at the end of the section.

\subsection{The complex of $\langle k \rangle$-embeddings} \label{subsection: complex of k embeddings}
Fix integers $n, k \geq 2$. 
Let $M$ be a manifold of dimension $(4n+1)$ with non-empty boundary. 
Consider transversal $\langle k \rangle$-embeddings  
$$\varphi^{0}, \varphi^{1}: V^{2n+1}_{k} \longrightarrow M$$
such that the transverse pull-back  $\varphi^{0}\pitchfork \varphi^{1}$ is diffeomorphic to $A_{k}$ as a $\langle k, k \rangle$-manifold. 
It follows that $\varphi^{0}(V^{2n+1}_{k})\cap\varphi^{1}(V^{2n+1}_{k}) \cong \widehat{A}_{k}$, where $\widehat{A}_{k}$ is the singular space obtained from $A_{k}$ as in Definition \ref{defn: singular k-k mfd}. 
It will be useful to have an abstract model for the space given by the union,  $\varphi^{0}(V^{2n+1}_{k})\cup\varphi^{1}(V^{2n+1}_{k})$.
\begin{Construction} \label{construction: pushout space} 
To begin the construction, fix a point $y \in \Int(V^{2n+1}_{k})$.  
For $i = 1, \dots, k$, let $\partial^{i}_{1}V^{2n+1}_{k}$ denote the component of the boundary given by $\Phi^{-1}(\beta V^{2n+1}_{k}\times\{i\})$, where $\langle k \rangle = \{1, \dots, k\}$. 
Let $\bar{\Phi}: \partial_{1}V^{2n+1}_{k} \longrightarrow \beta V^{2n+1}_{k}$ be the map used in Definition \ref{defn: associated singular space}. 
\begin{enumerate}\itemsep.4cm
\item[(i)] For $i = 1, \dots, k$, fix points $x_{i} \in \partial_{1}^{i}V^{2n+1}_{k}$ such that, 
$\bar{\Phi}(x_{1}) = \cdots = \bar{\Phi}(x_{k})$.
\item[(ii)] For $i = 1, \dots, k,$ choose embeddings 
$\gamma_{i}: [0, 1] \longrightarrow V^{2n+1}_{k}$
such that  
$$\gamma_{i}(0) = x_{i},  \quad \quad \gamma^{-1}_{i}(\partial_{1}V^{2n+1}_{k}) = \{0\},  \quad \quad \text{and}  \quad \quad \gamma_{i}(1) = y.$$
Then for each $i$, let $\bar{\gamma}_{i}: [0,1] \longrightarrow V^{2n+1}_{k}$ be the embedding given by the formula $$\bar{\gamma}_{i}(t) = \gamma(1 - t).$$ 
\item[(iii)] Recall that $A_{k} = [0,1]\times\langle k \rangle = \sqcup^{k}_{i= 1}[0,1]$. 
The maps
$$
\sqcup_{i=1}^{k}\gamma_{i}: A_{k} \longrightarrow V^{2n+1}_{k} \quad \text{and} \quad \sqcup_{i=1}^{k}\bar{\gamma}_{i}: A_{k} \longrightarrow V^{2n+1}_{k},
$$
 yield embeddings 
$$\Gamma: \widehat{A}_{k} \longrightarrow \widehat{V}^{2n+1}_{k} \quad \text{and} \quad \bar{\Gamma}: \widehat{A}_{k} \longrightarrow \widehat{V}^{2n+1}_{k}.$$
\item[(iv)] We define $Y^{2n+1}_{k}$ to be the space obtained by forming the push-out of the diagram,
\begin{equation} \label{equation: pushout manifold}
\xymatrix@C-.10pc@R-1.5pc{
& \widehat{A}_{k} \ar[dl]_{\Gamma} \ar[dr]^{\bar{\Gamma}} & \\
\widehat{V}^{2n+1}_{k} & & \widehat{V}^{2n+1}_{k}
}
\end{equation}
\item[(v)] By applying the \textit{Mayer-Vietoris sequence} and \textit{Van Kampen's theorem} we compute,
$$
H_{s}(Y^{2n+1}_{k}; \Z) \cong 
\begin{cases}
\Z/k\oplus \Z/k &\quad \text{if $s = 2n$,}\\
\Z^{\oplus(k-1)} &\quad \text{if $s = 1$,}\\
\Z &\quad \textit{if $s = 0$,}
\end{cases}
\quad \quad \quad
\pi_{1}(Y_{k}) \cong \Z^{\star(k-1)},
$$
where $\Z^{\star(k-1)}$ denotes the free-group on $(k-1)$-generators.
\end{enumerate}
\end{Construction}
The next proposition follows easily by inspection. 
\begin{proposition} \label{proposition: embedding of Y-k}
Let $\varphi^{0}, \varphi^{1}: V^{2n+1}_{k} \longrightarrow M$ be transversal $\langle k \rangle$-embeddings such that the pull-back is diffeomorphic to $A_{k}$ as a $\langle k, k\rangle$-manifold. 
Then the union $\varphi^{0}(V^{2n+1}_{k})\cup\varphi^{1}(V^{2n+1}_{k})$ is homeomorphic to the space $Y^{2n+1}_{k}$. 
\end{proposition}

\begin{notation} \label{notaion: union subspace}
Let $\varphi = (\varphi^{0}, \varphi^{1})$ be a pair of $\langle k \rangle$-embeddings $\varphi^{0}, \varphi^{1}: V^{2n+1}_{k} \longrightarrow M$ such that the transverse pull-back is diffeomorphic to $A_{k}$ as a $\langle k, k\rangle$-manifold.
We will denote by $Y_{k}(\varphi^{0}, \varphi^{1})$ the subspace of $M$ given by the union $\varphi^{0}(V^{2n+1}_{k})\cup\varphi^{1}(V^{2n+1}_{k})$. 
\end{notation}
 We now define a simplicial complex based on pairs of $\langle k\rangle$-embeddings $V_{k}^{2n+1} \rightarrow M$ as above.
\begin{defn} \label{defn: semi-simplicial space versions}
Let $M$ and $k$ be as above. 
Let $K(M)_{k}$ be the simplicial complex with vertex set given by the set of all pairs $(\varphi^{0}, \varphi^{1})$ of transverse $\langle k \rangle$-embeddings 
$$\varphi^{0}, \varphi^{1}: V^{2n+1}_{k} \longrightarrow M$$ 
such that the transverse pull-back is diffeomorphic to $A_{k}$ as a $\langle k, k\rangle$-manifold.
A set 
$$\{(\varphi^{0}_{0}, \varphi^{1}_{0}) \dots, (\varphi^{0}_{p}, \varphi^{1}_{p})\}$$ 
of vertices forms a $p$-simplex if 
$Y_{k}(\varphi^{0}_{i}, \varphi^{1}_{i})\cap Y_{k}(\varphi^{0}_{j}, \varphi^{1}_{j}) = \emptyset$ whenever $i \neq j$.
\end{defn}
Recall from Section \ref{section: Algebra}, the simplicial complex $L(\mb{M})_{k}$ associated to an object $\mb{M}$ of the category $\mathcal{L}_{-}^{s}$ of strictly skew symmetric linking forms. 
We will need to compare the simplicial complex $K(M)_{k}$ to the simplicial complex $L(\pi^{\tau}_{2n}(M))_{k}$, where $(\pi^{\tau}_{2n}(M), \; b)$ is the homotopical linking form associated to $M$, see (\ref{eq: homotopical linking form}). 
We construct a simplicial map 
\begin{equation} \label{eq: algebraic comparison map}
F: K(M)_{k} \longrightarrow L(\pi^{\tau}_{2n}(M))_{k}
\end{equation}
as follows. 
For a vertex $\varphi = (\varphi^{0}, \varphi^{1}) \in K(M)_{k}$, let $\langle [\varphi^{0}_{\beta}], [\varphi^{1}_{\beta}]\rangle \leq \pi^{\tau}_{2n}(M)$ denote the subgroup generated by the homotopy classes determined by the embeddings $\varphi^{\nu}_{\beta}: S^{2n} \rightarrow M$ for $\nu = 0, 1$. 
The classes $[\varphi^{\nu}_{\beta}]$ for $\nu = 0, 1$, each have order $k$ and furthermore $b([\varphi^{0}_{\beta}], [\varphi^{1}_{\beta}]) = \frac{1}{k}$. 
It follows that the sub-linking form given by $\langle [\varphi^{0}_{\beta}], [\varphi^{1}_{\beta}]\rangle  \leq \pi^{\tau}_{2n}(M)$ is isomorphic to the standard non-singular linking form $\mb{W}_{k}$. 
The map $F$ from (\ref{eq: algebraic comparison map}), is then defined by sending a vertex $\varphi$ to the morphism of linking forms $\mb{W}_{k} \rightarrow \pi^{\tau}_{2n}(M)$ determined by the correspondence
$$
\rho \mapsto [\varphi^{0}_{\beta}], \quad \sigma \mapsto  [\varphi^{1}_{\beta}],
$$
where $\rho$ and $\sigma$ are the standard generators of $\mb{W}_{k}$. 
The disjointness condition in Definition \ref{defn: semi-simplicial space versions} implies that this formula preserves all adjacencies and thus yields a well defined simplicial map. 
It follows easily that for any $(4n+1)$-dimensional manifold $M$ and integer $k \geq 2$ that 
\begin{equation}
r_{k}(\pi^{\tau}_{2n}(M)) \geq r_{k}(M)
\end{equation}
where recall, $r_{k}(\pi^{\tau}_{2n}(M))$ is the \text{$k$-rank} of the linking form $(\pi^{\tau}_{2n}(M), b)$ as defined in Definition \ref{defn: rank of linking form}, and $r_{k}(M)$ is the $k$-rank of the manifold $M$ as defined in the introduction. 
\begin{lemma} \label{thm: high connectedness of K}
Let $n, k \geq 2$ be integers with $k$ odd. 
Let $M$ be a $2$-connected manifold of dimension $4n+1$. 
Then the geometric realization $|K(M)_{k}|$ is $\frac{1}{2}(r_{k}(M) - 4)$-connected and $lCM(K(M)_{k}) \geq \frac{1}{2}(r_{k}(M) - 1)$.
\end{lemma}

The above lemma will be proven by applying Lemma \ref{lemma: link lift lemma} to the simplicial map (\ref{eq: algebraic comparison map}).
In order to apply Lemma \ref{lemma: link lift lemma} we will need to define a suitable symmetric relation on the vertices of the complex $K(M)_{k}$.
\begin{defn}
Define $\mathcal{T} \subset K(M)_{k}\times K(M)_{k}$ to be the subset consisting of all pairs $\left((\varphi^{0}_{a}, \varphi^{1}_{a}), \; (\varphi^{0}_{b}, \varphi^{1}_{b})\right)$ such that the $\langle k \rangle$-embeddings $\varphi^{0}_{a}$ and $\varphi^{1}_{a}$ are transverse to both $\varphi^{0}_{b}$ and $\varphi^{1}_{b}$.
\end{defn}
Clearly $\mathcal{T}$ forms a symmetric relation on the vertices of $K(M)_{k}$.
Two embeddings of $V^{2n+1}_{k}$ into $M$ are automatically transverse if they are disjoint. 
It follows from this fact that the relation $\mathcal{T}$ is \text{edge compatible} (recall from Section \ref{section: simplicial techniques} that edge compatibility means that $(v, w) \in \mathcal{T}$ whenever $v$ and $w$ are adjacent vertices in $K(M)_{k}$).

\begin{proof}[Proof of Lemma \ref{thm: high connectedness of K}]
Let $r_{k}(M) \geq g$. 
Since $L(\pi^{\tau}_{2n}(M))_{k}$ is $\frac{1}{2}(g - 4)$-connected and 
$$lCM(L(\pi^{\tau}_{2n}(M))_{k}) \geq \tfrac{1}{2}(g - 1),$$ 
the proof of the lemma will follow directly from Lemma \ref{lemma: link lift lemma} once we verify two properties:
\begin{enumerate} \itemsep.2cm
\item[(i)] the map $F$ has the link lifting property with respect to the relation $\mathcal{T}$ (see Definition \ref{defn: cone lifting property}),
\item[(ii)] $F(\lk_{K(M)_{k}}(\zeta)) \leq \lk_{L(\pi^{\tau}_{2n}(M))_{k}}(F(\zeta))$ for any simplex $\zeta \in K(M)_{k}$.
\end{enumerate}
We first establish property (i).
Let $f: \mb{W}_{k} \longrightarrow \pi^{\tau}_{2n}(M)$ be a morphism of linking forms (which determines a vertex in $L(\pi^{\tau}_{2n}(M))_{k}$).
Let $w_{1}, \dots, w_{m} \in K(M)_{k}$ be a collection of vertices in general position with respect to $\mathcal{T}$ such that 
$F(w_{i}) \in \textstyle{\lk_{L(\pi^{\tau}_{2n}(M))_{k}}(f)}$ for all $i = 1, \dots, m$.
Let $v_{1}, \dots, v_{N} \in K(M)_{k}$ be an arbitrary collection of vertices. 
To prove that $F$ has the link lifting property relative to $\mathcal{T}$ we will need to find $u \in K(M)_{k}$ with $F(u) = f$ such that 
$$w_{i} \in \textstyle{\lk_{K(M)_{k}}(u)} \quad \text{and} \quad (v_{j}, u) \in \mathcal{T}$$ 
for all $i$ and $j$.
We do this as follows. 
Let $\rho, \sigma \in \mb{W}_{k}$ denote the standard generators as defined in Section \ref{section: Algebra}. 
The elements $f(\rho), f(\sigma) \in \pi^{\tau}_{2n}(M)$ have order $k$ and thus by Corollary \ref{lemma: represent by embedding} we may choose $\langle k \rangle$-embeddings $\varphi^{0}, \varphi^{1}: V^{2n+1}_{k} \longrightarrow M$ such that $[\varphi^{0}_{\beta}] = f(\rho)$ and $[\varphi^{1}_{\beta}] = f(\sigma)$. 
Let us write 
$$
(\varphi^{0}_{i}, \varphi^{1}_{i}) := w_{i}, \quad (\psi^{0}_{j}, \psi^{1}_{j}) := v_{j} \quad \text{for $i = 1, \dots, m$ \; and \; $j = 1, \dots, N$.}
$$
The $\langle k \rangle$-embeddings $\varphi^{0}_{1}, \dots, \varphi^{0}_{m}, \varphi^{1}_{1}, \dots, \varphi^{1}_{m}$ are pairwise transverse. 
Furthermore, we have 
$$\Lambda_{k, k}^{1}(\varphi^{\nu}, \varphi^{\mu}_{i}) = 0 \quad \text{for $i = 1, \dots, m$ \; and \; $\nu, \mu = 0, 1$.}$$
By Corollary \ref{corollary: inductive disjunction} we may find isotopies of the $\langle k \rangle$-embeddings $\varphi^{0}$ and $\varphi^{1}$, to new $\langle k \rangle$-embeddings $\bar{\varphi}^{0}$ and $\bar{\varphi}^{1}$ such that 
$$\bar{\varphi}^{\nu}(V^{2n+1}_{k})\cap\varphi^{\mu}_{i}(V^{2n+1}_{k}) = \emptyset \quad \text{for $i = 1, \dots, m$ and $\nu, \mu = 0, 1$.}$$
Now, let $U \subset M$ denote the complement
$$
U := M\setminus\left[\cup_{i=1}^{m}Y_{k}(\varphi^{0}_{i}, \varphi^{1}_{i})\right].
$$
Furthermore, since $b(f(\rho), f(\sigma)) = \frac{1}{k} \mod 1$, it follows from Proposition \ref{proposition: connection to the linking form} that, 
$$\Lambda^{1}_{k, k}(\varphi^{0},\; \varphi^{1}; M) = [+A_{k}]  \in \Omega^{SO}_{1}(\text{pt.})_{\langle k, k\rangle},$$
and thus
$$
\Lambda^{1}_{k, k}(\bar{\varphi}^{0},\; \bar{\varphi}^{1}; U) = [+A_{k}],
$$
since the embeddings are isotopic.
We then may apply Corollary \ref{corollary: intersection at A-k} (or Corollary \ref{corollary: intersection at A-k 1}) so as to obtain an isotopy of $\bar{\varphi}^{0}$ through $\langle k\rangle$-embeddings in $U$, to a $\langle k\rangle$-embedding $\widehat{\varphi}^{0}: V^{2n+1}_{k} \longrightarrow U$, so that $\widehat{\varphi}^{0}\pitchfork\bar{\varphi}^{1} \cong A_{k}$. 
Finally we make a small perturbation of the $\langle k \rangle$-embeddings $\widehat{\varphi}_{0}$ and $\bar{\varphi}_{1}$ so that they are transverse to the embeddings $\psi^{\nu}_{j}$ for $j = 1, \dots, N$ and $\nu = 0, 1$. 
We may choose this perturbation so small so that the images of the resulting embeddings remain disjoint from $\cup_{i=1}^{m}Y_{k}(\varphi^{0}_{i}, \varphi^{1}_{i})$.
The resulting pair $(\widehat{\varphi}^{0}, \bar{\varphi}^{1})$ determines a vertex in $K(M)_{k}$ with $F((\bar{\varphi}^{0}, \varphi^{1})) = f$ and so we set $u := (\widehat{\varphi}^{0}, \bar{\varphi}^{1})$.
By construction, $w_{i} \in \lk_{K(M)_{k}}(u)$ for $i = 1, \dots, m$ and $(v_{j}, u) \in \mathcal{T}$ for $j = 1, \dots, N$.
This proves that the map $F$ has the link lifting property relative to the relation $\mathcal{T}$. 

To finish the proof we need to establish that $F(\lk_{K(M)_{k}}(\zeta)) \leq \lk_{L(\pi^{\tau}_{2n}(M))_{k}}(F(\zeta))$ for any simplex $\zeta \in K(M)_{k}$. 
This follows immediately from the fact that if $\phi, \psi: V^{2n+1}_{k} \longrightarrow M$ are disjoint $\langle k \rangle$-embeddings, then $b([\phi_{\beta}], [\psi_{\beta}]) = 0$.
This concludes the proof of the lemma. 
\end{proof}

\subsection{A Modification of $K(M)_{k}$} \label{subsection: reconstruction}
Let $(\varphi^{0}, \varphi^{1})$ be a vertex of $K(M)_{k}$ and consider the subspace $Y_{k}(\varphi^{0}, \varphi^{1}) \subset M$. 
We will need to make a further modification of $Y_{k}(\varphi^{0}, \varphi^{1})$ as follows. 
\begin{Construction} \label{construction: killing pi-1}
Let $(\varphi^{0}, \varphi^{1})$ be as above. 
Since $2 < \dim(M)/2$, we may choose an embedding
\begin{equation} \label{equation: kill pi-1 embedding}
\xymatrix{
G: (\sqcup_{i=1}^{k-1}D^{2}_{i}, \; \sqcup_{i=1}^{k-1}S^{1}_{i}) \ar[rr] && (M, \; Y_{k}(\varphi^{0}, \varphi^{1}))
}
\end{equation}
which satisfies the following conditions:
\begin{enumerate}
\item[(a)] 
$$G(\sqcup_{i=1}^{k-1}\Int(D^{2}_{i}))\bigcap Y_{k}(\varphi^{0}, \varphi^{1}) \; = \; \emptyset.$$
\item[(b)] The maps 
$$G|_{S^{1}_{i}}: S^{1} \longrightarrow Y_{k}(\varphi^{0}, \varphi^{1}) \quad \quad  \text{for $i = 1, \dots, k-1$,}$$ 
represent a minimal set of generators for $\pi_{1}(Y_{k}(\varphi^{0}, \varphi^{1}))$, which by Proposition \ref{proposition: embedding of Y-k} is the free group on $k-1$ generators.  
\end{enumerate}
Given such an embedding $G$ as in (\ref{equation: kill pi-1 embedding}), we denote 
\begin{equation} \label{equation: union simply connected space}
Y^{G}_{k}(\varphi^{0}, \varphi^{1}) \; := \; Y_{k}(\varphi^{0}, \varphi^{1})\bigcup G(\sqcup_{i=1}^{k-1}D^{2}_{i}).
\end{equation}
It follows from conditions i. and ii. above that $Y^{G}_{k}(\varphi^{0}, \varphi^{1})$ is simply connected and that 
$$H_{s}(Y^{G}_{k}(\varphi^{0}, \varphi^{1}); \; \Z) = 
\begin{cases}
\Z/k\oplus\Z/k &\quad \text{if $s = 2n$,}\\
\Z &\quad \text{if $s = 0$,}\\
0 &\quad \text{else.}
\end{cases}
$$
It follows that $Y^{G}_{k}(\varphi^{0}, \varphi^{1})$ has the homotopy type of the Moore-space $M(\Z/k\oplus\Z/k, 2n)$ and hence is homotopy equivalent to the manifold $W'_{k}$. 
We will think of $Y^{G}_{k}(\varphi^{0}, \varphi^{1}) \hookrightarrow M$ as being a choice of embedding of the $(2n+1)$-skeleton of $W'_{k}$ into $M$. 
\end{Construction}

Using the construction given above, we define a modification of the simplicial complex $K(M)_{k}$. 
Let $M$ be a $(4n+1)$-dimensional manifold with non-empty boundary. 
Let 
$$a: [0,\infty)\times \R^{4n} \longrightarrow M$$
be an embedding with $a^{-1}(\partial M) = \{0\}\times\R^{4n}$.  
\begin{defn} \label{defn: the embedding complex} 
Let $\bar{K}(M, a)_{k}$ be the simplicial complex whose vertices are given by $4$-tuples $(\varphi, G, \gamma, t)$ which satisfy the following conditions:
\begin{enumerate} \itemsep.2cm
\item[i.] $\varphi = (\varphi^{0}, \varphi^{1})$ is a vertex in $K(M)_{k}$.
\item[ii.] $\xymatrix{
G: (\sqcup_{i=1}^{k-1}D^{2}_{i}, \; \sqcup_{i=1}^{k-1}S^{1}_{i}) \ar[r] & (M, \; Y_{k}(\varphi^{0}, \varphi^{1}) )
}$ is an embedding as in Construction (\ref{construction: killing pi-1}).
\item[iii.] $t$ is a real number.
\item[iv.] $\gamma: [0,1] \longrightarrow M$ is an embedded path which satisfies:
\begin{enumerate} \itemsep6pt
\vspace{.2cm}
\item[(a)] $\gamma^{-1}(Y^{G}_{k}(\varphi^{0}, \varphi^{1})) = \{1\}$,
\item[(b)] there exists $\epsilon > 0$ such that for
$s \in [0, \epsilon)$, 
the equality 
$$\gamma(s) = a(s,  te_{1}) \in [0,1]\times \R^{4n}$$ 
is satisfied,
where $e_{1} \in \R^{4n}$ denotes the first basis vector.
\end{enumerate}
\end{enumerate}
A set of vertices $\{(\varphi_{0}, G_{0}, \gamma_{0}, t_{0}), \dots, (\varphi_{p}, G_{p}, \gamma_{p}, t_{p})\}$ forms a $p$-simplex if and only if
$$\bigg(\gamma_{i}([0,1])\cup Y_{k}^{G_{i}}(\varphi^{0}_{i}, \varphi^{1}_{i})\bigg)\bigcap\bigg(\gamma_{j}([0,1])\cup Y_{k}^{G_{j}}(\varphi^{0}_{j}, \varphi^{1}_{j})\bigg) = \emptyset \quad \text{whenever $i \neq j$.}$$
\end{defn}
There is a simplicial map 
\begin{equation} \label{equation: K to K-bar}
\bar{F}: \bar{K}(M, a)_{k} \longrightarrow K(M)_{k}, \quad (\varphi, G, \gamma, t) \mapsto \varphi.
\end{equation}
\begin{proposition} \label{proposition: highly connected 2}
Let $n, k \geq 2$ be integers with $k$ odd. 
Let $M$ be a $2$-connected, manifold of dimension $4n+1$ and let
$g \in \N$ be such that $r_{k}(M) \geq g$. 
Then the geometric realization $|\bar{K}(M)_{k}|$ is $\frac{1}{2}(g - 4)$-connected and $lCM(\bar{K}(M)_{k}) \geq \frac{1}{2}(g - 1)$.
\end{proposition}
\begin{proof}
The proposition is proven by essentially the same method as Lemma \ref{thm: high connectedness of K}. 
We will apply Corollary \ref{corollary: link lifting no relation} to the map $\bar{F}$ from (\ref{equation: K to K-bar}). 
To do this we need to verify that the map $\bar{F}$ has the link lifting property (Definition \ref{defn: link lifting property 1}) and that it preserves links (condition (iii) in the statement of Corollary \ref{corollary: link lifting no relation}). 
Since $|K(M)_{k}|$ is $\frac{1}{2}(g - 4)$-connected and $lCM(K(M)_{k}) \geq \frac{1}{2}(g - 1)$, we then may apply Lemma \ref{lemma: link lift lemma} to deduce the claim of the proposition.

Let $\varphi = (\varphi^{0}, \varphi^{1})$ be a vertex in $K(M)_{k}$.
Let $(\varphi_{1}, G_{1}, \gamma_{1}, t_{1}), \dots, (\varphi_{m}, G_{m}, \gamma_{m}, t_{m})$ be vertices in $\bar{K}(M, a)_{k}$ such that $\varphi_{i}$ is adjacent to $\varphi$ for $i = 1, \dots, m$.
Since $\dim(M)/2 > 2$, there is no obstruction to choosing an embedding $G$ as in Construction \ref{construction: killing pi-1} (with respect to $\varphi$ so as to construct $Y^{G}_{k}(\varphi^{0}, \varphi^{1})$) so that the image of $G$ is disjoint from the images of $G_{i}$ and $\gamma_{i}$ for all $i$. 
Furthermore, with $G$ chosen, we may then choose an embedded path $\gamma: [0,1] \longrightarrow M$, connecting $Y^{G}_{k}(\varphi^{0}, \varphi^{1})$ to $\partial M$ so as to yield a vertex $(\varphi, G, \gamma, t) \in \bar{K}(M, a)_{k}$, which maps to $\varphi$ under $\bar{F}$ and is adjacent to $(\varphi_{i}, G_{i}, \gamma_{i}, t_{i})$ for all $i$.
This proves the fact that $\bar{F}$ has the link lifting property. 
The fact that $\bar{F}$ preserves links is immediate from the definition of $\bar{F}$. 
This concludes the proof of the proposition.
\end{proof}

\subsection{Reconstructing embeddings}
Let $(\varphi, G, \gamma, t)$ be a vertex in $\bar{K}(M, a)_{k}$. 
We will need to consider smooth regular neighborhoods of the subspace
$Y^{G}_{k}(\varphi^{0}, \varphi^{1})\cup\gamma([0,1]) \; \subset \; M.$
The following lemma identifies the diffeomorphism type of such a regular neighborhood. 

\begin{lemma} \label{lemma: thickening}
Let $(\varphi, G, \gamma, t)$ be a vertex in $\bar{K}(M, a)_{k}$. 
If $k$ is odd then any closed regular neighborhood $U$ of the subspace 
$Y^{G}_{k}(\varphi^{0}, \varphi^{1})\cup\gamma([0,1]) \; \subset \; M$,
is diffeomorphic to the manifold $\bar{W}_{k} = W'_{k}\cup_{\alpha}([0, 1]\times D^{4n})$. 
\end{lemma}
\begin{proof}
By definition of regular neighborhood, the inclusion map $Y^{G}_{k}(\varphi^{0}, \varphi^{1}) \hookrightarrow U$ is a homotopy equivalence ($U$ collapses to $Y^{G}_{k}(\varphi^{0}, \varphi^{1})$, see \cite{Hi 62}). 
The maps $\varphi^{0}_{\beta}, \varphi^{1}_{\beta}: S^{2n} \longrightarrow U$ represent generators for $\pi_{2n}(U)$ and since $\varphi^{0}\pitchfork\varphi^{1} \cong A_{k}$, it follows that 
$$b([\varphi^{0}_{\beta}], \; [\varphi^{1}_{\beta}]) = \tfrac{1}{k} \mod 1$$ 
and hence, the linking form $(\pi^{\tau}_{2n}(U), b)$ is isomorphic to $\mb{W}_{k}$. 
It follows from Constructions \ref{construction: pushout space} and \ref{construction: killing pi-1} that the regular neighborhood $U$ is $(2n-1)$-connected. 
Now, $U$ is homotopy equivalent to the Moore-space $M(\Z/k\oplus\Z/k, 2n)$, and so the set of isomorphism classes of $(4n+1)$-dimensional vector bundles over $U$ is in bijective correspondence with the set $[M(\Z/k\oplus\Z/k, 2n), \; BSO]$. 
Since $\pi_{2n}(BSO; \Z/k) = 0$ whenever $k$ is odd, it follows that the tangent bundle $TU \rightarrow U$ is trivial and thus
$U$ is parallelizable.
We will show that the boundary $\partial U$ is diffeomorphic to $S^{4n}$. 
Once this is demonstrated, it will follow from the classification theorem, Theorem \ref{thm: classification theorem}, that $U$ is diffeomorphic to the manifold $W'_{k}$. 

Since $U$ is parallelizable, by \cite[Theorem 5.1]{KM 62} it will be enough to show that $\partial U$ is homotopy equivalent to $S^{4n}$. 
From Constructions \ref{construction: pushout space}, \ref{construction: killing pi-1} and the \textit{Universal Coefficient Theorem}, we have 
$$H^{s}(U; \Z) \cong 
\begin{cases}
\Z/k\oplus\Z/k &\quad \text{if $s = 2n+1$,}\\
\Z &\quad \text{if $s = 0$,}\\
0 &\quad \text{else.}
\end{cases}
$$
Using \textit{Lefschetz Duality} it then follows that 
$$H_{s}(U, \partial U; \Z) \cong 
\begin{cases}
\Z/k\oplus\Z/k &\quad \text{if $s = 2n$,}\\
0 &\quad \text{else.}
\end{cases}
$$
Consider the long exact sequence on homology associated to $(U, \partial U)$.
It follows immediately that $\partial U$ is $(2n-2)$-connected and that the long exact sequence reduces to
\begin{equation} \label{equation: U, dU exact seq.}
\xymatrix{
0 \ar[r] & H_{2n}(\partial U; \Z) \ar[r] & H_{2n}(U; \Z) \ar[r] & H_{2n}(U, \partial U; \Z) \ar[r] & H_{2n-1}(\partial U; \Z) \ar[r] & 0.
}
\end{equation}
We claim that the map $H_{2n}(U, \Z) \rightarrow H_{2n}(U, \partial U; \Z)$ is an isomorphism. 
To see this, consider the commutative diagram
$$\xymatrix{
H_{2n}(U; \Z) \ar[rrr]_{\cong}^{x \mapsto b(x, \underline{\hspace{.3cm}})} \ar[d] &&& H^{2n}(U; \Q/\Z) \ar[d]^{\cong} \\
H_{2n}(U, \partial U; \Z) \ar[rrr]^{\cong} &&& H^{2n+1}(U; \Z).
}$$
In the above diagram the bottom-horizontal map is the Leftschetz duality isomorphism, the right vertical map is the boundary homomorphism in the Bockstein exact sequence (which in this case is an isomorphism), and 
the top-horizontal map $x \mapsto b(x, \underline{\hspace{.3cm}})$ is an isomorphism since the homological linking form $(H_{2n}(U), \; b)$ is non-singular. 
It follows that the map 
$H_{2n}(U; \Z) \rightarrow H_{2n}(U, \partial U; \Z)$ 
is indeed an isomorphism and it then follows from the exact sequence of (\ref{equation: U, dU exact seq.}) that $\partial U$ has the same homology type of $S^{4n}$. 

To prove that $\partial U$ has the same homotopy type of $S^{4n}$, we must show that $\partial U$ is simply connected.
To do this it will suffice to show that $\pi_{i}(U, \partial U) = 0$ for $i = 1, 2$. 
With this established, the simple connectivity of $\partial U$ will follow by considering the long exact sequence of homotopy groups associated to $(U, \partial U)$, using the fact that $U$ is simply connected.
Let $f: (D^{i}, \partial D^{i}) \longrightarrow (U, \partial U)$ be a map with $i = 1, 2$. 
Since 
$$\dim(U) - \dim(Y^{G}_{k}(\varphi^{0}, \varphi^{1})) \geq 3,$$ 
we may deform $f$ so that its image is disjoint from $Y^{G}_{k}(\varphi^{0}, \varphi^{1})$. 
We then may find another (strictly smaller) regular neighborhood $U'$ of $Y^{G}_{k}(\varphi^{0}, \varphi^{1})$ such that $U' \subsetneq U$ and $f(D^{i}) \subset U\setminus U'$. 
The class $[f] \in \pi_{i}(U, \partial U)$ is in the image of the map 
$$ \pi_{i}(U\setminus \Int(U'),\; \partial U) \longrightarrow \pi_{i}(U,\; \partial U)$$
induced by inclusion. 
Using the uniqueness theorem for smooth regular neighborhoods (see \cite{Hi 62}), it follows that the manifold $U\setminus\Int(U')$ is an $H$-cobordism from $\partial U$ to $\partial U'$ and so it follows that$\pi_{i}(U\setminus \Int(U'), \partial U) = 0$. 
This proves that $[f] = 0$ and thus $\pi_{i}(U, \partial U) = 0$ since $f$ was arbitrary. 
It follows by considering the exact sequence on homotopy groups associated to the pair $(U, \partial U)$ that $\partial U$ is simply connected. 

Since $\partial U$ is simply connected and has the homology type of a sphere, it follows that $\partial U$ is a homotopy sphere. 
It then follows from \cite[Theorem 5.1]{KM 62} that $\partial U$ is diffeomorphic to $S^{4n}$ since $\partial U$ bounds a parallelizable manifold, namely $U$. 
This concludes the proof of the lemma. 
\end{proof}

We now define a new simplicial complex. 
\begin{defn} \label{defn: thickening complex}
Let $\widehat{K}(M, a)_{k}$ be the simplicial complex whose vertices are given by triples
$(\bar{\varphi}, \Psi, s)$
which satisfy the following conditions:
\begin{enumerate} \itemsep.2cm
\item[(i)] The $4$-tuple $\bar{\varphi} = (\varphi, G, \gamma, t)$ is a vertex in $\bar{K}(M, a)_{k}$.
\item[(ii)] $s$ is a real number. 
\item[(iii)] $\Psi: \bar{W}_{k}\times[s, \infty) \longrightarrow M$ is a smooth family of embeddings $\bar{W}_{k} \hookrightarrow M$ that satisfies the following:
\begin{enumerate} \itemsep.2cm
\vspace{.15cm}
\item[(a)] for each $t \in [s, \infty)$, the embedding $\Psi(\underline{\hspace{.3cm}}, \; t): \bar{W}_{k} \longrightarrow M$ is an element of $X_{0}(M, a)_{k}$, 
\item[(b)] $Y^{G}(\varphi^{0}, \varphi^{1})\cup\gamma([0,1]) \; \subset \; \Psi(\bar{W}_{k}, t)$ for all $t \in [s, \infty)$, 
\item[(c)] for any neighborhood $U$ of $Y^{G}(\varphi^{0}, \varphi^{1})\cup\gamma([0,1])$, there is $t_{U} \in [s, \infty)$ such that 
$\Psi(\bar{W}_{k}, t) \subset U$ when $t \geq t_{U}$.
\end{enumerate}
\end{enumerate}
A set of vertices 
$\{(\bar{\varphi_{0}}, \Psi_{0}, s_{0}), \dots, (\bar{\varphi}_{p}, \Psi_{p}, s_{p})\}$
forms a $p$-simplex if the associated set 
$\{\bar{\varphi}_{0}, \dots, \bar{\varphi}_{p}\}$
is a $p$-simplex in the complex $\bar{K}(M, a)_{k}$ (no extra pairwise condition on the $\Psi_{i}$ and $s_{i}$ are required).
\end{defn}

By construction of $\widehat{K}(M, a)_{k}$, there is a simplicial map,
\begin{equation} \label{equation: hat to bar map}
\widehat{F}: \widehat{K}(M, a)_{k} \longrightarrow \bar{K}(M, a)_{k}, \quad (\bar{\varphi}, \Psi, s) \mapsto \bar{\varphi}.
\end{equation}

\begin{proposition} \label{highly connected 3}
Let $n, k \geq 2$ be integers with $k$ odd. 
Let $M$ be a compact, $2$-connected, manifold of dimension $4n+1$.   
Let $g \in \N$ be such that $r_{k}(M) \geq g$. 
Then the geometric realization $|\widehat{K}(M)_{k}|$ is $\frac{1}{2}(g - 4)$-connected and $lCM(\widehat{K}(M)_{k}) \geq \frac{1}{2}(g - 1)$.
\end{proposition}
\begin{proof}
The proof of this proposition again follows the same strategy as Lemma \ref{thm: high connectedness of K}.
We will apply Corollary \ref{corollary: link lifting no relation} to the map $F$.
We will need to verify that the map $\widehat{F}$ has the linking lifting property as in Definition \ref{defn: link lifting property 1} and that $F$ preserves links. 

The fact that $\widehat{F}$ preserves links follows immediately from the definition. 
We will verify the link lifting property.
Let $\bar{\varphi} = (\varphi, G, \gamma, t)$ be a vertex of $\bar{K}(M)_{k}$.
Let  
$$(\bar{\varphi}_{1}, \Psi_{1}, s_{1}), \dots, (\bar{\varphi}_{m}, \Psi_{m}, s_{m})$$ 
be a collection of vertices in $\widehat{K}(M)_{k}$ such that $\bar{\varphi}$ is adjacent to $\bar{\varphi}_{i}$ in $\bar{K}(M)_{k}$ for $i = 1, \dots, m$.
We will denote
\begin{equation} \label{equation: new notation for complex}
Y_{k}(\bar{\varphi}) := Y^{G}_{k}(\varphi^{0}, \varphi^{1})\cup\gamma([0, 1]).
\end{equation}
Let $U \subset M$ be a regular neighborhood of $Y_{k}(\bar{\varphi})$. 
Since $U$ collapses to $Y_{k}(\bar{\varphi})$ (by definition of regular neighborhood), we may choose a one-parameter family of embeddings:
\begin{equation} \label{equation: compression isotopy}
\rho: U\times[s,\infty) \longrightarrow U
\end{equation}
which satisfies the following:
\begin{enumerate}
\item[i.] For all $t \in [s, \infty)$, the embedding $\rho_{t} = \rho|_{U\times\{t\}}: U \rightarrow U$ is the identity on $Y_{k}(\bar{\varphi})$.
\item[ii.] Given any neighborhood $U' \subset U$ of $Y_{k}(\bar{\varphi})$, there exists $t' > s$ such that $\rho_{t}(U) \subset U'$ for all $t \geq t'$. 
\end{enumerate}
We call such an isotopy a \textit{compression isotopy} of $U$ to $Y_{k}(\bar{\varphi})$.
By Lemma \ref{lemma: thickening}, there exists a diffeomorphism $\Psi: \bar{W}_{k} \stackrel{\cong} \longrightarrow U$ such that the composition 
$\bar{W}_{k} \stackrel{\Psi} \longrightarrow U \hookrightarrow M$
satisfies the conditions of Definition \ref{defn: the embedding complex}. 
It then follows that the triple $(\bar{\varphi}, \; \Psi\circ\rho, \; s)$ is a vertex of $\widehat{K}(M, a)_{k}$ that maps to $\bar{\varphi}$ under $\widehat{F}$.
It follows from the definition of $\widehat{K}(M)_{k}$ that $(\bar{\varphi}, \; \Psi\circ\rho, \; s)$ is automatically adjacent to $(\bar{\varphi}_{i}, \Psi_{i}, s_{i})$ for $i = 1, \dots, m$.
This proves that $\widehat{F}$ has the link lifting property. 
This completes the proof of the proposition.
\end{proof}

\subsection{Comparison with $X_{\bullet}(M, a)_{k}$.}
We are now in a position to finally prove Theorem \ref{theorem: high connectivity of X} by comparing $|X_{\bullet}(M, a)_{k}|$ to $|\widehat{K}(M, a)_{k}|$. 
We will need to construct an auxiliary semi-simplicial space related to the simplicial complex $\widehat{K}(M, a)_{k}$.
Let $M$ be a $(4n+1)$-dimensional manifold with non-empty boundary and let $a: [0,\infty)\times\R^{4n} \longrightarrow M$ be an embedding as used in Definition \ref{defn: the embedding complex}. 
We define two semi-simplicial spaces $\widehat{K}_{\bullet}(M, a)_{k}$ and $\widehat{K}'_{\bullet}(M, a)_{k}$. 
\begin{defn} \label{defn: aux semi-simplicial spaces}
The space of $p$-simplices $\widehat{K}_{p}(M, a)_{k}$ is defined as follows:
\begin{enumerate}
\item[i.] The space of $0$-simplices $\widehat{K}_{0}(M, a)_{k}$ is defined to have the same underlying set as the set of vertices of the simplicial complex $\widehat{K}(M, a)_{k}$. 
\item[ii.] The space of $p$-simplices $\widehat{K}_{p}(M, a)_{k} \subset (\widehat{K}_{0}(M, a)_{k})^{\times (p+1)}$ consists of the ordered $(p+1)$-tuples $((\bar{\varphi}_{0}, \Psi_{0}, s_{0}), \cdots, (\bar{\varphi}_{p}, \Psi_{p}, s_{p}))$ such that 
the associated unordered set 
$$\{(\bar{\varphi}_{0}, \Psi_{0}, s_{0}), \cdots, (\bar{\varphi}_{p}, \Psi_{p}, s_{p})\}$$ 
is a $p$-simplex in the simplicial complex $\widehat{K}(M, a)_{k}$.
\end{enumerate}
The spaces $\widehat{K}_{p}(M, a)_{k} $ are topologized using the $C^{\infty}$-topology on the spaces of embeddings.
The assignments $[p] \mapsto \widehat{K}_{p}(M, a)_{k}$ define a semi-simplicial space which we denote by $\widehat{K}_{\bullet}(M, a)_{k}$. 

Finally, $\widehat{K}'_{\bullet}(M, a)_{k} \subset \widehat{K}_{\bullet}(M, a)_{k}$ is defined to be the sub-semi-simplicial space consisting of all $(p+1)$-tuples 
$((\bar{\varphi}_{0}, \Psi_{0}, s_{0}), \cdots, (\bar{\varphi}_{p}, \Psi_{p}, s_{p}))$ 
such that $\Psi_{i}(\bar{W}_{k})\cap\Psi_{j}(\bar{W}_{k}) = \emptyset$ whenever $i\neq j$. 
\end{defn}
It is easily verified that both $\widehat{K}_{\bullet}(M, a)_{k}$ and $\widehat{K}'_{\bullet}(M, a)_{k}$ are topological flag complexes.

\begin{proposition} \label{proposition: highly connected semi-space}
Let $k, n \geq 2$ be integers with $k$ odd.  Let $M$ be a $2$-connected $(4n+1)$-dimensional manifold and let $g \geq 0$ be such that $r_{k}(M) \geq g$. 
Then the geometric realization $|\widehat{K}_{\bullet}(M, a)_{k}|$ is $\frac{1}{2}(g - 4)$-connected. 
\end{proposition}
\begin{proof}
Let $\widehat{K}_{\bullet}(M, a)^{\delta}_{k}$ denote the the \textit{discretization} of $\widehat{K}_{\bullet}(M, a)_{k}$ as defined in Definition \ref{defn: associated simplicial complex}. 
Consider the map
\begin{equation} \label{eq: forget ordering}
|\widehat{K}_{\bullet}(M, a)^{\delta}_{k}| \longrightarrow |\widehat{K}(M, a)_{k}|
\end{equation}
induced by sending an ordered list 
$((\bar{\varphi}_{0}, \Psi_{0}, s_{0}), \cdots, (\bar{\varphi}_{p}, \Psi_{p}, s_{p}))$ 
to its associated underlying set. 
For any such set $\{(\bar{\varphi}_{0}, \Psi_{0}, s_{0}), \cdots, (\bar{\varphi}_{p}, \Psi_{p}, s_{p})\}$ which forms a $p$-simplex in $\widehat{K}(M, a)^{\delta}_{k}$, 
there is only one possible ordering on it which yields an element of $\widehat{K}_{\bullet}(M, a)^{\delta}_{k}$. 
Thus the map (\ref{eq: forget ordering}) is a homeomorphism. 
By Proposition \ref{highly connected 3}, it follows that $\widehat{K}_{\bullet}(M, a)^{\delta}_{k}$ (which is clearly a topological flag-complex) is weakly Cohen-Macaulay of dimension $\frac{1}{2}(g - 2)$, as defined in Definition \ref{defn: semi-simplicial cohen mac}.
 It then follows from Theorem \ref{theorem: discretization} that $|\widehat{K}_{\bullet}(M, a)_{k}|$ is $\frac{1}{2}(g - 4)$-connected.
\end{proof}

We now consider the inclusion map $\widehat{K}'_{\bullet}(M, a)_{k} \longrightarrow \widehat{K}_{\bullet}(M, a)_{k}$. 
\begin{proposition}
For any $(4n+1)$-dimensional manifold $M$ with non-empty boundary, the map $|\widehat{K}'_{\bullet}(M, a)_{k}| \longrightarrow |\widehat{K}_{\bullet}(M, a)_{k}|$ induced by inclusion is a weak homotopy equivalence. 
\end{proposition}
\begin{proof}
For $p \geq 0$, let 
\begin{equation} \label{equation: relative homotopy element}
x \mapsto ((\bar{\varphi}^{x}_{0}, \Psi^{x}_{0}, s^{x}_{0}), \cdots, (\bar{\varphi}^{x}_{p}, \Psi^{x}_{p}, s^{x}_{p})) \quad \text{for $x \in D^{j}$}
\end{equation}
represent an element of the relative homotopy group 
\begin{equation} \label{equation: relative homotopy group}
\pi_{j}\bigg(\widehat{K}_{p}(M, a)_{k}, \; \widehat{K}'_{p}(M, a)_{k}\bigg) = 0.
\end{equation}
For each $x$, 
$Y_{k}(\bar{\varphi}^{x}_{i})\bigcap Y_{k}(\bar{\varphi}^{x}_{j}) = \emptyset \quad \text{whenever $i \neq j$.}$
Using condition (c) in Definition \ref{defn: thickening complex}, since $D^{j}$ is compact we may choose a real number $s \geq \max\{s^{x}_{i} \; | \; i = 0, \dots, p, \; \text{and} \; x \in D^{j}\}$, such that for any $x \in D^{j}$,
$$\Psi^{x}_{i}(\bar{W}_{k}, \;  t)\cap \Psi^{x}_{j}(\bar{W}_{k}, \;  t) = \emptyset \quad \text{whenever $t \geq s$ and $i \neq j$.}$$ 
For each $x \in D^{j}$, $t \in [0,1]$, and $i = 0, \dots, p$, let $s^{x}_{i}(t)$ denote the real number given by the sum 
$$(1-t)\cdot s^{x}_{i} + t\cdot s$$ 
and let 
$\Psi^{x}_{i}(t)$ denote the restriction of $\Psi^{x}_{i}$ to 
$\bar{W}_{k}\times[s^{x}_{i}(t), \; \infty).$
The formula,
$$(x, t) \; \mapsto \; ((\bar{\varphi}^{x}_{0}, \; \Psi^{x}_{0}(t),\; s^{x}_{0}(t)), \;  \cdots, \; (\bar{\varphi}^{x}_{p}, \; \Psi^{x}_{p}(t), \; s^{x}_{p}(t))) \quad \quad \text{for $t \in [0, 1]$}$$
yields a homotopy from the map defined in (\ref{equation: relative homotopy element}) to a map which represents the trivial element in the relative homotopy group (\ref{equation: relative homotopy group}).
This implies that for all $p, j \geq 0$, the relative homotopy group (\ref{equation: relative homotopy group}) is trivial and thus the inclusion $\widehat{K}'_{p}(M, a)_{k} \longrightarrow \widehat{K}_{p}(M, a)_{k}$ is a weak homotopy equivalence for all $p$.
It follows that the induced map $|\widehat{K}'_{\bullet}(M, a)_{k}| \longrightarrow |\widehat{K}_{\bullet}(M, a)_{k}|$ is a weak homotopy equivalence. 
\end{proof}

Finally, we consider the map 
\begin{equation} \label{equation: compare to X}
\widehat{K}'_{\bullet}(M, a)_{k} \longrightarrow X_{\bullet}(M, a)_{k}, \quad (\bar{\varphi}, \Psi, s) \mapsto \Psi_{s} = \Psi|_{\bar{W}_{k}\times\{s\}}.
\end{equation}
The following proposition implies Theorem \ref{theorem: high connectivity of X}.
\begin{proposition} \label{proposition: compare to X}
Let $n \geq 2$ and suppose that $k > 2$ is an odd integer. 
Then for any $(4n+1)$-dimensional manifold $M$ with non-empty boundary, the degree of connectivity of $|X_{\bullet}(M, a)_{k}|$ is bounded below by the degree of connectivity of $|\widehat{K}'_{\bullet}(M, a)_{k}|$. 
\end{proposition}
\begin{proof}
To prove the proposition it will suffice to construct a section of the map (\ref{equation: compare to X}). 
The existence of such a section implies that the map on homotopy groups induced by (\ref{equation: compare to X}) is a surjection. 
The result then follows. 
Let $x, y \in \pi_{2n}^{\tau}(\bar{W}_{k})$ be two generators such that $b(x, y) = \frac{1}{k} \mod 1$. 
By combining Corollary \ref{lemma: represent by $k$-embedding} and Corollary \ref{corollary: intersection at A-k}, we may choose $\langle k \rangle$-embeddings $\varphi^{0}, \varphi^{1}: V^{2n+1}_{k} \longrightarrow \bar{W}_{k}$ such that 
$$[\varphi^{0}_{\beta}] = x, \quad  [\varphi^{1}_{\beta}] = y, \quad \text{and} \quad  \varphi^{0}\pitchfork\varphi^{1} \cong A_{k}.$$
We then may apply Construction \ref{construction: killing pi-1} to obtain a vertex $\bar{\varphi} = (\varphi, G, \gamma, t) \in \bar{K}(\bar{W}_{k}, a)_{k}$.
Now, the whole manifold $\bar{W}_{k}$ is a regular neighborhood for $Y_{k}(\bar{\varphi})$. 
We may choose a compression isotopy 
$\rho: \bar{W}_{k}\times[0, \infty) \longrightarrow \bar{W}_{k}$ 
of $\bar{W}_{k}$ to $Y_{k}(\bar{\varphi})$
as in (\ref{equation: compression isotopy}) and which satisfies the same conditions associated to the isotopy (\ref{equation: compression isotopy}). 
It follows that $(\bar{\varphi}, \rho, 0)$ is an element of $\widehat{K}'_{0}(\bar{W}_{k}, a)_{k}$. 
Using $\bar{\varphi}$ and the compression isotopy $\rho$, we then define a semi-simplicial map 
\begin{equation}
X_{\bullet}(M, a)_{k} \longrightarrow \widehat{K}'_{\bullet}(M, a)_{k}, \quad \Psi \mapsto (\Psi\circ \bar{\varphi}, \; \Psi\circ \rho, \; 0),
\end{equation}
where $\Psi\circ \bar{\varphi}$ is the vertex in $\bar{K}(M, a)_{k}$ given by the $4$-tuple, 
$((\Psi\circ\varphi^{0}, \; \Psi\circ\varphi^{1}), \; \; \Psi\circ G, \; \; \Psi\circ\gamma, \; t).$
It follows that this map is a section of (\ref{equation: compare to X}).
\end{proof}

%%%%%%%%%%%%%%%%%%%%%%%%%%%%%%%%%%%%%%%%%%%%%%%%%%%%%%%%%%%%%%%%%%%%%%%%%%%%%%%%%%%%%%%%%%%%%%%%%%%%%%%%%%%%%%%%%%%%%%%%%%%%%%%%%%%%%%%%%%%%%%%%%%%%%%%%%%%%%%%%%%%%%%%%%%%%%%%%%%%%%%%%%%%%%%%%%%%%%%%%%%%%%%%%%%%%%%%%%%%%%%%%%%%%%%%%%%%%%%%%%%%%%%%%%%%%%%%%%%%%%%%%%%%%%%%%%%%%%%%%%%%%%%%%%%%%%%%%%%%%%%%%%%%%%%%%%%%

\section{Homological Stability} \label{Homological Stability}
\label{A model for the for the classifying spaces.}
With our main technical result Theorem \ref{theorem: high connectivity of X} established, in this section we show how Theorem \ref{theorem: high connectivity of X} implies the main result of the paper which is Theorem \ref{theorem: Main theorem}.

\subsection{A Model for $\BDiff^{\partial}(M)$.} \label{subsection: semi-simplicial resolution}
Let $M$ be a compact manifold of dimension $m$ with non-empty boundary.  
We now construct a concrete model for $\BDiff^{\partial}(M)$. 
Fix a collar embedding,
 $$h: [0,\infty)\times\partial M \longrightarrow M$$
 with $h^{-1}(\partial M) = \{0\}\times\partial M$.
 Fix once and for all an embedding, 
 $\theta: \partial M \longrightarrow \R^{\infty}$
 and let $S$ denote the submanifold $\theta(\partial M) \subset \R^{\infty}$.
\begin{defn} \label{defn: moduli M} 
 We define
$\mathcal{M}(M)$ to be the set of compact $m$-dimensional submanifolds 
$M' \subset [0,\infty)\times\R^{\infty}$
that satisfy:
\begin{enumerate}
\item[i.] $M'\cap(\{0\}\times\R^{\infty}) = S$ and $M'$ contains $[0, \epsilon)\times S$ for some $\epsilon > 0$.
\item[ii.] The boundary of $M'$ is precisely $\{0\}\times S$.
\item[iii.] $M'$ is diffeomorphic to $M$ relative to $S$. 
\end{enumerate}
Denote by $\mathcal{E}(M)$ the space of embeddings $\psi: M \rightarrow [0,\infty)\times\R^{\infty}$
for which there exists $\epsilon > 0$ such that $\psi\circ h(t, x) = (t, \theta(x))$ for all $(t, x) \in [0,\epsilon)\times\partial M$. 
The space $\mathcal{M}(M)$ is topologized as a quotient of the space $\mathcal{E}(M)$ where two embeddings are identified if they have the same image. 
\end{defn}

It follows from Definition \ref{defn: moduli M} that $\mathcal{M}(M)$ is equal to the orbit space, $\mathcal{E}(M)/\Diff^{\partial}(M)$.
By the main result of \cite{BF 81},  the quotient map, $\mathcal{E}(M) \longrightarrow \mathcal{E}(M)/\Diff^{\partial}(M) = \mathcal{M}(M)$
is a locally trivial fibre-bundle. This together with the fact that $\mathcal{E}(M)$ is weakly contractible implies that there is a weak-homotopy equivalence,
$
\mathcal{M}(M) \simeq \BDiff^{\partial}(M).
$

Now suppose that $m = 4n+1$ with $n \geq 2$. 
Let $k \geq 2$ be an integer. 
Recall from Section \ref{section: introduction} the manifold $\widetilde{W}_{k}$, given by forming the connected sum of $[0,1]\times\partial M$ with $W_{k}$. 
Choose a collared embedding 
$\alpha: \widetilde{W}_{k} \longrightarrow [0,1]\times\R^{\infty}$
such that
for $(i, x) \in \{0,1\}\times\partial M \subset \widetilde{W}_{k}$, the equation $\alpha(i, x) = (i, \theta(x))$ is satisfied.
For any submanifold $M' \subset [0,\infty)\times\R^{\infty}$, denote by 
$M' + e_{1} \subset [1,\infty)\times\R^{\infty}$
the submanifold obtained by linearly translating $M'$ over $1$-unit in the first coordinate. Then for $M' \in \mathcal{M}(M)$, the submanifold $\alpha(\widetilde{W}_{k}) \cup (M'\cup e_{1}) \subset [0,\infty)\times\R^{\infty}$ is an element of $\mathcal{M}(M\cup_{\partial M}\widetilde{W}_{k})$. Thus, we have a continuous map,
\begin{equation} \label{ref: k-stabilization map}
s_{k}: \mathcal{M}(M) \longrightarrow \mathcal{M}(M\cup_{\partial M}\widetilde{W}_{k}); \quad V \mapsto \alpha(\widetilde{W}_{k})\cup(V + e_{1}).
\end{equation}
We will refer to this map as the \textit{$k$th-stabilization map.} 
\begin{remark} \label{remark: stabilization map}
The construction of the stabilization map $s_{k}$ depends on the choice of embedding $\alpha: \tilde{W}_{k} \rightarrow [0,1]\times\R^{\infty}$. 
However, any two such embeddings are isotopic (the space of all such embeddings is weakly contractible). 
It follows that the homotopy class of $s_{k}$ does not depend on any of the choices made. 
In this way, the manifold $\tilde{W}_{k}$ determines a unique homotopy class of maps $\BDiff^{\partial}(M) \longrightarrow \BDiff^{\partial}(M\cup_{\partial M}\tilde{W}_{k})$ which is in the same homotopy class as the map (\ref{eq: k-stabilization map}) used in the statement of Theorem \ref{theorem: Main theorem}.  
\end{remark}

\subsection{A Semi-Simplicial Resolution}
Let $M$ be as in Section \ref{subsection: semi-simplicial resolution}. 
For each positive integer $K$, we construct a semi-simplicial space $Z_{\bullet}(M)_{k}$, equipped with an augmentation 
$\epsilon_{k}: Z_{\bullet}(M)_{k} \longrightarrow \mathcal{M}(M)$
such that the induced map $|Z_{\bullet}(M)_{k}| \longrightarrow \mathcal{M}(M)$ is highly connected. 
Such an augmented semi-simplicial space is called a \textit{semi-simplicial resolution}. 

Let $\theta: \partial M \hookrightarrow \R^{\infty}$ be the embedding used in the construction of $\mathcal{M}(M)$.
Pick once and for all a coordinate patch $c_{0}: \R^{m-1} \longrightarrow S = \theta(\partial M)$. This choice of coordinate patch induces for any $M' \in \mathcal{M}(M)$, a germ of an embedding
$
[0,1)\times\R^{m-1} \longrightarrow M'
$
as used in the construction of  the semi-simplicial space $\bar{K}_{\bullet}(M')_{k}$ from Definition \ref{defn: the embedding complex 1}. 
\begin{defn} \label{defn: resolution of moduli}
For each non-negative integer $l$, let $Z_{l}(M)_{k}$ be the set of pairs $(M', \bar{\phi})$ where $M' \in \mathcal{M}(M)$ and $\bar{\phi} \in X_{l}(M')_{k}$, where $X_{l}(M')_{k}$ is defined using the embedding germ 
$$[0,1)\times\R^{m-1} \longrightarrow M'$$ 
induced by the chosen coordinate patch $c_{0}: \R^{m-1} \longrightarrow S$. The space $Z_{l}(M)_{k}$ is topologized as the quotient, $Z_{l}(M)_{k} = (\mathcal{E}(M)\times X_{l}(M)_{k})/\Diff^{\partial}(M)$.
The assignments $[l] \mapsto Z_{l}(M)_{k}$ make $Z_{\bullet}(M)_{k}$ into a semi-simplicial space where the face maps are induced by the face maps in $X_{\bullet}(M)_{k}$.

The projection maps
 $Z_{l}(M)_{k}  \longrightarrow \mathcal{M}(M)$ given by  $(V, \bar{\phi}) \mapsto V$
yield an augmentation map 
$\epsilon_{k}: Z_{l}(M)_{k} \longrightarrow \mathcal{M}(M).$
We denote by $Z_{-1}(M)_{k}$ the space $\mathcal{M}(M).$ 
\end{defn}
By construction, the projection maps $Z_{l}(M)_{k}  \rightarrow \mathcal{M}(M)$ are locally trivial fibre-bundles with standard fibre given by $X_{l}(M)_{k}$. 
From this we have:
\begin{corollary} \label{prop: highly connected resolution}
The map $|\epsilon_{k}|: |Z_{l}(M)_{k}| \longrightarrow \mathcal{M}(M)$ induced by the augmentation is $\frac{1}{2}(r_{k}(M) - 2)$-connected. 
\end{corollary} 
\begin{proof}
It follows from \cite[Lemma 2.1]{RW 10} that there is a homotopy-fibre sequence
$|X_{l}(M)_{k}| \rightarrow |Z_{l}(M)_{k}| \rightarrow \mathcal{M}(M).$
The result follows from the long-exact sequence on homotopy groups. 
\end{proof}

\subsection{Proof of theorem \ref{theorem: Main theorem}} \label{section: homological stability}
We show how to use the \textit{semi-simplicial resolution} $\epsilon_{k}: Z_{\bullet}(M)_{k} \rightarrow \mathcal{M}(M)$ to complete the proof of Theorem \ref{theorem: Main theorem}.
First, we fix some new notation which will make the steps of the proof easier to state.
For what follows, let $M$ be a compact $(4n+1)$-dimensional manifold with non-empty boundary. 
Let $k > 2$ be an odd integer. 
For each $g \in \N$ we denote by $M_{g, k}$ the manifold obtained by forming the connected-sum of $M$ with $W^{\# g}_{k}$.
Notice that $\partial M = \partial M_{g, k}$ for all $g \in \N$. 
We consider the spaces $\mathcal{M}(M_{g, k})$.
For each $g \in \N$, the stabilization map from (\ref{ref: k-stabilization map}) yields a map,
$$s_{k}: \mathcal{M}(M_{g, k}) \longrightarrow \mathcal{M}(M_{g+1, k}), \quad M' \mapsto \widetilde{W}_{k}\cup(M' + e_{1}).$$
Using the weak equivalence $\mathcal{M}(M_{g, k}) \simeq \BDiff^{\partial}(M_{g, k})$, Theorem \ref{theorem: Main theorem} translates to the following:
\begin{theorem} \label{thm: main theorem neq notation}
The induced map
$(s_{k})_{*}: H_{l}(\mathcal{M}(M_{g, k})) \longrightarrow H_{l}(\mathcal{M}(M_{g+1, k}))$
is an isomorphism when $l \leq \frac{1}{2}(g - 3)$ and is an epimorphism when $l \leq \frac{1}{2}(g - 1)$. 
\end{theorem}

\noindent
Since $r(M_{g, k}) \geq g$ for $g \in \N$, it follows from Corollary \ref{prop: highly connected resolution} that the map 
$$|\epsilon_{k}|: |Z_{\bullet}(M_{g, k})_{k}| \longrightarrow Z_{-1}(M_{g, k})_{k} := \mathcal{M}(M_{g, k}).$$
 is $\frac{1}{2}(g - 2)$-connected. With this established, the proof of Theorem \ref{thm: main theorem neq notation} proceeds in exactly the same way as in \cite[Section 5]{GRW 12}. We provide an outline for how to complete the proof and refer the reader to \cite[Section 5]{GRW 12} for details.

For what follows we fix $g \in \N$. 
For each non-negative integer $l \leq g$ there is a map
\begin{equation} \label{eq: resolution level map}
F_{k}: \mathcal{M}(M_{g-l-1, k}) \longrightarrow Z_{l}(M_{g, k})_{k}
\end{equation}
which is defined in exactly the same way as the map from \cite[Proposition 5.3]{GRW 12}.
From \cite[Proposition 5.3, 5.4 and 5.5]{GRW 12} we have the following.
\begin{proposition} \label{prop: homotopy commutativity} Let $g \geq 4$.
\begin{enumerate} \itemsep2pt
\item[i.] The map $F_{k}: \mathcal{M}(M_{g-l-1, k}) \longrightarrow Z_{k}(M_{g, k})_{k}$ is a weak homotopy equivalence. 
\item[ii.] The following diagram is commutative, 
$$\xymatrix{
\mathcal{M}(M_{g-l-1, k}) \ar[d]^{F_{k}} \ar[rr]^{s_{k}} && \mathcal{M}(M_{g-l, k}) \ar[d]^{F_{k}} \\
Z_{l}(M_{g, k})_{k} \ar[rr]^{d_{k}} && Z_{l-1}(M_{g, k})_{k}.
}$$
\item[iii.] The face maps $d_{i}: Z_{l}(M_{g, k})_{k} \longrightarrow Z_{l-1}(M_{g, k})_{k}$ are weakly homotopic. 
\end{enumerate}
\end{proposition}
\begin{remark} 
The proof of Proposition \ref{prop: homotopy commutativity} proceeds in the same way as the proofs of \cite[Proposition 5.3, 5.4 and 5.5]{GRW 12}. The key ingredients of this proof are Propositions \ref{proposition: transitivity} and \ref{corollary: cancelation}. \end{remark}
\noindent
Consider the spectral sequence associated to the skeletal filtration of the augmented semi-simplicial space $Z_{\bullet}(M_{g, k})_{k} \rightarrow \mathcal{M}(M_{g, k})$,
with $E^{1}$-term given by 
$E^{1}_{j,l} = H_{j}(Z_{l}(M_{g, k})_{k})$ for $l \geq -1$ and $j \geq 0$.
The differential is given by $d^{1} = \sum(-1)^{i}(d_{i})_{*}$, where $(d_{i})_{*}$ is the map on homology induced by the $i$th face map in $Z_{\bullet}(M_{g, k})_{k}$. 
The group $E^{\infty}_{j, l}$ is a subquotient of the relative homology group $H_{j+l+1}(Z_{-1}(M_{g, k})_{k}, \; |Z_{\bullet}(M_{g, k})_{k}|)$. 
Proposition \ref{prop: homotopy commutativity} together with Corollary \ref{prop: highly connected resolution} imply the following:
\begin{enumerate} \itemsep2pt
\item[(a)] For $g \geq 4 + d$, there are isomorphisms $E^{1}_{j,l} \cong H_{l}(\mathcal{M}(M_{g-j-1, k})).$
\item[(b)] The differential
$d^{1}: H_{l}(\mathcal{M}(M_{g-j-1, k})) \cong E^{1}_{j,l} \longrightarrow E^{1}_{j-1,l} \cong H_{l}(\mathcal{M}(M_{g-j, k}))$ is equal to 
$(s_{k})_{*}$ when $j$ is even and is equal to zero when $j$ is odd.
\item[(c)] The term $E^{\infty}_{j,l}$ is equal to $0$ when $j +l \leq \frac{1}{2}(g-2)$. 
\end{enumerate}
To complete the proof one uses (c) to prove that the differential $d^{1}: E^{1}_{2j,l} \longrightarrow E^{1}_{2j-1,l}$ is an isomorphism 
when $0 < j \leq \frac{1}{2}(g - 3)$  and an epimorphism when $0 < j \leq \frac{1}{2}(g -1)$. 
This is done by carrying out the inductive argument given in \cite[Section 5.2: \textit{Proof of Theorem 1.2}]{GRW 12}. This establishes Theorem \ref{thm: main theorem neq notation} and the main result of this paper, Theorem \ref{theorem: Main theorem}.

\appendix
 
 \section{Disjunction} \label{section: modyifing higher dim intersections}
 \noindent
 
 We now develop a technique for modifying the intersections of embedded $\langle k\rangle$-manifolds that will allow us to prove Theorem \ref{thm: modifying intersections 1} stated in Section \ref{section: intersections}. 
 Recall from Section \ref{subsection: main result disjunction} the definition of \textit{diffeotopy.}
  \begin{defn} \label{defn: diffeotopy 2}
 Let $M$ be a manifold. 
 We will call a smooth, one parameter family of diffeomorphisms $\Psi_{t}: M \longrightarrow M$ with $t \in [0, 1]$ and $\Psi_{0} = \Id_{M}$ a \textit{diffeotopy}. 
 For a subspace $N \subset M$,
 we say that $\Psi_{t}$ is a \textit{diffeotopy relative $N$}, and we write $\Psi_{t}: M \longrightarrow M \; \text{rel} \; N$, if in addition, $\Psi_{t}|_{N} = \Id_{N}$ for all $t \in [0,1]$.
  \end{defn} 

 \subsection{A modulo-$k$ version of the Whitney trick} \label{subsection: mod k whitney trick}
We now discus a version of the Whitney trick for $\langle k \rangle$-manifolds. 
Let $M$ be an oriented, manifold of dimension $m$, let $X$ be an oriented, compact manifold of dimension $r$, and
let $P$ be an oriented, compact $\langle k \rangle$-manifold of dimension $p$. 
Suppose that $m \geq 6$, $p, r \geq 2$, and that $p + r = m$.
Suppose further that $M$ is simply connected and that $P$ and $X$ are path connected.  
Let
\begin{equation} \label{equation: k emb and smooth emb}
\varphi: (X, \partial X) \longrightarrow (M, \partial M) \quad \text{and} \quad f: (P, \partial_{0}P) \longrightarrow (M, \partial M)
\end{equation}
be a smooth embedding and a $\langle k \rangle$-embedding respectively such that 
$$\varphi(\partial X)\cap f(\partial_{0}P) = \emptyset.$$ 
We will need to consider the invariant $\Lambda^{0}_{k}(f, \varphi; M)$ defined in Section \ref{section: intersections}. 
Using the standard identification
$$\Omega^{SO}_{0}(\text{pt.})_{\langle k \rangle} = \Z/k,$$ 
the element $\Lambda^{0}_{k}(f, \varphi; M)$ is equal to the modulo $k$ reduction of the oriented, algebraic intersection number associated to the intersection of $f(\Int(P))$ and $\varphi(X)$.
The following theorem is a version of the classical \textit{Whitney trick} for $\langle k \rangle$-manifolds. 
\begin{theorem} \label{theorem: mod k whitney trick}
Let $f$ and $\varphi$ be exactly as in (\ref{equation: k emb and smooth emb}) above. 
Using the identification $\Omega^{SO}_{0}(\text{pt.})_{\langle k \rangle} = \Z/k$, suppose that
$\Lambda^{0}_{k}(f, \varphi; M) \; = \; j \; \mod k.$
Then there exists a diffeotopy 
$$\Psi_{t}: M \longrightarrow M \; \rel \partial M$$ 
such that, 
$$\Psi_{1}(\varphi(X))\cap f(\Int(P)) \cong +\langle j  \rangle.$$
\end{theorem}
To prove the above theorem we will need to use the next lemma. 
\begin{lemma} \label{lemma: intersection creation}
Let $P$ be a compact $\langle k \rangle$-manifold of dimension $p \geq 2$, let $M$ be a smooth manifold of dimension $m \geq 6$, and let $f: (P, \partial_{0}P) \longrightarrow (M, \partial M)$ be a $\langle k \rangle$-embedding. 
Let $r$ denote the integer $m - p$. 
Given any any positive integer $n$, 
there exists an embedding 
$$g: S^{r} \longrightarrow \Int(M \setminus f_{\beta P}(\beta P))$$ 
that satisfies the following:
\begin{enumerate}
\itemsep2pt
\item[(i)] $g(S^{r})\cap f(\Int(P)) \cong \pm\langle n\cdot k \rangle$,
\item[(ii)] the composition 
$
S^{r} \stackrel{g} \longrightarrow  \Int(M \setminus f_{\beta P}(\beta P)) \hookrightarrow \Int(M)
$
extends to an embedding of the disk,
$D^{r+1} \hookrightarrow \Int(M).$
\end{enumerate}
\end{lemma}

\begin{proof}
Choose a collar embedding $h: \partial_{1}P\times[0,1] \longrightarrow P$ with $h^{-1}(\partial_{1}P) = \partial_{1}P\times\{0\}$. 
For each $i \in \langle k \rangle$, let $f_{i}: \beta P\times[0, 1] \longrightarrow M$
denote the embedding given by the composition
$$
\xymatrix{
\beta P\times\{i\}\times[0, 1] \ar@{^{(}->}[r] & \beta P\times\langle k \rangle\times[0, 1] \ar[rr]^{\Phi^{-1}}_{\cong} && \partial_{1}P\times[0,1]  \ar[rr]^{f\circ h} && M,
}
$$
where $\Phi: \partial_{1}P \stackrel{\cong} \longrightarrow \beta P\times\langle k \rangle$ is the structure map. 
Let $E \subset M$ be a smooth tubular neighborhood of $f_{\beta}(\beta P) \subset M$ and let $\pi: E \longrightarrow f_{\beta}(\beta P)$ denote the bundle projection.
Recall that $\pi: E \longrightarrow f_{\beta}(\beta P)$ is disk bundle isomorphic to the disk bundle associated to the normal bundle of $f_{\beta}(\beta P)$ in $M$.
Choose a point $x \in \beta P$, let $y = f_{\beta}(x)$, and let $E_{y} = \pi^{-1}(y)$ be the fibre of the projection $\pi$ over $y$.
By shrinking the neighborhood $E$ down close enough to $f_{\beta}(\beta P)$, we may assume that for each $i \in \langle k \rangle$, the intersection $E_{y}\cap f_{i}(\beta P\times[0,1])$ is equal to the compact interval $f_{i}(\{x\}\times[0, t_{i}]) \subset M$ for some $t_{i} \in [0, 1]$. 
It follows from this that 
$$\partial E_{y}\cap f(P) \; = \; \{\; f_{i}(x, t_{1}), \dots, f_{i}(x, t_{k})\; \},$$ 
and thus $\partial E_{y}$ intersects $f(P)$ at exactly $k$ points, all with the same orientation induced by the orientations of $P$ and $M$. 
The desired embedding $g: S^{r} \longrightarrow M$ is then given by composing a diffeomorphism $S^{r} \cong \partial E_{y}$ with the inclusion $\partial E_{y} \hookrightarrow M$. 
Since $\partial E_{y}$ is bounded by $E_{y} \subset M$, the embedding $g: S^{r} \longrightarrow M$ clearly extends to an embedding $D^{r+1} \hookrightarrow M$.
This proves the lemma for the case that $n = 1$. 
For general $n$ we carry out the same construction $n$ times using $n$ distinct fibres $E_{y_{1}}, \dots, E_{y_{n}}$ of the disk bundle $\pi: E \longrightarrow f_{\beta}(\beta P)$. 
We then form the connected sum of the spheres $\partial E_{y_{1}}, \dots, \partial E_{y_{n}} \hookrightarrow M$ to obtain the desired embedding. 
\end{proof}

\begin{proof}[Proof of Theorem \ref{theorem: mod k whitney trick}] 
It will suffice to prove the following:
suppose that $f(\Int(P))\cap X$ consists of exactly $k$ points, all of which are positively oriented.  
Then there exists a diffeotopy $\Psi_{t}: M \longrightarrow M \; \rel \; \partial M$ such that $\Psi_{0} = Id_{M}$ and $\Psi_{1}(X)\cap f(P) = \emptyset$. 
So, suppose that $f(\Int(P))\cap X$ consists of exactly $k$ points, all of which are positively oriented.
By the previous lemma, there exists and embedding 
$$g: S^{r} \longrightarrow \Int(M \setminus X\cup f_{\beta}(\beta P))$$ 
such that $g(S^{r})\cap f(\Int(P))$ consists of exactly $k$ points, all of which are negatively oriented. 
Furthermore, the embedding $g$ can be chosen so that it admits an extension to an embedding 
$$\bar{g}: D^{r+1} \longrightarrow \Int(M \setminus X).$$ 
Let $\tilde{X} \subset M$ be the submanifold  obtained by forming the connected sum of $g(S^{r}) \subset M$ with $X$ along some embedded arc in $M$ that is disjoint from $f(P)$. 
It follows easily from the fact that $g$ extends to an embedding of a disk, that $\tilde{X}$ is ambient isotopic to $X$. 
By construction, it follows that we have 
$$f(\Int(P))\cap\tilde{X} \cong +\langle k\rangle\sqcup -\langle k \rangle.$$ 
Since both $f(\Int(P))$ and $\tilde{X}$ are path connected and $M$ is simply connected by assumption, we may then apply the \textit{Whitney trick} to obtain a diffeotopy 
$$\Psi_{t}: M \longrightarrow M \rel \partial M$$ 
with $\Psi_{1}(X)\cap f(P) = \emptyset$. 
This concludes the proof of the theorem. 
\end{proof}

  \subsection{A higher dimensional intersection invariant.} 
 We recall now a certain construction developed by Hatcher and Quinn in \cite{HQ 74}.
 Let $M$, $X$, and $Y$ be smooth manifolds of dimension $m$, $r$, and $s$ respectively, with $X$ and $Y$ compact. 
 Let $t = r+ s - m$. 
Let 
$$\varphi: (X, \partial X) \longrightarrow (M, \partial M) \quad \text{and} \quad  \psi: (Y, \partial Y) \longrightarrow (M, \partial M)$$ 
be smooth maps. 
 Let $E(\varphi, \psi)$ denote the \textit{homotopy pull-back} of $\varphi$ and $\psi$, which is explicitly given by
$$E(\varphi, \psi) = \{(x, y, \gamma) \in X\times Y\times \text{Path}(M)\; | \; \varphi(x) = \gamma(0), \;  \psi(y) = \gamma(1)\; \}.$$
Consider the diagram, 
\begin{equation}
\xymatrix{
E(\varphi, \psi) \ar[rr]^{\pi_{X}} \ar[drr]^{\pi_{M}} \ar[d]^{\pi_{Y}} && X \ar[d]^{\varphi} \\
Y \ar[rr]^{\psi} && M
}
\end{equation}
where $\pi_{X}$ and $\pi_{Y}$ are projection maps and $\pi_{M}$ is given by the formula $(x, y, \gamma) \mapsto \gamma(1/2)$. 
It is easily verified that this diagram commutes up to homotopy. 
Let $\nu_{X}$ and $\nu_{Y}$ denote the stable normal bundles associated to $X$ and $Y$. 
We will need to consider the stable vector bundle over $E(\varphi, \psi)$ given by the Whitney sum 
$$\pi_{X}^{*}(\nu_{X})\oplus\pi_{Y}^{*}(\nu_{Y})\oplus\pi_{M}^{*}(TM).$$
We will denote this stable bundle by $\widehat{\nu}(\varphi, \psi)$. 
We will need to consider the \textit{normal bordism group} 
$$\Omega^{\text{fr.}}_{t}(E(\varphi, \psi), \; \widehat{\nu}(\varphi, \psi)).$$
Elements of this bordism group are represented by triples $(N, f, F)$, where $N$ is a $t$-dimensional closed manifold, $f: N \longrightarrow E(\varphi, \psi)$ is a map, and $F: \nu_{N} \longrightarrow \widehat{\nu}(\varphi, \psi)$ is an isomorphism of stable vector bundles covering the map $f$. 

Now, suppose that the maps $\varphi$ and $\psi$ are transversal and that 
$$\varphi(\partial X)\cap\psi(\partial Y) = \emptyset.$$ 
Since $X$ and $Y$ are compact by assumption, it follows that the pullback $\varphi\pitchfork\psi \subset X\times Y$ is a closed submanifold of dimension $t$. 
There is a natural map
$$\iota_{\varphi, \psi}: \varphi\pitchfork\psi \longrightarrow E(\varphi, \psi), \quad (x, y) \mapsto (x, y, c_{\varphi(x)}),$$
where $c_{\varphi(x)}$ is the constant path at point $\varphi(x)$. 
Let $\nu_{\varphi\pitchfork\psi}$ denote the stable normal bundle associated to the pull-back $\varphi\pitchfork\psi$. 
The following is given in \cite[Proposition 2.1]{HQ 74} (see also the discussion on Pages 331-332). 
\begin{proposition}
There is a natural bundle isomorphism  
$\hat{\iota}_{\varphi, \psi}: \nu_{\varphi\pitchfork\psi} \stackrel{\cong} \longrightarrow \nu(\varphi, \psi)$,
determined uniquely by the homotopy classes of $\varphi$ and $\psi$, 
that covers the map 
$\iota_{\varphi, \psi}$.
In this way, the triple 
$(\varphi\pitchfork\psi, \; \iota_{\varphi, \psi}, \; \hat{\iota}_{\varphi, \psi})$
determines a bordism class in 
 $\Omega^{\text{fr.}}_{t}(E(\varphi, \psi), \; \widehat{\nu}(\varphi, \psi))$.
 \end{proposition} 
 
The bordism group $\Omega^{\text{fr.}}_{t}(E(\varphi, \psi), \; \widehat{\nu}(\varphi, \psi))$ can be quite difficult to compute in general. 
However, in the case that the manifolds $X$, $Y$, and $M$ are highly connected, the group $\Omega^{\text{fr.}}_{t}(E(\varphi, \psi), \; \widehat{\nu}(\varphi, \psi))$ reduces to something much more simple. 
The following proposition is proven in \cite[Section 3]{HQ 74}.
\begin{proposition}
Suppose that $X$ and $Y$ are $t$-connected and that $M$ is $(t + 1)$-connected (recall that $t = \dim(X) + \dim(Y) - \dim(M) = r + s - m$). 
Then the homomorphism 
$$\Omega^{\text{fr.}}_{t}(\text{pt.}) \rightarrow \Omega^{\text{fr.}}_{t}(E(\varphi, \psi), \; \widehat{\nu}(\varphi, \psi))$$ 
induced by the inclusion of any point into $E(\varphi, \psi)$, is an isomorphism. 
\end{proposition}

\begin{defn}
In the case that $X$ and $Y$ are $t$-connected and $M$ is $(t + 1)$-connected,
we will denote by 
\begin{equation}
\alpha_{t}(\varphi, \psi; M) \in \Omega^{\text{fr.}}_{t}(\text{pt.})
\end{equation}
the image of the bordism class in $\Omega^{\text{fr.}}_{t}(E(\varphi, \psi), \; \widehat{\nu}(\varphi, \psi))$ associated to $\varphi\pitchfork\psi$ under the isomorphism of the previous proposition. 
\end{defn}
\begin{remark}
We emphasize that it is not necessary for both $\varphi$ and $\psi$ to be embeddings in order for the class $\alpha_{t}(\varphi, \psi; M)$ to be defined. 
It is only necessary that $\varphi$ and $\psi$ be transversal as smooth maps. 
Furthermore it is easy to see that $\alpha_{t}(\varphi, \psi, M)$ is an invariant of the homotopy class of $\varphi$ and $\psi$. 
However, the next theorem (Theorem \ref{theorem: higher whitney trick hatcher}) does require that $\varphi$ and $\psi$ be embeddings.
\end{remark}
The following is proven in \cite[Theorem 2.2]{HQ 74} (and in \cite{W 63}).
 
\begin{theorem} \label{theorem: higher whitney trick hatcher}
Let 
$$\varphi: (X, \partial X) \longrightarrow (M, \partial M) \quad \text{and} \quad \psi: (Y, \partial Y) \longrightarrow (M, \partial M)$$ 
be embeddings such that $\varphi(\partial X)\cap\psi(\partial Y) = \emptyset$.
Suppose that $m > r + \frac{s}{2} + 1$ and $m > s + \frac{r}{2} + 1$, that $X$ and $Y$ are $t$-connected, and that $M$ is $(t + 1)$-connected.
Then if $\alpha_{t}(\varphi, \psi; M) = 0$, there exists a diffeotopy 
$$\Psi_{t}: M \longrightarrow M \; \rel \partial M$$
such that $\Psi_{1}(\varphi(X))\cap\psi(Y) = \emptyset$. 
\end{theorem}
\begin{remark}
In \cite{HQ 74} the above theorem is only explicitly proven in the case when $X$ and $Y$ are closed manifolds, though their proof can easily by modified to yield the  version stated above. 
In \cite{We 67}, a proof of the relative version stated exactly as above is given. 
\end{remark}

The next lemma, which we will use later, is a restatement of \cite[Theorem 1.1]{HQ 74}. 
\begin{lemma} \label{lemma: homotopy disjunction}
Let
$$\varphi: (X, \partial X) \longrightarrow (M, \partial M) \quad \text{and} \quad \psi: (Y, \partial Y) \longrightarrow (M, \partial M)$$
be embeddings with $\varphi(\partial X)\cap \psi(\partial Y) = \emptyset$. 
Suppose that $\varphi$ is homotopic relative $\partial X$, to a map $\varphi'$ such that $\varphi'(X)\cap\psi(Y) = \emptyset$.
If  $m > r + s/2 + 1$, then there exists a diffeotopy 
$$\Psi_{t}: M \longrightarrow M \rel \partial M$$ 
such that $(\Psi_{1}\circ\varphi(X))\cap\psi(Y) = \emptyset$. 
\end{lemma}

\begin{remark}
The main dimensional case when we will use Theorem \ref{theorem: higher whitney trick hatcher} and Lemma \ref{lemma: homotopy disjunction} is when $\dim(M) = 2n + 1$, $\dim(X) = \dim(Y) = n + 1$, and $n \geq 4$. 
\end{remark}

\subsection{Creating intersections}
There is a particular application of the above theorem that we will need to use. 
Let $M$ and $Y$ be oriented, connected manifolds of dimension $m$ and $s$ respectively
and let 
$$\psi: (Y, \partial Y) \longrightarrow (M, \partial M)$$ 
be an embedding. 
Let $r = m - s$ and let $\varphi: S^{r} \longrightarrow \Int(M)$ be a smooth map transverse to $\psi(Y) \subset M$. 
Let $j \geq 0$ be an integer strictly less than $r$ and let $\gamma: S^{r+j} \longrightarrow S^{r}$ be a smooth map.
Let 
\begin{equation} \label{equation: framed pontryagin thom}
\mathcal{P}_{j}: \pi_{r+j}(S^{r}) \stackrel{\cong} \longrightarrow \Omega^{\text{fr.}}_{j}(\text{pt.})
\end{equation}
denote the \textit{Pontryagin-Thom} isomorphism for framed bordism (see \cite{M 65a}). 
The following lemma shows how to compute the class 
$$\alpha_{j}(\varphi\circ \gamma, \; \psi; \; M) \in \Omega^{\text{fr.}}_{j}(\text{pt.})$$ 
in terms of $\alpha_{0}(\varphi, \psi, M) \in \Omega^{\text{fr.}}_{0}(\text{pt.})$ and the element $\mathcal{P}_{j}([\gamma]) \in \Omega^{\text{fr.}}_{j}(\text{pt.})$.
\begin{lemma} \label{lemma: hopf map trick}
Let $\psi$, $\varphi$ and $\gamma: S^{r+j} \rightarrow S^{r}$ be exactly as above.
Then
$$\alpha_{j}(\varphi\circ\gamma, \; \psi; \; M) \; = \; \alpha_{0}(\varphi, \psi; M)\cdot \mathcal{P}_{j}([\gamma]),$$
where the product on the right-hand side is the product in the graded bordism ring $\Omega^{\text{fr.}}_{*}(\text{pt.})$. 
\end{lemma}
\begin{proof} Let $s \in \Z$ denote the oriented, algebraic intersection number associated to the intersection of $\varphi(S^{r})$ and $\psi(Y)$.
By application of the Whitney trick, we may deform $\varphi$ so that
\begin{equation} \label{eq: pre-image equality}
\varphi(S^{r})\cap\psi(Y) = \{x_{1}, \dots, x_{\ell}\},
\end{equation}
where the points $x_{i}$ for $i = 1, \dots, \ell$ all have the same sign. 
It follows that 
$$(\varphi\circ \gamma)^{-1}(\psi(Y)) = \bigsqcup_{i=1}^{\ell}\gamma^{-1}(x_{i}).$$ 
For each $i \in \{1, \dots, \ell\}$, the framing at $x_{i}$ (induced by the orientations of $\gamma(S^{r})$, $\psi(Y)$ and $M$) induces a framing on the $j$-dimensional closed manifold $\gamma^{-1}(x_{i})$.
We denote the element of $\Omega_{j}^{\text{fr.}}(\text{pt.})$ given by $\gamma^{-1}(x_{i})$ with this induced framing by $[\gamma^{-1}(x_{i})]$. 
By definition of the Pontryagin-Thom map $\mathcal{P}_{j}$ (see \cite[Section 7]{M 65a}), the element $[\gamma^{-1}(x_{i})]$ is equal to $\mathcal{P}_{j}([\gamma])$ for $i = 1, \dots, \ell$. 
Using the equality (\ref{eq: pre-image equality}), it follows that 
$$\alpha_{j}(\varphi\circ \gamma, \; \psi; \; M) = \ell\cdot \mathcal{P}_{j}([\gamma]).$$ 
The proof then follows from the fact that $\alpha_{0}(\varphi, \psi, M)$ is identified with the algebraic intersection number associated to $\varphi(S^{r})$ and $\psi(Y)$. 
\end{proof}

\subsection{A technical lemma.}
 Before we proceed further, we develop a technical result that will play an important role in the proof of Theorem \ref{thm: modifying intersections}.
For $n \ge 4$, let $M$ be a $2$-connected, oriented $(2n+1)$-dimensional manifold and let $P$ be a simply connected, oriented, $\langle k\rangle$-manifold of dimension $n+1$. 
Let 
$$f: (P, \partial_{0}P) \longrightarrow (M, \partial M)$$ 
be a $\langle k\rangle$-embedding. 
Let $U$ be a tubular neighborhood of $f_{\beta}(\beta P) \subset M$ whose boundary intersects $f(\Int(P))$ transversally. 
Denote,
\begin{equation} \label{eq: Z tubular neighborhood}
Z := M\setminus \Int(U), \quad \quad P' := f^{-1}(Z), \quad \quad f' := f|_{P'}.
\end{equation}
It follows from the fact that $\partial U$ intersects $\Int(f(P))$ transversally that $P'$ is a smooth manifold with boundary (after smoothing corners) and that $f'$ maps $\partial P'$ into $\partial M$. 
Let $\xi$ denote the generator of the framed bordism group $\Omega^{\text{fr.}}_{1}(\text{pt.})$, which is isomorphic to $\Z/2$.
\begin{lemma} \label{lemma: higher intersection creation}
Let $f: (P, \partial_{0}P) \longrightarrow (M, \partial M)$ be as above and let $i_{Z}: Z \hookrightarrow M$ denote the inclusion map. 
There exists an embedding $\varphi: S^{n+1} \longrightarrow Z$ which satisfies:
\begin{enumerate} \itemsep.2cm
\item[i.] $\alpha_{1}(f', \varphi; Z) \; = \; k\cdot \xi \in \Omega^{\text{fr.}}_{1}(\text{pt.})$, 
\item[ii.] the composition $i_{Z}\circ\varphi: S^{n+1} \longrightarrow M$ is null-homotopic. 
\end{enumerate}
\end{lemma}
\begin{proof}
By Lemma \ref{lemma: intersection creation}, we may choose an embedding 
$\phi: S^{n} \longrightarrow M\setminus f_{\beta}(\beta P)$ an embedding that satisfies:
\begin{itemize} \itemsep.2cm
\item $\phi(S^{n})$ intersects $f(\Int(P))$ transversally,
\item $\phi(S^{n})\cap f(\Int(P)) \cong +\langle k \rangle$, 
\item $i_{Z}\circ\phi: S^{n} \rightarrow M$ extends to an embedding $D^{n+1} \hookrightarrow M$. 
\end{itemize} 
By shrinking the tubular neighborhood $U$ of $f_{\beta}(\beta P)$ if necessary, we may assume that 
$$\phi(S^{n}) \subset Z = M \setminus \Int(U).$$ 
Denote by $\widehat{\phi}: S^{n} \longrightarrow Z$ the map obtained by restricting the codomain of $\phi$. 
Let 
$$\gamma: S^{n+1} \longrightarrow S^{n}$$ 
represent the generator of $\pi_{n+1}(S^{n}) \cong \Z/2$. 
By Lemma \ref{lemma: hopf map trick} it follows that, 
$$\alpha_{1}(\widehat{\phi}\circ \gamma, \; f'; Z) = \alpha_{0}(\widehat{\phi}, \; f'; Z)\cdot\mathcal{P}_{1}([\gamma]) =  k\cdot \mathcal{P}_{1}([\gamma]) = k\cdot \xi,$$
where $\mathcal{P}_{1}: \pi_{n+1}(S^{n}) \longrightarrow \Omega_{1}^{\text{fr.}}(\text{pt.})$ is the Pontryagin-Thom map for framed bordism.  
Since $Z$ is $2$-connected and $n \geq 4$, we may apply \cite[Proposition 1]{W 67} (or the main theorem of \cite{H 61}), and find a homotopy of the map $\widehat{\phi}\circ\gamma$, to an embedding 
$\varphi: S^{n+1} \longrightarrow Z.$
Since the map $i_{Z}\circ\phi: S^{n} \longrightarrow M$ is null-homotopic, it follows that $i_{Z}\circ\varphi: S^{n+1} \rightarrow M$ is null-homotopic as well.  
This completes the proof of the lemma. 
\end{proof}

\subsection{Modifying Intersections.}
We now state the main result of this section (which is a restatement of Theorem \ref{thm: modifying intersections 1} from Section \ref{subsection: main result disjunction}).
Fix an integer $n \geq 4$, let $M$ be an oriented, $2$-connected manifold of dimension $2n+1$. 
 Let $P$ and $Q$ be compact, oriented, simply connected, $\langle k \rangle$-manifolds of dimension $n+1$ such that both $\beta P$ and $\beta Q$ are path connected.
 \begin{theorem} \label{thm: modifying intersections} With $M$, $P$, and $Q$ as above and let 
 $$f: (P, \partial_{0}P) \longrightarrow (M, \partial M) \quad and \quad g: (Q, \partial_{0}Q) \longrightarrow (M, \partial M)$$ 
 be transversal $\langle k\rangle$-embeddings such that 
 $f(\partial_{0}P)\cap g(\partial_{0}Q) = \emptyset.$
 Suppose that $\Lambda^{1}_{k,k}(f, g; M) = 0$. 
 If the integer $k$ is odd
 then there exists a diffeotopy 
 $$\Psi_{t}: M \longrightarrow M \; \rel \partial M$$ 
 such that $\Psi_{1}(f(P))\cap g(Q) = \emptyset$. 
  \end{theorem}

 \noindent
The proof of the above theorem is proven in stages via several intermediate propositions. 
\begin{proposition} \label{prop: k-whitney trick application} Let 
$$f: (P, \partial_{0}P) \longrightarrow (M, \partial M) \quad \text{and} \quad g: (Q, \partial_{0}Q) \longrightarrow (M, \partial M)$$
be $\langle k\rangle$-embeddings as above and 
suppose that  
$$
\beta_{1}(\Lambda^{1}_{k,k}(f, g; M)) \; = \; \Lambda^{0}_{k}(f_{\beta}, g; M)  \; =  \; 0.
$$
Then there exists a diffeotopy 
$\Psi_{t}: M \longrightarrow M \rel \partial M$
such that   
 $$\Psi_{1}(f_{\beta}(\beta P))\cap g(Q) = \emptyset.$$
\end{proposition}
\begin{proof}
Since $0 = \beta_{1}(\Lambda^{1}_{k,k}(f, g; M)) = \Lambda^{0}_{k}(g, f_{\beta}; M)$, it follows that the algebraic intersection number associated to $f_{\beta}(\beta P)$ and $g(\Int Q)$ is a multiple of $k$. 
The desired diffeotopy exists by Theorem \ref{theorem: mod k whitney trick}.
\end{proof}

 \begin{proposition} \label{lemma: k intersection 1}
Let 
$g: (Q, \partial_{0}Q) \longrightarrow (M, \partial M)$
be a $\langle k \rangle$-embedding as above. 
Let $X$ be a smooth manifold of dimension $n+1$ and let 
$\varphi: (X, \partial X) \longrightarrow (M, \partial M)$
be a smooth embedding such that
$$\varphi(\partial X)\cap g(\partial_{0}Q) = \emptyset.$$
If the integer $k$ is odd, then there exists a diffeotopy, 
$\Psi_{t}: M \rightarrow M \; \rel \partial M$
such that,
$$\Psi_{1}(\varphi(X)) \cap g(Q)  \; = \; \emptyset.$$
\end{proposition} 
\begin{proof}
By Proposition \ref{proposition: k-intersection}, we have
$$\beta(\Lambda^{1}_{k}(g, \varphi; M)) \; = \;  \Lambda^{0}(g_{\beta}, \varphi; M) \in \Omega^{SO}_{0}(\text{pt.})$$
where 
$$\beta:  \Omega^{SO}_{1}(\text{pt.})_{\langle k\rangle} \longrightarrow \Omega^{SO}_{0}(\text{pt.}), \quad [V] \mapsto [\beta V]$$
is the Bockstein homomorphism. 
By (\ref{equation: elementary bordism calculation}), this Bockstein homomorphism is the zero map for all $k$ (the group $\Omega^{SO}_{1}(\text{pt.})_{\langle k\rangle}$ is equal to zero). 
 It follows that $\Lambda^{0}(g_{\beta}, \varphi; M) \in \Omega^{SO}_{0}(\text{pt.})$ is the zero element and 
 thus the oriented, algebraic intersection number associated to $g_{\beta}(\beta Q)\cap X$ is equal to zero. 
 By application of the \textit{Whitney trick}  \cite[Theorem 6.6]{M 65b}, we may find a diffeotopy of $M$, relative $\partial M$, which pushes $X$ off of the submanifold $g_{\beta}(\beta Q) \subset M$. 
 Using this, we may now assume that $\varphi(X)\cap g(\partial_{1}Q) = \emptyset$. 
 
 Let $U \subset M$ be a closed tubular neighborhood of $f_{\beta}(\beta P)$, disjoint from $X$, such that the boundary of $U$ intersects $f(P)$ transversely. 
As in (\ref{eq: Z tubular neighborhood}), we denote 
 $$Z := M\setminus \Int U, \quad \quad  P' := f^{-1}(Z), \quad \quad f' := f|_{P'}.$$
 Notice that $P'$ is a manifold with boundary and that $f'$ is an embedding which maps $(P', \partial P')$ into $(Z, \partial Z)$.
 Furthermore, $\varphi$ maps $(X, \partial X)$ into $(Z, \partial Z)$.
 To prove the corollary it will suffice to construct a diffeotopy 
 $\Psi'_{t}: Z \longrightarrow Z \; \rel \partial Z$
 such that $\Psi'_{1}(X)\cap P' = \emptyset$. 
 By Theorem \ref{theorem: higher whitney trick hatcher}, 
the obstruction to the existence of such a diffeotopy is the class
$\alpha_{1}(f', \varphi; Z) \in \Omega_{1}^{\text{fr}}(\text{pt.}).$
If $\alpha_{1}(f', \varphi; Z)$ is equal to zero, we are done. 
So suppose that $\alpha_{1}(f', \varphi; Z) = \xi$ where $\xi$ is the non-trivial element in $\Omega_{1}^{\text{fr}}(\text{pt.}) \cong \Z/2$.
Denote by $i_{Z}: Z \hookrightarrow M$ the inclusion map.
By Lemma \ref{lemma: higher intersection creation} there exits an embedding $\phi: S^{n+1} \longrightarrow Z$ such that:
 \begin{itemize} \itemsep.2cm
 \item  $\alpha_{1}(f', \phi; Z) = k\cdot \xi$ where $\xi \in \Omega^{\text{fr}}_{1}(\text{pt.}) \cong \Z/2$ is the standard generator, 
 \item the embedding $i_{Z}\circ \phi: S^{n+1} \longrightarrow M$ is null-homotopic.
 \end{itemize}
 Since $k$ is odd, we have $\alpha_{1}(f', \phi; Z) = \xi$. 
  We denote by $\widehat{\varphi}: X \longrightarrow M$ the embedding obtained by forming the connected sum of $\varphi(X)$ with $i_{Z}\circ\varphi(S^{n+1})$ along the thickening of an embedded arc that is disjoint from $f(P)$, $U$, and $X$. 
Since $i_{Z}\circ\varphi: S^{n+1} \longrightarrow M$ is null-homotopic, it follows that $\widehat{\varphi}$ is homotopic, relative to $\partial X$, to the original embedding $\varphi$.
We have
$$\alpha_{1}(f', \widehat{\varphi}; Z) = \alpha_{1}(f', \varphi; Z) + \alpha_{1}(f', \phi; Z) = \xi + \xi = 0,$$
and so there exists a diffeotopy $\Psi'_{t}: Z \rightarrow Z \; \rel \partial Z$ such that $\Psi'_{1}(\widehat{\varphi}(X))\cap f'(P') = \emptyset$. 
We then extend $\Psi'_{t}$ identically over $M\setminus Z$ to obtain a diffeotopy 
$$\widehat{\Psi}_{t}: M \longrightarrow M\; \rel \partial M$$ 
such that $\widehat{\Psi}_{1}(\widehat{\varphi}(X))\cap f(P) = \emptyset$. 
Now, since $\varphi$ is homotopic relative $\partial X$ to the embedding $\widehat{\Psi}_{1}\circ\widehat{\varphi}$ and $\widehat{\Psi}_{1}(\widehat{\varphi}(X))\cap f(P) = \emptyset$, we may apply Lemma \ref{lemma: homotopy disjunction} to obtain a diffeotopy 
$$\Psi_{t}: M \longrightarrow M \rel \partial M$$
such that $(\Psi_{1}\circ\varphi(X))\cap f(P) = \emptyset$. 
This concludes the proof of the proposition.
\end{proof}

We can now complete the proof of Theorem \ref{thm: modifying intersections}. 
  \begin{proof}[Proof of Theorem \ref{thm: modifying intersections}]
 By hypothesis we have $\Lambda^{1}_{k,k}(f, g; M) = 0,$ and thus $\Lambda^{0}_{k}(f_{\beta}, g; M) = 0$, and so 
 by Proposition \ref{prop: k-whitney trick application} we may assume that 
$f_{\beta}(\beta P)\cap g(Q) = \emptyset$.
Choose a closed tubular neighborhood $U \subset M$ of $f_{\beta}(\beta P)$, disjoint from $g(Q)$, with boundary transverse to $f(P)$. 
As in (\ref{eq: Z tubular neighborhood}) we denote, 
\begin{equation} \label{equation: cut out manifolds}
Z := M \setminus \Int U, \quad P' := f^{-1}(Z), \quad \text{and} \quad f' := f|_{P'}. 
\end{equation}
With these definitions, $P'$ is an oriented manifold with boundary and 
$$f': (P', \partial P') \longrightarrow (Z, \partial Z)$$ 
is an embedding. 
Furthermore, since $U$ was chosen to be disjoint from $g(Q)$, we have $g(Q) \subset Z$. 
Let $g': (Q, \partial_{0}) \longrightarrow (Z, \partial Z)$ denote the $\langle k \rangle$-embedding obtained by restricting the codomain of $g$. 
To finish the proof, we then apply Proposition \ref{lemma: k intersection 1} to the embedding $f': (P', \partial P') \longrightarrow (Z, \partial Z)$ and $\langle k \rangle$-embedding $g': (Q, \partial_{0}Q) \longrightarrow (Z, \partial Z)$, to obtain a diffeotopy of $Z$ (relative $\partial Z$) that pushes $f'(P')$ off of $g'(Q)$. 
This completes the proof of the theorem.
 \end{proof}
 
 We now come to an important corollary. 
 Recall from Section \ref{subsection: classification of k, l manifolds} the $\langle k, k\rangle$-manifold $A_{k}$. 
 \begin{corollary} \label{corollary: intersection at A-k}
Let $f$ and $g$ be exactly as in the statement of Theorem \ref{thm: modifying intersections}. 
 Suppose that the class 
 $\Lambda^{1}_{k, k}(f, g; M)$ 
 is equal to the class represented by the closed $1$-dimensional $\langle k, k\rangle$-manifold $+A_{k}$. 
 If $k$ is odd then there exists a diffeotopy $\Psi_{t}: M \longrightarrow M \; \rel \partial M$ such that the transverse pull-back $(\Psi_{1}\circ f) \pitchfork g$ is diffeomorphic to $A_{k}$. 
 \end{corollary}
 \begin{proof}
 Since $\Lambda^{1}_{k,k}(f, g; M)$ is equal to the class represented by $+A_{k}$ in $\Omega^{SO}_{1}(\text{pt.})_{\langle k, k\rangle}$, it follows that $f\pitchfork g$ is diffeomorphic (as an oriented $\langle k, k\rangle$-manifold) to the disjoint union of precisely one copy of $+A_{k}$ together with some other oriented $\langle k, k\rangle$-manifold, that represents the zero element in $\Omega^{SO}_{1}(\text{pt.})_{\langle k, k\rangle}$. 
 We may write
 \begin{equation}
 f(P)\cap g(Q) \; = \; A \sqcup Y,
 \end{equation}
 where 
 $$A \cong +A_{k} \quad \text{and} \quad  [Y] = 0 \;  \text{in $\Omega^{SO}_{1}(\text{pt.})_{\langle k, k\rangle}$.}$$
 Let $U \subset M$ be a closed neighborhood of $f_{\beta}(\beta P)\cup A$, disjoint from $Y$, with boundary transverse to both $f(P)$ and $g(Q)$. 
We then denote
\begin{equation}
\xymatrix@C-.10pc@R-1.8pc{
Z := M\setminus\Int(U), & P' := f^{-1}(Z), & Q' := g^{-1}(Z).
}
\end{equation}
Notice that both $P'$ and $Q'$ are $\langle k\rangle$-manifolds with 
$$\xymatrix@C-.10pc@R-1.8pc{
\partial_{0}P' = f^{-1}(\partial Z), & \partial_{1}P' = (f|_{\partial_{1}P})^{-1}(Z), & \beta P' = f_{\beta}^{-1}(Z), \\
\partial_{0}Q' = g^{-1}(\partial Z), & \partial_{1}Q' = (g|_{\partial_{1}Q})^{-1}(Z), & \beta Q' = g_{\beta}^{-1}(Z).
}$$
We denote by 
$$f': (P', \partial_{0}P') \longrightarrow (Z, \partial Z) \quad \text{and} \quad g': (Q', \partial_{0}Q') \longrightarrow (Z, \partial Z)$$ 
the $\langle k\rangle$-embeddings given by restricting $f$ and $g$.
By construction, the pull-back $f'\pitchfork g'$ is diffeomorphic as an oriented $\langle k, k\rangle$-manifold to $Y$, which represents the zero element in $\Omega_{1}^{SO}(\text{pt.})_{ \langle k, k\rangle}$.
It follows that $\Lambda^{1}_{k,k}(f', g'; Z) = 0$.  
By Theorem \ref{thm: modifying intersections} we obtain a diffeotopy $\Psi_{t}: Z \longrightarrow Z \; \rel \partial Z$,
 such that $\Psi_{1}(f^{'}(P'))\cap g'(Q') = \emptyset$.
This concludes the proof.
\end{proof}

\section{$\langle k \rangle$-Immersions and Embeddings} \label{section: k-immersions}
In this section we determine the conditions for when a $\langle k \rangle$-map can be deformed to a $\langle k \rangle$-immersion or a $\langle k \rangle$-embedding. 
The techniques of this section enable us to prove Theorem \ref{theorem: immersion to embedding} (which is restated again in this section as Theorem \ref{theorem: represent by embedding 2}). 
\subsection{A recollection of Smale-Hirsch theory}
Let $N$ and $M$ be smooth manifolds of dimensions $n$ and $m$ respectively. 
Denote by $\Imm(N, M)$ the space of immersions $N \rightarrow M$, topologized in the $C^{\infty}$-topology. 
Let $\Imm^{f}(N, M)$ denote the space of bundle maps
$TN \longrightarrow TM$
which are fibre-wise injective. 
Elements of the space $\Imm^{f}(N, M)$ are called \textit{formal immersions}. 
There is a map 
$\xymatrix{
\mathcal{D}: \Imm(N, M) \longrightarrow \Imm^{f}(N, M)
}$
defined by sending an immersion $\phi: N \longrightarrow M$ to the bundle injection given by its differential $D\phi: TN \longrightarrow TM$. 
The following theorem is proven in \cite[Chapter III, Section 9]{A 84} and is originally due to Hirsch and Smale.  
\begin{theorem} \label{theorem: smale-hirsch theorem}
The if $\dim(N) < \dim(M)$, then the map $\mathcal{D}: \Imm(N, M) \longrightarrow \Imm^{f}(N, M)$ is a weak homotopy equivalence. 
In the case that $\dim(N) = \dim(M)$, then $\mathcal{D}$ is a weak homotopy equivalence if $N$ is an open manifold. 
\end{theorem}
\noindent 
Let $\widehat{\Imm}(N, M)$ denote the space of pairs 
$(\phi, \mb{v}) \in \Imm(N, M)\times\Maps(N, TM)$
that satisfy:
\begin{enumerate}\itemsep.2cm
\item[i.] $\pi(\mb{v}(x)) = \phi(x)$ for all $x \in N$, where $\pi: TM \rightarrow M$ is the bundle projection,
\item[ii.] for each $x \in N$, the vector $\mb{v}(x)$ is transverse to the vector subspace 
$$D\phi(T_{x}N) \subset T_{\phi(x)}M,$$
where $D\phi$ is the differential of $\phi$. 
\end{enumerate}
Similarly, we define $\widehat{\Imm}^{f}(N, M)$ to be the space of pairs $(\psi, \mb{v}) \in \Imm^{f}(N, M)\times\Maps(N, TM)$
which satisfy:
\begin{enumerate}
\item[i.] $\pi(\mb{v}(x)) = \pi(\psi(x))$ for all $x \in N$, where $\pi: TM \rightarrow M$ is the bundle projection,
\item[ii.] for all $x \in N$, the vector $\mb{v}(x)$ is transverse to the vector subspace 
$$\psi(T_{x}N) \subset T_{\pi(\psi(x))}M.$$ 
\end{enumerate}
There is a map 
\begin{equation} \label{equation: differential map}
\xymatrix{
\widehat{\mathcal{D}}: \widehat{\Imm}(N, M) \longrightarrow \widehat{\Imm}^{f}(N, M), \quad (\phi, \mb{v}) \mapsto (D\phi, \mb{v}).
}
\end{equation}
 The following is an easy corollary of Theorem \ref{theorem: smale-hirsch theorem}.
 \begin{corollary} \label{theorem: h-principle}
Suppose that $\dim(N) < \dim(M)$. 
 Then the map $\widehat{\mathcal{D}}$ from (\ref{equation: differential map}) is a weak homotopy equivalence. 
 \end{corollary}

\subsection{The space of $\langle k \rangle$-immersions} 
We now proceed to prove a version of Corollary \ref{theorem: h-principle} for immersions of $\langle k \rangle$-manifolds. 
For what follows, let $M$ be a manifold of dimension $m$ and let $P$ be a $\langle k \rangle$-manifold of dimension $p$. 
We will need to construct a suitable space of $\langle k \rangle$-immersions and formal $\langle k \rangle$-immersions. 

Choose a collar embedding $h: \partial_{1}P\times[0,\infty) \longrightarrow P$, with $h^{-1}(\partial_{1}P) = \partial_{1}P\times\{0\}$. 
Denote by $\mb{v}_{h} \in \Gamma_{\partial_{1}P}(TP)$ the inward pointing vector field along $\partial_{1}P$ determined by the differential of the collar embedding $h$.
Using $\mb{v}_{h}$ we have maps,
\begin{equation} \label{equation: restriction map}
\xymatrix@C-.10pc@R-1.8pc{
R: \Imm(P, M) \longrightarrow \widehat{\Imm}(\partial_{1}P, M), & \phi \mapsto (\phi|_{\partial P}, \; D\phi\circ\mb{v}_{h}), \\
R^{f}:  \Imm^{f}(P, M) \longrightarrow \widehat{\Imm}^{f}(\partial_{1}P, M), & \psi \mapsto (\psi|_{\partial P}, \; \psi\circ\mb{v}_{h}).
}
\end{equation}
The next lemma follows from the basic results of \cite[Chapter III: Section 9]{A 84}. 
\begin{lemma} \label{lemma: restriction map fibration}
The map $R^{f}$ is a Serre-fibration in the case that $\dim(P) \leq \dim(M)$. 
The map $R$ is a Serre-fibration in the case that $\dim(P) < \dim(M)$.   
\end{lemma}
\noindent
Let $\bar{\Phi}: \partial_{1}P \longrightarrow \beta P$ be the map given by the composition 
$\xymatrix{
\partial_{1}P \ar[r]^{\Phi\ \ \ }_{\cong\ \ \ } & \beta P \times\langle k \rangle \ar[r]^{\ \ \ \text{proj}_{\beta P}} & \beta P.
}$
Using $\bar{\Phi}$ we have a map 
\begin{equation}
\xymatrix{
T_{k}: \widehat{\Imm}(\beta P, M) \longrightarrow \widehat{\Imm}(\partial_{1}P, M), \quad (\phi, \mb{v}) \mapsto (\phi\circ\bar{\Phi}, \; \mb{v}\circ\bar{\Phi}).
}
\end{equation}
Similarly, by using the differential 
 $D\bar{\Phi}$ of $\bar{\Phi}$, we define a map 
\begin{equation}
\xymatrix{
T^{f}_{k}: \widehat{\Imm}^{f}(\beta P, M) \longrightarrow \widehat{\Imm}^{f}(\partial_{1}P, M), \quad (\psi, \mb{v}) \mapsto (\psi\circ D\bar{\Phi}, \; \mb{v}\circ\bar{\Phi}).
}
\end{equation}

\begin{defn} \label{defn: space of formal k-immersions}
We define $\Imm_{\langle k \rangle}(P, M)$ to be the space of pairs 
$$\xymatrix{
(\phi, (\phi', \mb{v})) \in \Imm(P, M)\times\widehat{\Imm}(\beta P, M)
}$$
such that $T_{k}(\phi', \mb{v}) = R(\phi)$. 
Similarly we define $\Imm^{f}_{\langle k \rangle}(P, M)$ to be the space of pairs 
$$\xymatrix{
(\psi, (\psi', \mb{v})) \in \Imm^{f}(P, M)\times\widehat{\Imm}^{f}(\beta P, M)
}$$
such that $T^{f}_{k}(\psi', \mb{v}) = R^{f}(\psi)$. 
\end{defn}

\begin{remark} \label{remark: space of k-immersions}
Let $(\phi, (\phi', \mb{v})) \in \Imm_{\langle k \rangle}(P, M)$. 
By construction, the immersion $\phi: P \longrightarrow M$ is a $\langle k \rangle$-immersion and $\phi' = \phi_{\beta}$. 
The pair $(\phi', \mb{v})$ is completely determined by the $\langle k \rangle$-immersion $\phi$ and so, the space $
\Imm_{\langle k \rangle}(P, M)$ is homeomorphic to the subspace of $\Maps_{\langle k \rangle}(P, M)$ consisting of all $\langle k \rangle$-immersions $P \rightarrow M$. 
\end{remark}

\begin{lemma} \label{lemma: h-cartesian diagrams}
The following two commutative diagrams 
$$
\xymatrix{
\Imm_{\langle k \rangle}(P, M) \ar[r] \ar[d] & \Imm(P, M) \ar[d]^{R} &&   \Imm^{f}_{\langle k \rangle}(P, M) \ar[d] \ar[r] & \Imm^{f}(P, M) \ar[d]^{R^{f}} \\
\widehat{\Imm}(\beta P, M) \ar[r]^{T_{k}} & \widehat{\Imm}(\partial_{1}P, M), &&  \widehat{\Imm}^{f}(\beta P, M) \ar[r]^{T^{f}_{k}} & \widehat{\Imm}^{f}(\partial_{1}P, M),
}
$$
are homotopy cartesian.
\end{lemma}
\begin{proof}
This follows immediately from Lemma \ref{lemma: restriction map fibration} and the fact that both of the diagrams are pull-backs. 
\end{proof}
\noindent
Finally we may consider the map 
\begin{equation} \label{equation: h-k principle 1}
\xymatrix{
\mathcal{D}_{k}: \widehat{\Imm}_{\langle k \rangle}(P, M) \longrightarrow \widehat{\Imm}^{f}_{\langle k \rangle}(P, M), \quad \quad (\phi, \; (\phi', \mb{v})) \mapsto (D\phi, \; (D\phi', \mb{v})).
}
\end{equation}
We have the following theorem.
\begin{theorem} \label{theorem: k-h principle}
Suppose that $\dim(P) < \dim(M)$. 
Then the map $\mathcal{D}_{k}$ of (\ref{equation: h-k principle 1}) is a weak homotopy equivalence. 
\end{theorem}
\begin{proof}
The map from (\ref{equation: h-k principle 1}) induces a map between the two commutative squares in Lemma \ref{lemma: h-cartesian diagrams}. 
The maps between the entries on the bottom row and the entries on the upper-right are weak homotopy equivalences by Theorem \ref{theorem: smale-hirsch theorem} and Corollary \ref{theorem: h-principle}. 
It then follows from Lemma \ref{lemma: h-cartesian diagrams} that the upper-left map (which is (\ref{equation: h-k principle 1})) is a weak homotopy equivalence. 
\end{proof}

\subsection{Representing homotopy classes of $\langle k \rangle$-maps by $\langle k\rangle$-immersions} \label{representing}
Let $P$ be a $\langle k \rangle$-manifold of dimension $p$ and let $h: \partial_{1}\times[0, \infty) \longrightarrow P$ be a collar embedding with $h^{-1}(\partial_{1}P) = \partial_{1}P\times\{0\}$. 
We have a bundle map 
\begin{equation}
\Phi^{*}: TP|_{\partial_{1}P} \longrightarrow T(\beta P)\oplus \epsilon^{1}
\end{equation}
given by the composition,
$\xymatrix{
TP|_{\partial_{1}P}  \ar[r]^{\cong \ \ \ } & T(\partial_{1}P)\oplus\epsilon^{1} \ar[rr]^{D\bar{\Phi}\oplus Id_{\epsilon^{1}}} && T(\beta P)\oplus\epsilon^{1},
}$
where the first map is the bundle isomorphism induced by the collar embedding $h$. 
Using this bundle isomorphism $\Phi^{*}$, we define a new space $T\widehat{P}$ as a quotient of $TP$ by identifying two points $v, v' \in TP|_{\partial_{1}P} \subset TP$ if and only if $\Phi^{*}v = \Phi^{*}v'$. 
With this definition, there is a natural projection $\widehat{\pi}: T\widehat{P} \longrightarrow \widehat{P}$ which makes the diagram
\begin{equation} \label{equation: bundle diagram}
\xymatrix@C-.10pc@R-.7pc{
TP \ar[d]^{\pi} \ar[rr] && T\widehat{P} \ar[d]^{\widehat{\pi}} \\
P \ar[rr] && \widehat{P}
}
\end{equation}
commute. 
It is easy to verify that the projection map $\widehat{\pi}: T\widehat{P} \longrightarrow \widehat{P}$ is a vector bundle 
and that the upper-horizontal map in the above diagram is a bundle map that is an isomorphism on each fibre.  
\begin{defn} \label{defn: k-parallelizability}
The $\langle k \rangle$-manifold $P$ is said to be \textit{parallelizable} if the induced vector bundle $\widehat{\pi}: T\widehat{P} \rightarrow \widehat{P}$ is trivial. 
\end{defn}

\begin{corollary} \label{corollary: trivial bundle representation}
Let $P$ be a parallelizable $\langle k \rangle$-manifold and let $M$ be a manifold of dimension greater than $\dim(P)$. 
Let $f: P \longrightarrow M$ be a $\langle k \rangle$-map and
consider the induced map $\widehat{f}: \widehat{P} \longrightarrow M$. 
Suppose that the pull-back bundle 
$\widehat{f}^{*}(TM) \longrightarrow \widehat{P}$
is trivial. 
Then $f$ is homotopic through $\langle k \rangle$-maps to a $\langle k \rangle$-immersion. 
\end{corollary}
\begin{proof}
Since both $T\widehat{P} \rightarrow \widehat{P}$ and $\widehat{f}^{*}(TM) \rightarrow \widehat{P}$ are trivial vector bundles and $\dim(M) > \dim(P)$, we may choose a bundle injection $T\widehat{P} \rightarrow \widehat{f}^{*}(TM)$ covering the identity on $\widehat{P}$, and hence a fibrewise injective bundle map $\widehat{\psi}: T\widehat{P} \longrightarrow TM$ that covers the map $\widehat{f}$. 
Using the quotient construction from (\ref{equation: bundle diagram}), the bundle map $\widehat{\psi}$ induces a unique formal $\langle k\rangle$-immersion $\psi \in \Imm^{f}_{\langle k \rangle}(P, M)$ whose underlying $\langle k \rangle$-map is $f$. 
It then follows from Theorem \ref{theorem: k-h principle} that there exists a $\langle k \rangle$-immersion $\phi \in \Imm_{\langle k \rangle}(P, M)$ such that $\mathcal{D}(\phi)$ is on the same path component as $\psi$. 
It then follows that $\phi$ is homotopic through $\langle k \rangle$-maps to the map that underlies $\psi$, which is $f$. 
This completes the proof of the corollary.
\end{proof}

\subsection{The self-intersections of a $\langle k \rangle$-immersion.}
For what follows let $M$ be a manifold of dimension $m$ and let $P$ be a $\langle k\rangle$-manifold of dimension $p$. 
We will need to analyze the self-intersections of $\langle k\rangle$-immersions $P \rightarrow M$. 
\begin{defn} \label{defn: general position}
For $M$ a manifold and $P$ a $\langle k \rangle$-manifold, 
 a $\langle k\rangle$-immersion $f: P \longrightarrow M$ is said to be in \textit{general position} if the following conditions are met:
\begin{enumerate} \itemsep.2cm
\item[i.] The immersion $f_{\beta}: \beta P \rightarrow M$ is self-transverse. 
\item[ii.] The restriction map $f|_{\Int(P)}: \Int(P) \longrightarrow M$ is a self-transverse immersion and is transverse to the immersed submanifold $f_{\beta}(\beta P) \subset M$.
\end{enumerate}
\end{defn}
\noindent
Let $f: P \longrightarrow M$ be a $\langle k \rangle$-immersion that is in general position. 
Let $\hat{q}: P \longrightarrow \widehat{P}$ denote the quotient projection and let $\widehat{\triangle}_{P} \subset P\times P$ be the subspace defined by setting 
$$\widehat{\triangle}_{P} = (\hat{q}\times\hat{q})^{-1}(\triangle_{\widehat{P}}),$$
where $\triangle_{\widehat{P}} \subset \widehat{P}\times\widehat{P}$ is the diagonal subspace. 
It follows from Definition \ref{defn: general position} that the map 
$$(f\times f)|_{(P\times P)\setminus\widehat{\triangle}_{P}}: (P\times P)\setminus\widehat{\triangle}_{P} \longrightarrow M\times M$$ 
is transverse to the diagonal submanifold $\triangle_{M} \subset M\times M$.
We denote by $\Sigma_{f} \subset (P\times P)\setminus \widehat{\triangle}_{P}$ the submanifold given by 
\begin{equation} \label{eq: double point set}
\Sigma_{f} := \bigg((f\times f)|_{(P\times P)\setminus\widehat{\triangle}_{P}}\bigg)^{-1}(\triangle_{M}).
\end{equation}
By the techniques of Section \ref{subsection: k,l intersections}, $\Sigma_{f}$ has the structure of a $\langle k, k\rangle$-manifold with
$$\xymatrix@C-.10pc@R-2.0pc{
\partial_{1}\Sigma_{f}  = f|_{\partial_{1}P}\pitchfork f, & \partial_{2}\Sigma_{f}  = f\pitchfork f|_{\partial_{1}P}, & \partial_{1,2}\Sigma_{f}  = f|_{\partial_{1}P}\pitchfork f|_{\partial_{1}P}, \\
\beta_{1}\Sigma_{f} = f_{\beta}\pitchfork f, & \beta_{2}\Sigma_{f}  = f\pitchfork f_{\beta}, & \beta_{1,2}\Sigma_{f}  = f_{\beta}\pitchfork f_{\beta}.
}$$
The involution 
$$P\times P\setminus \widehat{\triangle}_{P} \longrightarrow P\times P\setminus \widehat{\triangle}_{P}, \quad (x, y) \mapsto (y, x)$$
restricts to an involution on $\Sigma_{f} \subset P\times P\setminus \widehat{\triangle}_{P}$ which we denote by 
\begin{equation} \label{eq: self intersection involution}
T_{\Sigma_{f}}: \Sigma_{f} \longrightarrow \Sigma_{f}. 
\end{equation}
It is clear that the involution $T_{\Sigma_{f}}$ has no fixed-points. Since 
 $$
 \partial_{1}\Sigma_{f} \subset (\partial_{1}P)\times P \quad \text{and} \quad \partial_{2}\Sigma_{f} \subset P\times(\partial_{1}P),
 $$
 it follows that 
 $$
 T_{\Sigma_{f}}(\partial_{1}\Sigma_{f}) \subset \partial_{2}\Sigma_{f} \quad \text{and} \quad T_{\Sigma_{f}}(\partial_{2}\Sigma_{f}) \subset \partial_{1}\Sigma_{f}.
 $$ 

 If both $M$ and $P$ are oriented, then $\Sigma_{f}$ obtains a unique orientation induced from orientations on $P$ and $M$ in the standard way. 
 Furthermore, $T_{\Sigma_{f}}$ preserves orientation if $m-p$ is even and reverses orientation if $m-p$ is odd. 
 We sum up the observations made above into the following proposition. 
 \begin{proposition} \label{eq: k, k with involution} 
 Let $P$ be an oriented $\langle k \rangle$-manifold of dimension $p$ and let $M$ be an oriented manifold of dimension $m$. 
 Let $f: P \longrightarrow M$ be a $\langle k \rangle$-immersion which is in general position. Then the double-point set $\Sigma_{f}$ has the structure of an oriented $\langle k, k \rangle$-manifold of dimension $2p - m$, equipped with a free involution 
 $T_{\Sigma_{f}}: \Sigma_{f} \longrightarrow \Sigma_{f}$
 such that 
 $$T_{\Sigma_{f}}(\partial_{1}\Sigma_{f}) \subset \partial_{2}\Sigma_{f} \quad  \text{and} \quad  T_{\Sigma_{f}}(\partial_{2}\Sigma_{f}) \subset \partial_{1}\Sigma_{f}.$$ 
The involution $T_{\Sigma_{f}}$ preserves orientation if $m-p$ is even and reverses orientation if $m-p$ is odd. 
\end{proposition}

\subsection{Modifying Self-Intersections} \label{subsection: Modifying Self-Intersections} 
In this section, we develop a technique for eliminating the self-intersections of a $\langle k \rangle$-immersion $P \rightarrow M$ by deforming the $\langle k \rangle$-immersion to a $\langle k \rangle$-embedding via a homotopy through $\langle k \rangle$-maps. 
Let $n \geq 2$.
We will solve this problem in the case that $P$ is a \textcolor{red}{closed}, oriented, $2$-connected $\langle k\rangle$-manifold of dimension $2n+1$ and $M$ is a $2$-connected, oriented, $(4n+1)$-dimensional manifold.

By Proposition \ref{eq: k, k with involution}, if 
$f: P \longrightarrow M$
is such a $\langle k \rangle$-immersion in general position, then the double-point set $\Sigma_{f}$ is a closed $1$-dimensional $\langle k, k\rangle$-manifold, equipped with an orientation preserving involution $T: \Sigma_{f} \longrightarrow \Sigma_{f}$ with no fixed points,
such that 
$$T(\partial_{1}\Sigma_{f}) = \partial_{2}\Sigma_{f} \quad \text{and} \quad T(\partial_{2}\Sigma_{f}) = \partial_{1}\Sigma_{f}.$$ 
We will need the following general result about such compact, $1$-dimensional, $\langle k, k\rangle$-manifolds equipped with an involution as above. 

\begin{lemma} \label{lemma: equivariant k, k}
Let $N$ be a $1$-dimensional, closed, oriented, $\langle k, k \rangle$-manifold. 
Suppose that $N$ is equipped with an orientation preserving involution, $T: N \longrightarrow N$, with no fixed points, such that 
$$T(\partial_{1}N) = \partial_{2}N \quad \text{and} \quad T(\partial_{2}N) = \partial_{1}N.$$ 
Then, 
$$\beta_{1}N = \beta_{2}N =  +\langle j \rangle \sqcup -\langle j \rangle$$ 
for some integer $j$.
\end{lemma}
\begin{proof}
We prove this by contradiction. 
Suppose that $\beta_{1}N = +\langle j \rangle \sqcup -\langle l \rangle$ where $j \neq l$. 
Since $T$ preserves orientation and $T(\partial_{1}N) = \partial_{2}N$ and $T(\partial_{2}N) = \partial_{1}N$, 
it follows that $\beta_{2}N = +\langle j \rangle \sqcup -\langle l \rangle$ as well. 

If we forget the $\langle k, k \rangle$-structure on $N$, then $N$ is just an oriented, $1$-dimensional manifold with boundary equal to 
\begin{equation} \label{equation: impossible zero manifold}
\partial_{1}N\sqcup\partial_{2}N = \left[(+\langle j \rangle \sqcup -\langle l \rangle)\times\langle k\rangle\right] \; \bigcup \; \left[(+\langle j \rangle \sqcup -\langle l \rangle)\times\langle k\rangle\right].
\end{equation}
By reorganizing the above union, we see that the zero-dimensional manifold in (\ref{equation: impossible zero manifold}) is equal to $+\langle 2\cdot k\cdot j\rangle \sqcup -\langle 2\cdot k\cdot l\rangle$.
However since $j \neq l$, there does not exist any oriented, one dimensional manifold whose boundary is equal to $+\langle 2\cdot k\cdot j\rangle \sqcup -\langle 2\cdot k\cdot l\rangle$. 
This yields a contradiction and completes the proof of the lemma.
\end{proof}

\begin{proposition} \label{prop: even boundary intersection} Let $P$ be a closed $\langle k\rangle$-manifold of dimension $2n+1$, let $M$ be a manifold of dimension $4n+1$ and let $f: P \longrightarrow M$ be a $\langle k\rangle$-immersion.
Then there is a regular homotopy (through $\langle k\rangle$-immersions) of $f$ to a $\langle k\rangle$-immersion $f': P \longrightarrow M$, such that 
$$\beta_{1}\Sigma_{f'} \; = \; \beta_{2}\Sigma_{f'} \; = \; f'_{\beta}(\beta P)\cap f'(\Int(P)) \; = \; \emptyset.$$
\end{proposition}
\begin{proof}
First, by choosing a small, regular homotopy, we may assume that $f$ is in general position. 
Since $\beta P$ is a closed $2n$-dimensional manifold and $2n < \frac{4n+1}{2}$, the fact that $f$ is in general position implies that $f_{\beta}: \beta P \longrightarrow M$ is an embedding. 

Furthermore, we may assume that $f_{\beta}(\beta P)$ is disjoint from the image of the double point set of the immersion $f|_{\Int(P)}: \Int(P) \longrightarrow M$. 

Consider the intersection $f_{\beta}(\beta P)\cap f(\Int(P))$. 
We choose a closed, disk neighborhood $U \subset \Int(P)$ that contains $f|_{\Int(P)}^{-1}(f_{\beta}(\beta P))$, such that the restriction $f|_{U}: U \longrightarrow M$ is an embedding (we may choose $U$ so that $f|_{U}$ is an embedding because $f_{\beta}(\beta P)$ is disjoint from the image of the double point set of $f|_{\Int(P)}$).
By Lemma \ref{lemma: equivariant k, k} it follows that there is a diffeomorphism
$$f|_{U}^{-1}(f_{\beta}(\beta P)) \;  \cong \;  \beta_{1}\Sigma_{f} \;   \cong \;  +\langle j \rangle \sqcup -\langle j \rangle$$ 
for some integer $j$, and so, the oriented, algebraic intersection number associated to the intersection $f(U)\cap f_{\beta}(\beta P)$ is equal to zero. 
By the Whitney trick, we may find an isotopy through embeddings $\phi_{t}: U \longrightarrow M$ with 
$$\phi_{0} = f|_{U} \quad \text{and} \quad \phi_{t}|_{\partial U} = f|_{\partial U} \quad \text{for all $t \in [0,1]$}$$ 
such that $\phi_{1}(U)\cap f_{\beta}(\beta P)$. 

We then may extend this isotopy over the rest of $P$ by setting it equal to $f$ for all $t \in [0, 1]$ on the complement of $U \subset P$.
This concludes the proof of the lemma. 
\end{proof}

\begin{corollary} \label{corollary: immersion to embedding} 
Let $P$ be a $2$-connected, closed, oriented $\langle k\rangle$-manifold of dimension $2n+1$. 
Let $M$ be a $2$-connected, oriented, manifold of dimension $4n+1$, and let $f: P \longrightarrow M$ be a $\langle k\rangle$-immersion. 
Then $f$ is homotopic through $\langle k\rangle$-maps to a $\langle k\rangle$-embedding. 
\end{corollary}
\begin{remark}
In the statement of the above corollary, we are not asserting that any $\langle k \rangle$-immersion $f: P \longrightarrow M$ is regularly homotopic to a $\langle k \rangle$-embedding. 
The homotopy through $\langle k \rangle$-maps constructed in the proof of this result may very well not be a homotopy through $\langle k \rangle$-immersions. 
\end{remark}
\begin{proof} Assume that $f$ is in general position.
By the previous proposition we may assume that $f_{\beta}: \beta P \longrightarrow M$ is an embedding and that $\beta_{1}\Sigma_{f} = \emptyset$. 
We may choose a collar embedding 
$$h: \partial_{1}P\times [0,\infty) \longrightarrow P \quad \text{with} \quad h^{-1}(\partial_{1}P) = \partial P_{1}\times\{0\},$$ 
such that for each $i \in \langle k\rangle$, the restriction map
$$f|_{h(\partial^{i}_{1}P\times[0,\infty))}: h(\partial^{i}_{1}P\times[0,\infty)) \longrightarrow M$$ 
is an embedding, where $\partial^{i}_{1}P = \Phi^{-1}(\beta P\times\{i\})$. 
Now let $U \subset M$ be a closed tubular neighborhood of $f_{\beta P}(\beta P) \subset M$, disjoint from the image $f(P\setminus h(\partial_{1}P\times[0,\infty)))$, such that the boundary $\partial U$ is transverse to $f(P)$. 
We define, 
\begin{equation}
\xymatrix{
Z := M\setminus\Int(U), & P' := f^{-1}(Z), &  f' := f|_{P'}.
}
\end{equation}
By construction, $P'$ and $Z$ are a manifolds with boundary, $f'$ maps $\partial P'$ into $\partial Z$, and $f'(\Int(P')) \subset \Int(Z)$. 
The corollary will be proven if we can find a homotopy of $f'$, relative $\partial P'$, to a map 
$$f'': (P', \partial P') \longrightarrow (Z, \partial Z)$$ 
which is an embedding.
Using the $2$-connectivity of both $P'$ and $Z$ (and the dimensional conditions on $P'$ and $Z$), the existence of such a homotopy follows from \cite[Theorem 4.1]{H 61} (or from \cite[Theorem 1.1]{Ir 65}).
\end{proof}

\subsection{Proof of Theorem \ref{theorem: immersion to embedding}} \label{subsection: immersion to embedding}
We are now in a position to prove Theorem \ref{theorem: immersion to embedding} from Section \ref{subsection: k-immersions}. 
It follows as a corollary of the results developed throughout this section. 
Here is the theorem restated again for the convenience of the reader. 
\begin{theorem} \label{theorem: represent by embedding 2}
Let $n \geq 2$ be an integer and let $k > 2$ be an odd integer.
Let $M$ be a $2$-connected, oriented manifold of dimension $4n+1$.
Then any $\langle k \rangle$-map $f: V^{2n+1}_{k} \longrightarrow M$ is homotopic through $\langle k \rangle$-maps to a $\langle k \rangle$-embedding. 
\end{theorem}
\begin{proof}
Since $M$ is $2$-connected, it follows that the map $\widehat{f}: \widehat{V}^{2n+1}_{k} \longrightarrow M$ (which is the map induced by the $\langle k \rangle$-map $f$), extends to a map $M(\Z/k, 2n) \longrightarrow M$, where $M(\Z/k, 2n)$ is a $\Z/k$-Moore-space (see Lemma \ref{lemma: Z/k bijection}). 
It then follows that the vector bundle $\widehat{f}^{*}(TM) \longrightarrow \widehat{V}^{2n+1}_{k}$ is classified by a map $\widehat{V}^{2n+1}_{k} \longrightarrow BSO$ that factors through a map $
M(\Z/k, 2n) \longrightarrow BSO$.
When $k$ is odd, the $\Z/k$-homotopy group $\pi_{2n}(BSO; \Z/k)$ is trivial. 
It follows that the bundle $\widehat{f}^{*}(TM) \longrightarrow \widehat{P}$ is trivial. 
Now, it is easy to verify that the $\langle k \rangle$-manifold $V^{2n+1}_{k}$ is parallelizable as a $\langle k \rangle$-manifold (see Definition \ref{defn: k-parallelizability}). 
It then follows from Corollary \ref{corollary: trivial bundle representation} that the map $f$ is homotopic through $k$-maps to a $\langle k \rangle$-immersion, which we denote by $f': V^{2n+1}_{k} \longrightarrow M$. 
The proof of the theorem then follows by applying Corollary \ref{corollary: immersion to embedding} to the $\langle k \rangle$-immersion $f'$. 
\end{proof}

  \end{document}